\documentclass{amsart}

\usepackage{graphicx}
\usepackage{amssymb}
\usepackage{dsfont}
\usepackage[all]{xy}
\usepackage{enumitem}
\usepackage{mathrsfs}
\usepackage{color}

 \usepackage{amsbsy}
\usepackage{amsmath}
\usepackage{amstext}
\usepackage{amsthm}
\usepackage{amscd}

\setlist[description]{leftmargin=2em}

\usepackage
[pdfauthor={Francesc Castella and Carl Wang-Erickson},
pdftitle={Class groups and local indecomposability for non-CM forms},
bookmarks=false]
{hyperref}
\newtheorem{thm}{Theorem}[subsection]
\newtheorem{cor}[thm]{Corollary}
\newtheorem{lem}[thm]{Lemma}
\newtheorem{prop}[thm]{Proposition}
\theoremstyle{definition}
\newtheorem{defn}[thm]{Definition}

\theoremstyle{remark}
\newtheorem{rem}[thm]{Remark}

\numberwithin{equation}{section}
\DeclareMathOperator{\Frac}{Frac}
\DeclareMathOperator{\Hom}{\mathrm{Hom}}
\DeclareMathOperator{\End}{\mathrm{End}}

\DeclareMathOperator{\Gal}{{Gal}}

\DeclareMathOperator{\Sel}{Sel}

\DeclareMathOperator{\Ext}{Ext}

\newcommand{\ord}{\mathrm{ord}}
\newcommand{\spl}{\mathrm{spl}}
\DeclareMathOperator{\Res}{Res}

\DeclareMathOperator{\Ind}{Ind}

\DeclareMathOperator{\rank}{rank}
\DeclareMathOperator{\Frob}{Frob}

\newcommand{\GL}{\mathrm{GL}}

\DeclareMathOperator{\Spec}{Spec}

\newcommand{\id}{\mathrm{id}}

\newcommand{\cC}{{\mathcal C}}

\newcommand{\cE}{{\mathcal E}}
\newcommand{\cF}{{\mathcal F}}
\newcommand{\cG}{{\mathcal G}}

\newcommand{\cI}{{\mathcal I}}

\newcommand{\cK}{{\mathcal K}}
\newcommand{\cL}{{\mathcal L}}
\newcommand{\cM}{{\mathcal M}}

\newcommand{\cO}{{\mathcal O}}

\newcommand{\cU}{{\mathcal U}}

\newcommand{\cX}{{\mathcal X}}
\newcommand{\cY}{{\mathcal Y}}
\newcommand{\cZ}{{\mathcal Z}}

\newcommand{\frc}{{\mathfrak c}}
\newcommand{\fC}{{\mathfrak{C}}}

\newcommand{\frg}{{\mathfrak g}}

\newcommand{\frl}{{\mathfrak l}}

\newcommand{\frp}{{\mathfrak p}}
\newcommand{\frq}{{\mathfrak q}}

\newcommand{\frv}{{\mathfrak v}}

\newcommand{\fZ}{{\mathfrak{Z}}}

\newcommand{\bC}{{\mathbb C}}

\newcommand{\F}{{\mathbb F}}

\newcommand{\bI}{{\mathbb I}}

\newcommand{\bQ}{{\mathbb Q}}

\newcommand{\bT}{{\mathbb T}}

\newcommand{\bZ}{{\mathbb Z}}

\newcommand{\sW}{{\mathfrak W}}

\newcommand{\frL}{{\mathfrak L}}

\newcommand{\Q}{{\mathbb Q}}
\newcommand{\Z}{{\mathbb Z}}

\newcommand{\wt}[1]{\ensuremath{\widetilde{#1}}}

\newcommand{\ra}{\rightarrow}

\newcommand{\lra}{\longrightarrow}
\newcommand{\lrisom}{\buildrel\sim\over\lra}
\newcommand{\risom}{\buildrel\sim\over\ra}
\newcommand{\rinj}{\hookrightarrow}

\newcommand{\rsurj}{\twoheadrightarrow}

\newcommand{\sm}[4]{\ensuremath{\big(\begin{smallmatrix}#1 & #2 \\ #3 & #4\end{smallmatrix}\big)}}

\newcommand{\til}{\widetilde}

\newcommand{\lb}{{[\![}}
\newcommand{\rb}{{]\!]}}

\newcommand{\CM}{\mathrm{CM}}

\newcommand{\dia}{{\langle-\rangle}}
\newcommand{\lr}[1]{{\langle{#1}\rangle}}

\newcommand{\m}{\mathfrak{m}}

\newcommand{\p}{\mathfrak{p}}
\newcommand{\oQ}{\overline{\Q}}

\newcommand{\ttmat}[4]{\left( \begin{array}{cc}
#1 & #2 \\
#3 & #4
\end{array}
\right)}

\makeatletter
\let\c@equation\c@thm
\makeatother
\numberwithin{equation}{subsection}

\title{Class groups and local indecomposability for non-CM forms}

\author{Francesc Castella}
\address{Department of Mathematics, Princeton University\\
    Princeton, NJ 08544, USA}
\email{fcabello@math.princeton.edu}

\author{Carl Wang-Erickson}
\address{Department of Mathematics, Imperial College London \\
	London SW7 2AZ, UK}
\email{c.wang-erickson@imperial.ac.uk}

\address{Department of Mathematics, UCLA, Los Angeles, CA 90095-1555, USA}
\email{hida@math.ucla.edu}

\begin{document}

%\thanks{}%

%\date{\today}%
%\dedicatory{}%
%\commby{}%

\maketitle
%\markright{My default running title}
\vspace{-5.mm}
\begin{center}
	\footnotesize{(with an appendix by HARUZO HIDA)}
\end{center}

% ----------------------------------------------------------------
\begin{abstract}
In the late 1990s, R.~Coleman and R.~Greenberg (independently) asked for a global property characterizing those $p$-ordinary cuspidal eigenforms whose associated Galois representation becomes decomposable upon restriction to a decomposition group at $p$. It is expected that such $p$-ordinary eigenforms are precisely those with complex multiplication. 

In this paper, we study Coleman--Greenberg's question using Galois deformation theory. In particular, for $p$-ordinary eigenforms which are congruent to one with complex multiplication, we prove that the conjectured answer follows from the $p$-indivisibility of a certain class group.  
\end{abstract}

%\maketitle
\tableofcontents

% ----------------------------------------------------------------

\section{Introduction}
\label{sec: intro}

\subsection{Overview}
\label{subsec: OV}

As recorded in \cite[Question 1]{GV2004}, R.~Greenberg has asked when the $2$-dimensional $p$-adic Galois representation $\rho_f$ of $\Gal(\oQ/\Q)$ attached to a $p$-ordinary cuspidal eigenform $f$ of weight $k \geq 2$ has the property of being $p$-locally split, i.e.\ its restriction to a decomposition group $\Gal(\oQ_p/\Q_p)$ at $p$ is isomorphic to the sum of two characters. An equivalent form of this question, which appears to be a very subtle problem in the $p$-adic theory of modular forms, was independently raised by R.~Coleman  \cite[Remark 2, pg.\ 232]{coleman-classical}.\footnote{See \cite[Theorem~4.3.3,  Theorem~4.4.8]{Breuil-Emerton} for the equivalence between the two formulations.}

One easily sees that  $p$-ordinary eigenforms with complex multiplication have this property, and the converse is expected to hold, i.e. (see \cite[Conj.\ (0.1)]{emerton2004}):
\begin{equation}
\label{CG}
\textrm{$\rho_f\vert_{\Gal(\oQ_p/\Q_p)}$ is split}\quad\overset{?}\Longrightarrow\quad\textrm{$f$ has complex multiplication.}\tag{CG}
\end{equation}

Let $\bQ(f)\subset\mathbb{C}$ be the Hecke field of $f$. The Galois representation $\rho_f$ is valued in ${\rm GL}_2(E)$, where $E$ is the completion of $\bQ(f)$ at a prime $v$ above $p$. Serre \cite{serre1968} established (CG) when $k=2$ and $\bQ(f)=\bQ$ using Serre--Tate deformation theory. Still in weight $2$, Serre's argument was extended independently by Emerton \cite{emerton2004} and Ghate \cite{ghate2004} provided $\rho_f$ is ordinary and $p$-split for \emph{all} primes $v$ of $\bQ(f)$ above $p$ (we then say that $\rho_f$ is \emph{totally} $p$-split); the general weight $2$ case was recently established by Zhao \cite{zhao} building on Hida's breakthrough \cite{hidaJAMS2013}. For weights $k>2$, Emerton \cite{emerton2004} showed  that (CG) follows from a $p$-adic analogue of Grothendieck's variational Hodge conjecture, provided $\rho_f$ is totally $p$-split. In a different direction, building on modularity lifting results \cite{Buzzard-Taylor, Buzzard-JAMS} in weight $1$, Ghate--Vatsal \cite{GV2004} showed under mild hypotheses that (CG) holds for all but finitely many $p$-ordinary eigenforms in any single Hida family.

The main result of this paper is Theorem \ref{thm:main}, which gives a sufficient condition for (\ref{CG}) to hold for all forms in a fixed congruence class $\bar f$, allowing for any $p$-adic weight. This condition is that a certain quotient $X$ (later denoted $X(\psi^-)$) of the $p$-part of the class group of the number field cut out by the associated mod $p$ Galois representation $\bar\rho_f$ is zero. Such an $X$ can be associated to any congruence class that contains some member with complex multiplication; we impose only mild additional assumptions. We list some examples of vanishing $X$ in \S\ref{subsec: eg}. 

Greenberg's pseudo-nullity conjecture \cite[Conj.\ (3.5)]{greenberg2001} suggests that a certain Iwasawa-theoretic class group $X^-_\infty$ (later denoted $X^-_\infty(\psi^-)$), which surjects onto $X$, has finite cardinality. To illustrate the influence of $X_\infty^-$, under an extra assumption, we prove in Theorem \ref{thm:main2} that the finiteness of $X_\infty^-$ can be used to produce another proof of the main result of \cite{GV2004} for the class of $\bar\rho_f$ we consider in this paper. 

It is natural to ask whether there exist converse arguments establishing the finiteness of $X^-_\infty$. Thus we give modular characterizations of the vanishing of $X_\infty^-$ (Theorem~\ref{thm: equiv main}) and its finiteness (Theorem~\ref{thm: equiv main2}). 

\subsection{Setup}
\label{subsec: setup}

In order to state question (CG) and the main result of this paper precisely, we introduce the objects of study. Here $G_F$ denotes an absolute Galois group of a field $F$, $O_F$ denotes the appropriate standard integer ring of $F$, and ``CM'' is short for ``complex multiplication.'' Let $p$ be a prime (later, $p \geq 5$). 

\subsubsection{The question}
\label{sssec: the question} 

We fix embeddings of algebraically closed fields $\oQ \rinj \oQ_q$ for all primes $q$, and $\oQ \rinj \bC$. These embeddings give rise to a choice of $q$-adic valuation on any algebraic complex number. They also determine a choice of decomposition group $G_q := \Gal(\oQ_q/\Q_q) \rinj G_\Q$ and complex conjugation $c \in G_\Q$. We write $I_q \subset G_q$ for the inertia subgroup. 

Choose a classical normalized cuspidal Hecke newform $f'$ of weight $k \geq 2$ and level $N' \geq 1$. If $p \nmid N'$, let $f$ be a $p$-stabilization of $f'$ of level $\Gamma_0(p) \cap \Gamma_1(N')$; otherwise, let $f = f'$. Thus $f$ is an eigenvector for the $U_p$-operator. Let 
\[
f=\sum_{n\geq 1}a_n(f)q^n
\] 
be the $q$-expansion of $f$ at the cusp $\infty$, write $\Q(f)/\Q$ for the subfield of $\bC$ generated by the coefficients (also the Hecke eigenvalues) $a_n(f)$, and write $v = v_f$ for the prime of $\Q(f)$ over $p$ that is distinguished by the embeddings above. We call $f$ \emph{$p$-ordinary} when its $U_p$-eigenvalue $a_p(f) \in \bC$, which is known to be an algebraic integer, is a $p$-adic unit. 

There is attached to $f$ an absolutely irreducible $p$-adic Galois representation 
\begin{equation}
\label{eq: assoc Gal rep}
\rho_f : G_\Q \lra \GL_2(\Q(f)_v)
\end{equation}
characterized by the property that
\begin{equation}\label{eq:tr-Frob}
\textrm{trace}\;\rho_f(\Frob_q) = a_q(f)\quad
\textrm{for all primes $q \nmid N'p$},
\end{equation}
where $\Frob_q \in G_\Q$ is a choice of arithmetic Frobenius element at $q$. It is known that $f$ is $p$-ordinary if and only if $\rho_f\vert_{G_p}$ admits a 1-dimensional unramified quotient with $\Frob_p$-eigenvalue $a_p(f)$. 

We call such a representation of $G_\Q$, when equipped with the $\Frob_p$-eigenvalue, \emph{$p$-ordinary}. Similarly, we call a representation $\rho$ of $G_\Q$ \emph{$p$-locally split} when, in addition, $\rho\vert_{G_p}$ is isomorphic to the direct sum of two characters. We ask the question recorded in \S\ref{subsec: OV}: when $k \geq 2$, what property of $f$ determines whether $\rho_f$ is $p$-locally split? 

As discussed above, the proposal, denoted (CG), is that such $f$ have \emph{CM}. While there exist representation-theoretic notions of CM that are arguably more encompassing, we give the simplest equivalent definition: $f$ is called CM when there exists an imaginary quadratic field $K/\Q$ such that the attached quadratic Dirichlet character $\bigl(\frac{K/\Q}{\cdot}\bigr)$ satisfies 
\begin{equation}
\label{eq: CM condition}
\qquad  a_n(f) \bigl(\textstyle{\frac{K/\Q}{n}}\bigr) = a_n(f), \quad\textrm{for almost all $n \geq 1$} \quad \text{ (the CM condition)}. 
\end{equation}

\subsubsection{Fixing the congruence class}
\label{sssec: fix cong class}

It is natural to study (CG) over one congruence class of eigenforms modulo $p$ at a time. Let $\F$ be a finite field of characteristic $p$. Let $\bar f \in \F\lb q\rb$ be the reduction modulo $v_f$ of the $f \in O_{\Q(f)}\lb q\rb$ that we designated above. Let 
\[
\bar\rho := (\rho_f\;{\rm mod}\;{v_f}) : G_\Q \lra \GL_2(\F)
\] 
be the associated representation. The Hecke eigenvalues of $\bar f$ are determined by $\bar\rho$ similarly to \eqref{eq:tr-Frob}. Since $f$ is a $p$-ordinary eigenform, we know that 
\begin{enumerate}
	\item[(1'')] $\bar\rho$ is odd and $\bar\rho\vert_{G_p}$ admits an unramified quotient with $\Frob_p$-eigenvalue $\bar\alpha_p := a_p(\bar f)$. 
\end{enumerate}

Let $N \geq 1$ denote the tame level of $\bar f$, which equals the (prime-to-$p$) Artin conductor of $\bar\rho$. While in general $N$ divides the prime-to-$p$ part $N'_{(p)}$ of $N'$, in this paper we address $f$ that are \emph{minimal}, that is, $N = N'_{(p)}$. 

Because question (CG) addresses $p$-ordinary eigenforms $f$ such that $\rho_f \vert_{G_p}$ splits, \cite[Prop.\ 6]{ghate2005} ensures that in the presence of (2') and (3') below, we may replace (1'') with the more restrictive assumption 
\begin{enumerate}
	\item[(1')] $\bar\rho$ is odd and $\bar\rho\vert_{G_p} \simeq \bar\chi_1 \oplus \bar\chi_2$, where $\bar\chi_2$ is unramified and $\bar\chi_2(\Frob_p) = \bar\alpha_p$.
\end{enumerate}

Our results on (CG) rely on conditions that imply that all Galois representations that give rise to  $\bar\rho$ arise from Hecke eigenforms, i.e.\ ``$R = \bT$.'' Such $R = \bT$-type results are subject to the following assumptions, when $p$ is odd. 
\begin{enumerate}
	\item[(2')] $\bar\chi_1 \neq \bar\chi_2$, which is known as the residually $p$-distinguished condition on $\bar\rho$.
	\item[(3')] $\bar\rho\vert_{G_M}$ is absolutely irreducible, where $M = \Q(\sqrt{(-1)^{(p-1)/2}p})$.
\end{enumerate}

\subsubsection{The residually CM $p$-ordinary setting}
\label{sssec: resid CM}

The following (0)--(4) are the assumptions we work under for the results of this paper. 
\begin{enumerate}
	\item[(0)] $p \geq 5$ and $\bar f$ has CM, in the sense of \eqref{eq: CM condition}. 
\end{enumerate}
It follows that there exists an imaginary quadratic field $K/\Q$ and a character
	\[
	\bar{\psi} : G_K \lra \F^\times \quad \text{ such that } 	\bar{\rho}\cong \Ind_K^\Q \bar{\psi}.
	\]
Let $\psi : G_K \ra W(\F)^\times$ denote the Teichm\"uller lift of $\bar\psi$, let $\frc' \subset O_K$ denote the conductor of $\psi$, and let $\frc \subset O_K$ be the (prime-to-$p$) Artin conductor of $\bar\psi$. Recalling the complex conjugation $c \in G_\Q$ established above, the anti-cyclotomic character associated to $\bar\psi$ is 
\[
\bar{\psi}^- := \bar\psi \cdot (\bar\psi^c)^{-1},
\]
where $\bar\psi^c(\gamma)$ denotes $\bar\psi(c\gamma c)$. 

Having assumed (0), assumptions (1')--(3') are implied (respectively) by
\begin{enumerate}
	\item $p$ splits in $K$, i.e.\ 
	\[
	p \cO_K = \p \p^\ast, 
	\]
	where $\p$ is the prime distinguished by our fixed embedding $\oQ \rinj \oQ_p$, and, also, $\bar\psi$ is unramified at $\p^\ast$ with $\bar\psi(\Frob_{\p^\ast}) = \bar\alpha_p$. One may then check that $N = \mathrm{Norm}_{K/\Q}(\frc)\vert\mathrm{Disc}(K)\vert$. 
	\item $\bar\psi^-\vert_{G_\p}$ is non-trivial and $v_\p(\frc') \leq 1$.
	\item $\bar\psi^-$ has order at least $3$. 
\end{enumerate} 
(For $(3)\Rightarrow\;$(3'), see \cite[Prop.~5.2(2)]{hida2015}.) Finally, we impose the following mild assumption.
\begin{enumerate}[resume]
	\item $\frc + \frc^c = O_K$. 
\end{enumerate}

\subsection{Results, Part I}
\label{subsec: results}

Our first main result addresses the representation $\rho_g : G_\Q \ra \GL_2(\oQ_p)$ attached to a normalized $p$-ordinary $p$-adic eigenform $g \in \overline{\Z}_p\lb q\rb$ that has tame level $N$, arbitrary $p$-adic weight, and a congruence with $\bar f$. We refer to whether $g$ has CM by the same definition \eqref{eq: CM condition}, which makes sense for any $p$-adic weight. 

The theorems in this section are subject to a condition on the following ideal class group. 
Let $\psi^-:G_K\rightarrow W^\times$ denote the Teichm\"uller lift of $\bar\psi^-$ to the Witt vector ring $W=W(\F)$. Let $K(\psi^-)/K$ be the finite abelian extension cut out by $\psi^-$, and denote by $X(\psi^-)$ the $\psi^-$-isotypical component of the $p$-cotorsion of the ideal class group of $K(\psi^-)$. 

\begin{thm}
	\label{thm:main}
	Assume (0)--(4) of \S\ref{subsec: setup}. Let $g$ denote a $p$-ordinary $p$-adic eigenform of tame level $N$ and arbitrary $p$-adic weight that is congruent to $\bar f$. If $X(\psi^-) = 0$ and $\rho_g\vert_{G_p}$ is split, then $g$ has CM. 
\end{thm}

We apply the theorem to (CG). 
\begin{cor}
	\label{cor: CG}
	Assume (0)--(4) of \S\ref{subsec: setup}. If $X(\psi^-) = 0$, then (CG) is true when restricted to those eigenforms of level $N$ with a congruence with $\bar f$. 
\end{cor}

See \S\ref{subsec: eg} for explicit examples where (CG) is verified. 

\begin{rem}
	\label{rem: analytic X vanishing}
	The condition $X(\psi^-) = 0$ can be ensured analytically in some cases: it is implied by the anti-cyclotomic Katz $p$-adic $L$-function $L_p^-(\psi^-)^\ast$ in \S\ref{subsec: katz} being a unit (see e.g.\ \cite[Cor.~5.2.7]{BCGKPST2017}). We also note that the implication (CG) is trivial in the congruence class of $\bar f$ unless a different Katz $p$-adic $L$-function $L_p^-(\psi^-)$, also defined in \S\ref{subsec: katz}, is not a unit. Indeed, when $L_p^-(\psi^-)$ is a unit, any $g$ congruent to $\bar f$ has CM (see Theorem~\ref{thm: CM and non-CM components}). 
\end{rem}

In fact, we prove that the vanishing of $X(\psi^-)$ is \emph{equivalent} to a stronger form of the expected implication (CG). To formulate this, we refer to a modulo $p$ generalized eigenform $\bar g' \in \F\lb q\rb$ whose eigensystem equals that of $\bar f$. We specify these objects in \S\ref{subsec: weights and GE}, also explaining that such a $\bar g'$ induces a Galois representation 
\[
\rho_{\bar g'} : G_\Q \lra \GL_2(A_{\bar g'}),
\]
where $A_{\bar g'}$ is a finite-dimensional augmented $\F$-algebra, such that  
\[
(\rho_{\bar g'}\;{\rm mod}\;{\m_{A_{\bar g'}}}) \simeq \bar\rho \quad \text{ and } \quad \rho_{\bar g'} \not\simeq \bar{\rho} \otimes_\F A_{\bar g'}.
\]
We also explain that the conditions ``$p$-locally split'' and ``CM'' can be sensibly applied to such $\bar g'$. 
\begin{thm}
	\label{thm: equiv main}
	Assume (0)--(4) of \S\ref{subsec: setup}. The following conditions are equivalent. 
	\begin{enumerate}[label=(\roman*)]
		\item $X(\psi^-) \neq 0$.
		\item There exists a modulo $p$ generalized eigenform $\bar g'$ such that 
		\begin{enumerate}
			\item the Hecke eigensystem of $\bar g'$ is equal to that of $\bar f$,
			\item $\bar g'$ does not have CM, and 
			\item $\rho_{\bar g'}\vert_{G_p}$ is split.
		\end{enumerate}
	\end{enumerate}
		When these conditions are true, then $\bar g'$ in (ii) may be chosen so that its Hecke span is 2-dimensional, or, equivalently, $A_{\bar g'} \simeq \F[\epsilon]/(\epsilon^2)$. 
\end{thm}

\subsection{Results, Part II}
\label{subsec: results2} 

We expect that there are many choices of $(K,\bar\psi)$ such that $X(\psi^-)$ does not vanish, as the results of \S\ref{subsec: results} require. The following theorems address the general case. 

We consider the following Iwasawa-theoretic class group tower over $X(\psi^-)$. Let $K^-_\infty/K$ be the anti-cyclotomic $\Z_p$-extension of $K$. Let $K^-_\infty(\psi^-)$ be the composite of  $K^-_\infty$ and $K(\psi^-)$, and let $X^-_\infty(\psi^-)$ be the $\psi^-$-isotypical component of Galois group of the maximal pro-$p$  abelian unramified extension of $K^-_\infty(\psi^-)$. There is a surjection $X^-_\infty(\psi^-) \rsurj X(\psi^-)$, and standard arguments about the action of $\Gal(K^-_\infty(\psi^-)/K)$ on $X^-_\infty(\psi^-)$ imply that 
\[
X^-_\infty(\psi^-) = 0 \quad \text{if and only if} \quad X(\psi^-) = 0. 
\]
In light of Greenberg's pseudo-nullity conjecture \cite[Conj.~(3.5)]{greenberg2001}, it is natural to expect that $X^-_\infty(\psi^-)$ is \emph{finite} in cardinality (note that our assumptions rule out trivial zeros). We prove a proportionally weakened version of Theorem~\ref{thm:main} in this case. 

\begin{thm}
	\label{thm:main2}
	Assume (0)--(4) of \S\ref{subsec: setup} and that the class number of $K$ is prime to $p$. 
	If $X^-_\infty(\psi^-)$ has finite cardinality, then there exist at most finitely many ordinary $p$-adic eigenforms $g$ of tame level $N$ congruent to $\bar f$ such that $\rho_g \vert_{G_p}$ is split and $g$ does not have complex multiplication. 
\end{thm}

\begin{rem}
	\label{rem: analytic X finiteness}
	We note in \S\ref{subsec: ac Iwasawa objects} that $X^-_\infty(\psi^-)$ is infinite if and only if the $p$-adic $L$-function $L_p^-(\psi^-)$ and $L_p^-(\psi^-)^\ast$ mentioned in Remark \ref{rem: analytic X vanishing} have a common factor. It follows from smoothness results of the ordinary eigencurve in cohomological weights (i.e.\ $k\in\Z_{\geq 2}$; see \cite[Cor.~1.4]{hida1986a}, along with a duality argument) that such a common factor cannot correspond to a $p$-adic weight in $\Z\smallsetminus\{1\}$. 
\end{rem}

\begin{rem}
	\label{rem: later rem on GV}
	The conclusion of Theorem \ref{thm:main2} was proven subject only to the conditions (1')--(3') of \S\ref{subsec: setup} by Ghate--Vatsal \cite[Thm.~13]{GV2004}. We describe the relationship between the two methods in Remark~\ref{rem: thm main2 and GV}. 
\end{rem}

In analogy with Theorem \ref{thm: equiv main}, we also can give a modular characterization of the infinitude of $X^-_\infty(\psi^-)$. However, a more pleasant criterion applies to a mild generalization $\cX^-_\infty(\psi^-)$ of $X^-_\infty(\psi^-)$, which surjects onto $X^-_\infty(\psi^-)$ (see \S\ref{subsec: ac Iwasawa objects} for the definition), and is isomorphic to it when $p$ does not divide the class number of $K$. 
%It is not clear to us whether there is any relationship between this modular criterion and (CG). 

Similarly to the mod $p$ case above, to any generalized $p$-adic eigenform $g'$ with eigensystem equal to that of a $p$-adic eigenform with CM $f$ with coefficient field $E/W[1/p]$, there is associated a Galois representation
\[
\rho_{g'} : G_\Q \lra \GL_2(A_{g'}),
\]
where $A_{g'}$ is a finite-dimensional augmented local $E$-algebra, such that 
\[
(\rho_{g'}\;{\rm mod}\;{\m_{A_{g'}}}) \simeq \rho_f \quad \text{ and } \quad \rho_{g'} \not\simeq \rho_f \otimes_E A_{g'}.
\]
As before, the conditions of being $p$-locally split and of being CM can be sensibly applied to $\rho_{g'}$. 

\begin{thm}
	\label{thm: equiv main2}
	Assume (0)--(4) of \S\ref{subsec: setup}. The following conditions are equivalent. 
	\begin{enumerate}
		\item $\cX^-_\infty(\psi^-)$ has infinite cardinality.
		\item There exists a generalized $p$-adic eigenform $g'$ of tame level $N$ such that:
		\begin{enumerate}
			\item the Hecke eigensystem of $g'$ has CM and is congruent to $\bar f$,
			\item $g'$ does not have CM, and
			\item $\rho_{g'}\vert_{G_p}$ is split. 
		\end{enumerate}
	\end{enumerate}
	When these conditions are true, then $g'$ in (2) may be chosen so that its Hecke span is 2-dimensional, or, equivalently, $A_{g'} \simeq E[\epsilon]/(\epsilon^2)$. 
\end{thm}

\subsection{Method of Galois deformation theory}
\label{subsec: method}

By Hida's influential work \cite{hida1986a}, $p$-ordinary $p$-adic eigenforms of tame level $N$ with a congruence with $\bar f$ 
(such as $g$ in the statement of Theorem \ref{thm:main}, for example) 
are in bijective correspondence with ring homomorphisms $\bT \ra \oQ_p$, where $\bT$ is the ``big'' local $p$-adic Hecke algebra arising from the Hecke action on $p$-ordinary modular forms of tame level $N$ whose residual Hecke eigensystem is congruent to $\bar f$. On the other hand, upon assumptions (1'') and (2'), there exists a universal $p$-ordinary deformation ring $R^\ord$ (constructed by Mazur \cite{mazur1989}) parameterizing $p$-ordinary deformations of $\bar\rho$. Hida's further result \cite{hida1986} --- that the Galois representations attached to $p$-ordinary eigenforms interpolate in families --- implies that there exists a natural map $R^\ord \ra \bT$. Under assumptions (1''), (2'), and (3') along with mild local conditions, Diamond \cite{diamond1997}, following Wiles \cite{wiles1995}, has shown that this induces an isomorphism $R^\ord \risom \bT$. 

Replacing (1'') with (1') so that the expected implication (CG) is not trivial on $\bT$, we use a universal Galois deformation ring denoted $R^\spl$ (constructed by Ghate--Vatsal \cite{GV2011}) that parameterizes $p$-locally split representations of $\Gal(\oQ/\Q)$ deforming $\bar\rho$. It follows from the definitions that there is a surjection $R^\ord \rsurj R^\spl$. Thus, homomorphisms $R^\spl \to \oQ_p$ are in bijection with normalized $p$-ordinary eigenforms $g$ such that $\rho_g\vert_{G_p}$ is split. 

Assuming (0), there exist $p$-ordinary CM forms congruent to $\bar f$, resulting in a quotient $\bT \rsurj \bT^\CM$, where $\bT^\CM$ arises from the Hecke action on these CM forms. The fact that the Galois representations arising from $p$-ordinary CM eigenforms are $p$-locally split is reflected in the fact that there exists a surjection $R^\spl \rsurj \bT^\CM$ fitting in a commutative diagram
\begin{equation}
\label{eq:main diagram}
\begin{split}
\xymatrix{
	R^\ord \ar[r]^\sim \ar@{->>}[d] & \bT \ar@{->>}[d] \\
	R^\spl \ar@{->>}[r] & \bT^\CM
}
\end{split}
\end{equation}

In terms of this deformation-theoretic picture, our main result is Theorem~\ref{thm: main R=T}, which states that the surjection $R^\spl \rsurj \bT^\CM$ is an isomorphism \textit{if and only if} $X(\psi^-) = 0$. Theorem~\ref{thm:main} follows directly from this. The argument for Theorem~\ref{thm:main2} is similar, with the addition of commutative algebra arguments set up in \S\ref{sec: resultant} and further results on the structure of $\bT$ reviewed in \S\ref{sec: CM and non-CM}.

Theorem \ref{thm: main R=T} is deduced from Theorem \ref{thm: conormal spaces}, which shows that $\cX^-_\infty(\psi^-)$ constitutes the conormal module of $\Spec( \bT^\CM) \subset \Spec (R^\spl)$. With this structure of $R^\spl$ understood, Theorems \ref{thm: equiv main} and \ref{thm: equiv main2} are applications of $R^\ord \risom \bT$ and the duality between Hecke algebras and cusp forms. 

\subsection{A question}
One upshot of Theorem \ref{thm: main R=T} is that (CG) lies somewhat deeper than the simplest possible ``big $R\!=\!\bT$''-type theorem one could hope for, namely, $R^\spl \cong \bT^\CM$. Is there a Hecke algebra that always corresponds to $R^\spl$? What is the module of ``$p$-split'' modular forms? We intend to take this up in future work. 

\subsection{The appendix to this paper}

These investigations arose from an attempt to study (CG), for congruence classes  $\bar\rho=\Ind_K^\Q\bar\psi$ as introduced in \S\ref{sssec: resid CM} above, after restriction of the Galois representations from $G_\Q$ to $G_K$, using the methods of Wake and the second author \cite{WWE1} to control residually reducible representations. In the process, we realized that some of these arguments amounted to an application of a refined version of Shapiro's lemma to move between deformations of representations of $G_\Q$ and $G_K$. This is the method that is developed in \S\ref{sec: conormal} to prove the key Theorem \ref{thm: conormal spaces}; in particular, the proof of our results makes no use of the theory of ordinary pseudorepresentations of \cite{WWE1}. 

Independently and at about the same time as us, Haruzo Hida established similar results to ours by building on \cite{WWE1} as well as his recent work  \cite{hida-cyclicity}; see \S A.3 for a discussion of the theory of ordinary pseudorepresentations. He has very kindly offered to write his proof of our Theorem~\ref{thm:main} (assuming the class number of $K$ is prime to $p$) as an appendix to this paper. 

\subsection{Examples}
\label{subsec: eg}

Theorem \ref{thm:main} applies to tuples $(p, K,\psi)$, where the $\psi^-$-isotypical part of the ideal class group of $K(\psi^-)$ vanishes. In order for the theorem to apply non-trivially, we are interested in cases where:
\begin{enumerate}[label=(\roman*), leftmargin=2em]
\item $\bT \not\cong \bT^\CM$, i.e.\ there exist non-CM cusp forms congruent to $\bar f$, and
\item $X(\psi^-) = 0$. 
\end{enumerate}
For it is in these cases where Theorem \ref{thm:main} implies that there are no exceptions to (CG) congruent to $\bar f$. 

There are seven examples of $(p,K,\psi)$ satisfying (i) listed in \cite[pg.\ 268]{tilouine1989-CM} (four of which appear in \cite[pg.\ 142]{hida1985}), calculated by Maeda or Mestre. They also each satisfy the running assumptions in our paper, because $p O_K = \p\p^\ast$, $\psi^-$ has order at least 3, and $\psi$ is ramified exactly at $\p$. Among these examples, three of them satisfy $[K(\psi^-) : K] \leq 13$, so that we found it manageable to calculate $K(\psi^-)$ and its class group using \texttt{PARI/GP} or \texttt{Magma} on a single machine. In each of these three cases, $p$ does not divide the class number of $K(\psi^-)$, so that (ii) is satisfied and Theorem \ref{thm:main} applies. These examples are 
\begin{center}
\begin{tabular}{c|c|c}
$p$ & $K$ &  $\psi$\\ \hline 
13 & $\Q(i)$ & $\omega_\p^{8}$  \\
23 & $\Q(\sqrt{-7})$ & $\omega_\p^{10}$ \\
79 & $\Q(\sqrt{-7})$ & $\omega_\p^{12}$
\end{tabular}
\end{center}
The character $\omega_\p$ of $G_K$ is the Teichm\"uller lift of the following character $\bar\omega_\p : G_K \ra \F_p^\times$. Let $w := \#O_K^\times$ and let $\bar\omega_\p : (O_K/\p)^\times \risom \F_p^\times$ be the canonical identification. Then, for every multiple $a$ of $w$, one makes sense of $\omega_\p^a$ by taking the $(a/w)$-th power of the character of $G_K$ associated via class field theory to the character 
\[
\bar\omega_\p^w : (O_K / \p)^\times/O_K^\times \rinj \F_p^\times.
\]

To illustrate the example $(p,K, \psi) = (13,\Q(i), \omega_\p^8)$, we observe that $\psi^-$ has order $3$ and cuts out the $S_3$-extension of $\Q$ with minimal polynomial
\[
x^6 - 2x^5 + 2x^4 - 6x^3 + 25x^2 - 20x + 8.
\]
Its class number is 3. 

\begin{rem}
At the moment, we know of no single example where (ii) fails (which implies that (i) holds), so that the surjection $R^\mathrm{spl} \rsurj \bT^\CM$ is not an isomorphism and also the conditions of Theorem \ref{thm: equiv main} are satisfied.
\end{rem}

\subsection{Acknowledgements}
The authors would like to thank Haruzo Hida for helpful discussions on this topic, and for offering to write his results on it as an appendix to this paper. The authors also thank him for providing funding for C.W.E.'s travel to UCLA, where this project was initiated. 

The authors also thank Matt Emerton, Ralph Greenberg, Mahesh Kakde, Mark Kisin, Bharath Palvannan, Preston Wake, and Liang Xiao for interesting discussions related to this work, and the anonymous referee for a very careful reading of this paper, whose detailed suggestions helped us improve the exposition of our results.

During the preparation of this work, F.C.\ was partially supported by NSF grant DMS-1801385; C.W.E.\ was partially supported by Engineering and Physical Sciences Research Council grant EP/L025485/1; H.H.\ was partially supported by NSF grant DMS-1464106.

\subsection{Notation and conventions}
\label{subsec: notation}

Homomorphisms between profinite topological groups and algebras, and related Galois cohomology modules, are implicitly meant to be continuous. 

When $F$ is a number field with a set of places $\Sigma$, we let $G_{F,\Sigma}$ denote the Galois group of $F_{\Sigma}/F$, where $F_\Sigma$ is the maximal subextension of $\oQ/F$ that is ramified only at the places in $\Sigma$. Other conventions about Galois groups, such as decomposition groups $G_q$, have been stated in \S\ref{subsec: setup}. We use the case that $F = \bQ$ and $\Sigma$ is the set $S$ of places supporting $Np\infty$, thus the Galois group $G_{\Q,S}$. We use $G_{K,S}$ to denote $G_{K, S_K}$, where $S_K$ is the set of places of $K$ over $S$. 

When $F$ is either $K$ or $\Q$ and $T$ is a $G_{F,S}$-module, we write $C^i(O_F[1/pN], T)$ for the standard cochain complex of (inhomogeneous) $G_{F,S}$-cochains valued in $T$, and $H^i(O_F[1/pN], T)$ for its cohomology. We also use the notation $Z^i(O_F[1/pN], T)$ and $B^i(O_F[1/pN], T)$ for the submodules of cocycles and coboundaries, respectively. For a local field $M$ arising as a completion of $F$, with absolute Galois group $G_M$, we  use $C^i(M, T)$, $H^i(M, T)$, $Z^i(M,T)$, and $B^i(M,T)$ to denote the analogues of the global objects above.

\section{Ordinary modular forms and Galois representations}
\label{sec: interpolation}

In this section, we review background from the theory of $p$-adic interpolation of $p$-ordinary modular forms and Galois representations. 

\subsection{Hida theory} 

Throughout this paper, we freely refer to the $p$-adic families of $p$-ordinary eigenforms constructed by Hida (see \cite{hida1986a, hida1986}), along with the associated Hecke algebras and big Galois representations. This section summarizes the parts of this theory that we shall apply, following \cite[\S3]{WWE1} in some of this summary.

We take the data $\bar f$, $\bar\rho$, and $N$ of \S\ref{sssec: fix cong class} to be fixed in advance. 

\subsubsection{Ordinary $\Lambda$-adic cusp forms and Hecke algebras}

For $r \geq 1$, let $S_2(\Gamma_1(Np^r))_{\Z_p}^\ord$ be the ordinary summand of the $\Z_p$-module of cuspidal forms of weight $2$ and level $Np^r$ with coefficients in $\Z_p$. Let
\[
S'_\Lambda = \varinjlim_r S_2(\Gamma_1(Np^r))^\ord_{\Z_p}, 
\]
the limit being over the natural inclusion maps. Let $\bT'$ be the $\Z_p$-algebra generated by the endomorphisms of $S'_\Lambda$ given by the Hecke operators
\begin{equation}
\label{eq: Hecke operators}
T_n, U_\ell, U_p, \lr{d}, \text{ where } (n, Np) = 1, (d, Np) = 1, \ell \mid N \text{ is prime}. 
\end{equation}
The action of these operators on the modulo $p$ $p$-stabilized eigenform $\bar f$ gives rise to a maximal ideal of $\bT'$ with residue field $\F$. Let $\bT''$ denote the completion of $\bT'$ at this maximal ideal. 

We write $\bar\chi$ for $\det \bar\rho$, and $\chi$ for the Teichm\"uller lift of $\bar\chi$. Using the isomorphism $G_\Q^\mathrm{ab} \cong \hat \Z^\times$ of class field theory to think of $\chi$ as a Dirichlet character on $(\Z/pN\Z)^\times$ valued in $\oQ_p^\times$, we define $\Lambda_\Q$ as the $\chi$-isotypical quotient of $\Z_p\lb \Z_p^\times \times (\Z/N\Z)^\times\rb$. Likewise, using the projection $\hat\Z^\times \rsurj \Z_p^\times \times (\Z/N\Z)^\times$, we define the character 
\begin{equation}
\label{eq: dia character}
\dia_\Q : G_\Q \lra \Lambda_\Q^\times,
\end{equation}
which is a deformation of $\bar\chi$ from $\F$ to $\Lambda_\Q$. 

There is a natural map $\Z_p\lb \Z_p^\times \times (\Z/N\Z)^\times\rb \ra \bT''$ sending $d \mapsto \lr{d}$ for $d \in \Z$ with $(d, Np) = 1$. We let 
\[
\bT := \bT'' \otimes_{\Z_p\lb \Z_p^\times \times (\Z/N\Z)^\times\rb} \Lambda_\Q,
\] 
that is, we specialize $\bT$ so that the nebentype on $(\Z/pN\Z)^\times$ is constant and equal to $\chi$ (as opposed to a non-constant  deformation, which is possible when $p \mid \phi(N)$). 
Let $S_\Lambda := S'_\Lambda \otimes_{\bT'} \bT$; this is the module of $p$-ordinary $\Lambda$-adic cusp forms congruent to $\bar f$ and with nebentype precisely $\chi$, and $\mathbb{T}$ the corresponding Hecke algebra. 

By Hida's control theorem \cite[\S3]{hida1986a},  both $\bT$ and $S_\Lambda$ are free $\Lambda_\Q$-modules of finite rank, and by [\emph{loc.\ cit.}, \S2] the pairing
\begin{equation}
\label{eq: Lambda duality}
	\langle\; ,\; \rangle : \bT \times S_\Lambda \lra \Lambda_\Q, \quad (T,f) \mapsto a_1(T\cdot f)
\end{equation}
is a perfect pairing of $\Lambda_\Q$-modules. Consequently, we may view $\cF \in S_\Lambda$ as a $\Lambda$-adic $q$-series in $\Lambda_\Q\lb q\rb$ via
\begin{equation}
\label{eq: Lambda q-series}
\cF \mapsto \sum_{n \geq 1} \langle T'_n, \cF\rangle q^n,
\end{equation}
where $T'_n = T_n$ for $(n, Np) =1$ and, otherwise, $T'_n$ is the usual polynomial (see e.g. \cite[Thm.~3.24]{shimura-IAT}) in the operators of \eqref{eq: Hecke operators} with coefficients in $\Z$. 

\subsubsection{Cohomological weights}
\label{sssec: cohom weights}

We define a \emph{$p$-adic weight} to be a characteristic zero height 1 prime $P$ of $\Lambda_\Q$. Any weight arises from a pair of characters $(\phi_k,\chi')$, 
\[
\phi_k: \Z_p^\times \ra \oQ_p^\times \quad \text{ and } \quad \chi' : (\Z/p^rN\Z)^\times \ra \oQ_p^\times\quad  (\text{some } r \geq 1)
\]
such that 
\[
(\phi_k \cdot \chi') \vert_{(\Z/pN\Z)^\times} = \chi
\]
under the canonical decomposition $\Z_p^\times \cong \F_p^\times \times (1 + p\Z_p)$. In general $k$ is a formal label, but when we start with $k\in\Z$, then $\phi_k$ is the homomorphism $\phi_k(x) := x^{k-1}$. The height $1$ prime $P = P_{k,\chi'} \subset \Lambda_\Q$ associated to $(\phi_k,\chi')$ is defined to be the kernel of the factorization of the ring homomorphism $\Z_p\lb \Z_p^\times \times (\Z/N\Z)^\times\rb \ra \oQ_p$ through $\Lambda_\Q$ induced by $\phi_k \cdot \chi'$. A weight $(\phi_k,\chi')$ is called \emph{cohomological} when $k \in \Z_{\geq 2}$.

By Hida's control theorem, $\bT$ and $S_\Lambda$ interpolate their classical analogues in cohomological weight. That is, for any $p$-adic weight $(\phi_k, \chi')$ with $k \in \Z_{\geq 2}$, we recover the module of cusp forms of this weight $k$ and nebentype $\chi'$ that are congruent to $\bar f$ via
\[
S_\Lambda \otimes_{\Lambda_\Q} \Lambda_\Q/P_{k,\chi'} \cong S_{k,\chi'} := S_k(\Gamma_1(Np^{r}), \chi')^\ord_{\bar f} \subset S_k(\Gamma_1(Np^{r}), \chi')^\ord_{\Z_p}.
\]
Similarly, denoting by $\bT_{k,\chi'}$ the Hecke algebra generated by the Hecke action on $S_{k,\chi'}$, we have a ring isomorphism
\[
\bT \otimes_{\Lambda_\Q} \Lambda_\Q/P_{k,\chi'} \cong \bT_{k,\chi'}
\]
and the $\Lambda_\Q$-adic duality \eqref{eq: Lambda duality} specializes modulo $P_{k,\chi'}$ to the $\bar f$-congruent part of the classical duality between $S_k(\Gamma_1(Np^{r}), \chi')$ and its Hecke algebra. 

We will use these consequences of the foregoing theory.
\begin{lem}
	\label{lem: classical duality}
	There is a bijection between forms in $S_{k,\chi'} \otimes_{\Lambda_\Q/P_{k,\chi'}} \oQ_p$ and $\Lambda_\Q$-linear maps $\bT \ra \oQ_p$ factoring through $\bT \otimes_{\Lambda_\Q} \Lambda_\Q/P_{k,\chi'}$, restricting to a bijection between normalized eigenforms and multiplicative maps. 
\end{lem} 

\begin{proof}
This is standard: see \cite[Cor.~3.2]{hida1986a} and \cite[Thm.~1.2]{hida1986}.
\end{proof}

\begin{lem}
\label{lem: reduced}
$\bT$ is reduced. 
\end{lem}
\begin{proof}
This follows from the argument of \cite[Lem.\ 5.4]{hida2015}. Indeed, the nilradical of $\bT_{k,\chi'}$ is known to act faithfully on oldforms that are old at levels dividing $N$ according to \cite[Cor.\ 3.3]{hida1986a}, and there are no such oldforms in cohomological weight by the assumption that $N$ is the Artin conductor of $\bar\rho$. Therefore $\bT \otimes_{\Lambda_\Q} \Lambda_\Q/P_{k,\chi'}$ is reduced for $k \in \Z_{\geq 2}$, and since cohomological weights are dense in $\Spec \Lambda_\Q$ and $\bT$ is flat over $\Lambda_\Q$, $\bT$ is reduced. 
\end{proof}

\subsubsection{Associated Galois representations}
\label{sssec:AGR}

Hida \cite{hida1986a} proved that the Galois representations $\rho_f$ of \eqref{eq: assoc Gal rep} associated to $p$-ordinary cuspidal eigenforms $f$ interpolate along $\bT$. Under some assumptions, this interpolation takes on the following particularly strong form. For the statement, we write $x_f : \bT \ra E_f \subset \oQ_p$ for the homomorphism associated to a cohomological $p$-ordinary eigenform $f$ as per Lemma~\ref{lem: classical duality}, where $E_f$ is the residue field of $x_f$. 
\begin{prop}
\label{prop: interpolation}
Upon assumptions (1'') and (2') of \S\ref{subsec: setup}, there exists a continuous representation 
\[
\rho_{\bT} : G_\Q \lra \GL_2(\bT),
\]
characterized by the interpolation condition
\[
\rho_f \simeq \rho_\bT \otimes_{\bT, x_f} E_f.
\]
Moreover, $\rho_\bT$ is ramified only at places supporting $Np\infty$ and restricts to $G_p$ with form
\begin{equation}
\label{eq: ord form of rho_T}
\rho_\bT\vert_{G_p} \simeq \ttmat{\dia_\Q\vert_{G_p} \cdot \nu^{-1}}{*}{0}{\nu},
\end{equation}
where $\nu : G_p \ra \bT^\times$ is an unramified character sending an arithmetic Frobenius $\Frob_p$ to $U_p$ and $\dia_\Q$ was defined in \eqref{eq: dia character}. 
\end{prop}
\begin{proof}
Using assumptions (1'') and (2'), Hida's interpolation result \cite[Thm.\ 2.1]{hida1986a} may be upgraded to the claimed form: see e.g.\ \cite[Props.\ 2.2.7 and 2.2.9]{EPW2006}. Then the characterization claim follows from the fact that $\bT$ is flat over $\Lambda_\Q$ and reduced by Lemma \ref{lem: reduced}, as $\bT$ therefore injects into the product of the $E_f$. 
\end{proof}

Also, it follows from the above interpolation and the properties of $\rho_f$ that the determinant of $\rho_\bT$ is given by 
\begin{equation}
\label{eq: det of ord}
\det \rho_\bT \cong \dia_\Q \otimes_{\Lambda_\Q} \bT. 
\end{equation}
In particular, we have equality of $\F$-valued characters of $G_\Q$, $(\det \rho_\bT \mod{\m_\bT}) = (\dia_\Q \mod{\m_{\Lambda_\Q}}) = \bar\chi$.

\subsubsection{Complex multiplication in Hida families}
\label{sssec: CM-hida}

When we impose assumption (0) --- i.e., that $\bar\rho$ is induced from $\bar\psi$ --- there exist classical $p$-ordinary eigenforms with CM that are congruent to $\bar f$ and have tame level $N$. In each cohomological weight $(\phi_k, \chi')$, these form Hecke submodules
\[
S^\CM_{k,\chi'} \subset S_{k,\chi'}. 
\]
The action of $\bT$ on these submodules in cohomological weight results in a quotient $\bT \rsurj \bT^\CM$ which acts faithfully on them (see e.g. \cite[Prop.~5.1]{hida2015}). 

Recalling from \eqref{eq: CM condition} the definition of CM form, we observe that this also applies to any element of $S_\Lambda$, using \eqref{eq: Lambda q-series}. Thus we have a sub-$\Lambda_\Q$-module of $\Lambda$-adic CM forms $S^\CM_\Lambda \subset S_\Lambda$. 

It is known (see e.g.\ \cite[\S{5}]{hida2015}) that $S^\CM_\Lambda$ is Hecke-stable, $\bT^\CM$ and $S^\CM_\Lambda$ are free $\Lambda_\Q$-modules, and the duality \eqref{eq: Lambda duality} restricts to a $\Lambda_\Q$-linear perfect pairing
\[
\bT^\CM \times S^\CM_\Lambda \lra \Lambda_\Q.
\]
This duality along with the control theorem results in a CM-version of the control in cohomological weights $(\phi_k,\chi')$, 
\[
\bT^\CM \otimes_{\Lambda_\Q} \Lambda_\Q/P_{k, \chi'} \risom \bT^\CM_{k,\chi'}, \quad S^\CM_\Lambda \otimes_{\Lambda_\Q} \Lambda_\Q/P_{k, \chi'} \risom S^\CM_{k,\chi'}.
\]

We let $I_\CM:=\ker(\bT \rsurj \bT^\CM)$, and denote by $\rho_\CM$ the restriction of $\rho_\bT$ to the CM locus: $\rho_\CM := \rho_\bT \otimes_\bT \bT^\CM$.

\subsection{Non-classical weights and generalized eigenforms}
\label{subsec: weights and GE}

We will have significant interest in both 
\begin{enumerate}[label=(\roman*)]
	\item $p$-ordinary $p$-adic cusp forms of non-cohomological weight, and 
	\item $p$-ordinary modulo $p$ cusp forms. 
\end{enumerate}
In both cases, we also need to define generalized eigenforms and their associated Galois representations. 

We define $p$-ordinary cusp forms of non-cohomological weight by interpolation. These are all implicitly ``of tame level $N$''.
\begin{defn}\label{defn:eigen}\hfill
\begin{enumerate}
	\item 
	A \emph{$p$-adic $p$-ordinary cusp form of $p$-adic weight $(\phi_k, \chi')$} with a congruence with $\bar f$ is an element of $S_{k, \chi'} := S_\Lambda \otimes_{\Lambda_\Q} \Lambda_\Q/P_{k,\chi'}$. 
	\item
	A \emph{$p$-ordinary $p$-adic Hecke eigensystem} congruent to $\bar f$ is a homomorphism $\bT \ra \oQ_p$, and its weight $(\phi_k,\chi')$ is determined by the unique height 1 prime $P \subset \Lambda_\Q$ through which the composite $\Lambda_\Q \ra \bT \ra \oQ_p$ factors.
\end{enumerate} 
\end{defn}

\begin{rem}
Note that $S_{k,\chi}$ is equal to the module of classical $p$-ordinary forms, denoted identically, when the weight is cohomological. 
\end{rem}

The notions of 
\begin{itemize}
	\item Hecke eigenform,
	\item generalized Hecke eigenform, and
	\item CM by $K$ (the condition of \eqref{eq: CM condition})  
\end{itemize}
apply to such objects in the same manner as to their classical counterparts. In particular, Lemma \ref{lem: classical duality} generalizes straightforwardly to any $p$-adic weight. Thus the eigensystems from Definition~\ref{defn:eigen}(2) are in natural bijection with normalized eigenforms, i.e., ``multiplicity one'' holds in the presence of (1'')--(3'). 

For the sake of clarity, we specify the meaning of ``generalized eigenform''. We use the notation $(-)[1/p]$ as shorthand for $(-) \otimes_{\Z_p} \Q_p$. 
\begin{defn}
	\label{defn: generalized eigenform}
	Let $g'$ be $p$-adic $p$-ordinary cusp form in $S_{k, \chi'}$ that is congruent to $\bar f$. Denote by $\bT[1/p]g'$ the $\bT[1/p$]-span of $g'$ in $S_{k,\chi'}[1/p]$. We call $g'$ a \emph{generalized eigenform} when 
	\begin{enumerate}[label=(\roman*)]
		\item $g'$ is not an eigenform, and 
		\item $\mathrm{soc}(\bT[1/p]g')$ is simple as a $\bT[1/p]$-module, where $\mathrm{soc}(\bT[1/p]g')$ denotes the socle of $\bT[1/p]g'$ as a $\bT[1/p]$-module.
\end{enumerate}

	From such a generalized eigenform, we obtain a $p$-adic $p$-ordinary eigensystem $\bT \ra \oQ_p$ of weight $(\phi_k, \chi')$ via the $\bT$-action on this socle. Denote by $E_{g'}$ the subfield of $\oQ_p$ generated by the image of $\bT$ in $\End_{\Q_p}(\mathrm{soc}(\bT[1/p]g'))$. We also say that \emph{the Hecke eigensystem of $g'$ is $g$} when $g \in S_{k,\chi'}$ is a eigenform and also is an $E_{g'}$-basis for $\mathrm{soc}(\bT[1/p]g')$. 
\end{defn}

We also define the $p$-ordinary modulo $p$ cusp forms required for Theorem \ref{thm: equiv main}.
\begin{defn}
\label{defn: mod p cusp form}
	A \emph{$p$-ordinary modulo $p$ cusp form} (of tame level $N$) congruent to $\bar f$ is an element of $S_\F := S_\Lambda \otimes_{\Lambda_\Q} \F$. 
\end{defn}

Exactly as in the $p$-adic case, the definition of eigenform, generalized eigenform, and CM by $K$ are identically formulated in $S_\F$. Note, however, that the socle of the Hecke span of an element of $S_\F$ is always simple and even 1-dimensional over $\F$, being spanned by $\bar f$. Thus every element of $S_\F$ is a generalized eigenform with Hecke eigensystem precisely $\bar f$. 

Finally, we require Galois representations associated to generalized eigenforms $g' \in S_{k,\chi'}$ and $\bar g' \in S_\F$. 

\begin{defn}
Let $A_{g'}$ be the $\Q_p$-subalgebra of $\End_{\Q_p}(\bT  g' \otimes_{\Z_p} \Q_p)$ generated by the Hecke action on the Hecke span $\bT[1/p]g'$ of $g'$. Thus we have a natural homomorphism $\bT \ra A_{g'}$, and the Galois representation $\rho_{g'}$ associated to $g'$ is given by 
\[
\rho_{g'} := \rho_\bT \otimes_\bT A_g'.
\]
The definition for $\rho_{\bar g'}$ is formulated identically. 
\end{defn}

\begin{lem}
\label{lem: GE endo}
There is a canonical structure of augmented $E_{g'}$-algebra on $A_{g'}$, compatible with the maps they receive from $\bT$. 
There is an identical statement for a generalized eigenform $\bar g' \in S_\F$ in place of $g'$. 
\end{lem}

\begin{proof}
Observe that $E_{g'}$ is the residue field of $\bT \otimes_{\Lambda_\Q} \Lambda_\Q/P_{k,\chi}$ at its prime ideal $\wp_{g'}$, because $\wp_{g'}$ is the kernel of the Hecke action homomorphism 
\[
\bT \otimes_{\Lambda_\Q} \Lambda_\Q/P_{k,\chi} \ra \End_{\Q_p}(\mathrm{soc}(\bT[1/p]g')).
\]
Likewise, $A_{g'}$ admits a surjection from the completion 
$(\bT \otimes_{\Lambda_\Q} \Lambda_\Q/P_{k,\chi})_{\wp_{g'}}^\wedge$ at this residue field. As this completion is naturally endowed with the structure of an augmented local Artinian $E_{g'}$-algebra, this gives $A_{g'}$ the same kind of structure. This augmentation structure $E_{g'} \rinj A_{g'} \rsurj E_{g'}$ is $\bT$-equivariant, by construction.
\end{proof}

\subsection{The ordinary deformation ring}
\label{subsec: ord def ring}

In this section, we recall an  minimal ordinary deformation ring and its comparison to a Hecke algebra. 

Recall that we have fixed $\bar\rho$ as in \S\ref{subsec: results}, with coefficient field $\F$, and that $W = W(\F)$ is the Witt ring of $\F$. Recall also that we denote the semi-simplification of $\bar\rho\vert_{G_p}$ by $\bar\chi_1 \oplus \bar\chi_2$, where $\bar\chi_2$ is assumed to be unramified. We use $\simeq$ to represent isomorphisms of representations up to conjugation, while we use $=$ to denote identical homomorphisms into $\GL_2$. Finally, recall also the notation $G_{\Q,S}$ from \S\ref{subsec: notation}. 

Let $\mathrm{CNL}_W$ denote the category of complete Noetherian local $W$-algebras $A$ with residue field $A/\m_A \cong \F$. 
\begin{defn}[{The minimal ordinary deformation functor, e.g.\ \cite[\S3.1]{DFG2004}}]
\label{defn: Dord}
Let $D^\ord : \mathrm{CNL}_W \ra \mathrm{Sets}$ be the functor associating to $A$ the set of strict equivalence classes of homomorphisms $\rho_A : G_{\Q,S} \ra \GL_2(A)$ such that
\begin{enumerate}[label=(\roman*), leftmargin=2em]
\item $\rho_A \otimes_A \F = \bar\rho$;
\item $\rho_A\vert_{G_p} \simeq \sm{\chi_1}{b}{0}{\chi_2}$, where $\chi_2 : G_p \ra A^\times$ deforms $\bar\chi_2$ and is unramified; 
\item for primes $\ell \mid N$ such that $\#\bar\rho(I_\ell) \neq p$, reduction modulo $\m_A$ induces an isomorphism $\rho_A(I_\ell) \risom \bar\rho(I_\ell)$;
\item for primes $\ell \mid N$ such that $\#\bar\rho(I_\ell) = p$, $\rho_A^{I_\ell}$ is $A$-free of rank 1. 
\end{enumerate}
The ``strict'' equivalence relation is conjugation by an element of $1 + M_{2\times 2}(\m_A) \subset \GL_2(A)$. Note also that $\#\bar\rho(I_\ell) = p$ is equivalent to $\bar\rho(I_\ell)$ having unipotent image. 

Deformations $\rho_A$ of $\bar\rho$ satisfying the conditions defining $D^\ord$ will be known as \emph{$p$-ordinary of tame level $N$}, or just \emph{$p$-ordinary}. 
\end{defn}

The term ``minimal'' refers to conditions (iii) and (iv), while ``ordinary'' refers to condition (ii). These conditions are well-known to be relatively representable on deformation problems, as follows. 

\begin{prop}
The conditions (1'') and (2') of \S\ref{subsec: setup} imply that $D^\ord$ is representable by $R^\ord \in \mathrm{CNL}_W$. In this case, there is a universal ordinary deformation $\rho^\ord : G_{\Q,S} \ra \GL_2(R^\ord)$ of $\bar\rho$. 
\end{prop}
\begin{proof}
Upon these conditions, the representability of a deformation ring for conditions (i) and (ii) of Definition \ref{defn: Dord} is originally due to Mazur \cite[\S1.7, Prop.\ 3]{mazur1989}. A simplification of the argument for (ii) applies to show that (iv) is relatively representable as well. It is standard that condition (iii) is relatively representable. 
\end{proof}

Assuming (1'')--(3'), and under some mild additional conditions, one may produce a map $R^\ord \ra \bT$ corresponding to the representation $\rho_\bT$ and prove that it is an isomorphism. This was first done in many cases by Wiles \cite{wiles1995}, followed by generalizations such as those of Diamond \cite{diamond1996, diamond1997}. Note, however, that some of these generalizations require modifications to $R^\ord$ or $\bT$. We state here only the case we need, where we assume (0)--(4) of \S\ref{subsec: setup}. In this generality, the isomorphism is due to Wiles \cite[Thm.\ 4.8]{wiles1995}.

\begin{thm}[Wiles]
	\label{thm: Rord=Tord}
	Assume (0)--(4) of \S\ref{subsec: setup}. Then the representation $\rho_\bT$ of Proposition \ref{prop: interpolation} induces an isomorphism $R^\ord \risom \bT$ of complete intersection rings. 
\end{thm}

Due to assumption (4), there are no $\ell \mid N$ of type (iv) in the sense of  Definition \ref{defn: Dord}; they are all of type (iii). While it is implicit in Theorem \ref{thm: Rord=Tord} that $\rho_\bT$ satisfies condition (iii), it will be useful later to have seen the following verification. 

\begin{lem}
\label{lem: local reduction}
Assume conditions (0)--(4) of \S\ref{subsec: setup}. Then reduction modulo $\m_\bT$ induces isomorphisms
\[
\rho_\bT(I_\ell) \lrisom \bar\rho(I_\ell) \quad \text{for all } \ell \mid N.
\]
\end{lem}
\begin{proof}
Because $\bT$ is reduced (Lemma \ref{lem: reduced}), by Lemma \ref{lem: classical duality} it will suffice to prove the result after replacing $\rho_\bT$ by $\rho_f$ for an eigenform $f$ with a cohomological weight $(k,\chi')$ of $\Lambda_\Q$. 

Choose some prime $\ell \mid N$, and write $\bar\rho\vert_{G_\ell} \simeq \bar\chi_{\ell,1} \oplus \bar\chi_{\ell,2}$, where (only) $\bar\chi_{\ell,1}$ is ramified. It follows that $H^1(\Q_\ell, (\bar\chi_{\ell,1}\bar\chi_{\ell,2}^{-1})^\pm) = 0$. This in turn implies that $\rho_f\vert_{G_\ell} \simeq \chi_{\ell,1} \oplus \chi_{\ell,2}$, where $\chi_{\ell, i}$ deforms $\bar\chi_{\ell,i}$. Because we have fixed the determinant at $\ell$ (i.e.\ $\det \rho_f\vert_{I_\ell} = \chi'\vert_{I_\ell})$, we observe that the claimed isomorphism fails if and only if $\chi_{\ell,2}$ is ramified if and only if the conductor of $\rho_f\vert_{G_\ell}$ exceeds that of $\bar\rho\vert_{G_\ell}$. However, we have assumed that $f$ is of level $N$, which is defined to be the prime-to-$p$ conductor of $\bar\rho$. 
\end{proof}

We have this addendum to Lemma \ref{lem: GE endo}. 
\begin{lem}
\label{lem: GE Gal reps}
If we let $g$ denote the eigensystem of $g'$, we have
\[
\rho_{g'} \not\simeq \rho_g \otimes_{E_{g'}} A_{g'}. 
\]
\end{lem}

\begin{proof}
Since the socle of $\bT[1/p]g'$ is one-dimensional over $E_{g'}$ but $g'$ is not an eigenform, the Hecke action map $\bT \ra A_{g'}$ cannot factor through the $\bT$-algebra map $E_{g'} \ra A_{g'}$ that corresponds to the Hecke action on $g$. Since $R^\ord$, and hence $\bT$ as well (Theorem~\ref{thm: Rord=Tord}), is a quotient of the unrestricted deformation ring of $\bar\rho$, this means that distinct homomorphisms to $A_{g'}$ out of $\bT$ must correspond to non-isomorphic Galois representations. 
\end{proof}

\subsection{The $p$-locally split deformation ring}
\label{subsec: spl def ring}

The following deformation problem was first considered by Ghate--Vatsal	\cite{GV2011}.

\begin{defn}
\label{defn: Dspl}
Let $D^\spl : \mathrm{CNL}_W \ra \mathrm{Sets}$ be the subfunctor of $D^\ord$ associating to $A$ the set of strict equivalence classes of homomorphisms of the form 
\[
\rho_A\vert_{G_p} \simeq \ttmat{\chi_1}{0}{0}{\chi_2}. 
\]
Deformations $\rho_A$ of $\bar\rho$ satisfying the conditions defining $D^\spl$ will be known as \emph{$p$-split}. 
\end{defn}

\begin{prop}[Ghate--Vatsal]
\label{prop: Rsplit valid}
Assume conditions (1') and (2') of \S\ref{subsec: setup}. Then $D^\spl$ is representable by $R^\spl \in \mathrm{CNL}_W$. 
\end{prop}
\begin{proof}
This is \cite[Prop.\ 3.1]{GV2011}. 
\end{proof}

\begin{cor}
	\label{cor: diagram valid}
	Assume conditions (0)--(4) of \S\ref{subsec: setup}. Then the Galois representations $\rho_\bT$ and $\rho_\CM$ induce diagram \eqref{eq:main diagram}. 
\end{cor}
\begin{proof} 
	We already know that $R^\ord \risom \bT$ from Theorem \ref{thm: Rord=Tord}. The canonical surjection $R^\ord \rsurj R^\spl$ arises from Proposition \ref{prop: Rsplit valid}. Because $\rho_\CM$ is induced via $\Ind_K^\Q$ (see Proposition \ref{prop: CM Galois}) and $p$ splits in $K$, $\rho_\CM\vert_{G_p}$ is $p$-split. Thus $\rho_\CM$ induces a surjection $R^\spl \rsurj \bT^\CM$. The commutativity of \eqref{eq:main diagram} is clear.
\end{proof}

\section{Anti-cyclotomic Iwasawa theory}
\label{sec: ac iwasawa}

In this section, we assemble background information about objects of anti-cyclotomic Iwasawa theory and their relation to Galois cohomology. We will apply the assumptions (0)--(4) of \S\ref{subsec: setup} and use the characters $\bar\psi$ and $\bar\psi^-$ defined there. 

\subsection{Anti-cyclotomic extensions and Iwasawa algebras}
\label{subsec: AC extns}

Recall that we assume that $p\cO_K=\p\p^\ast$ splits, with $\overline{\bQ}\subset\overline{\bQ}_p$ inducing $\p$. We have $G_{K,S}$ as in \S\ref{subsec: notation}. Our notation mostly follows \cite[pg.\ 636]{hida2015}. 

Let $\mathfrak{C}$ be the prime-to-$p$ conductor of $\bar\psi^-:G_{K,S}\rightarrow\mathbb{F}^\times$, which is equal to $\frc \cdot \frc^c$ by assumption (4). Then we consider the following abelian quotients of $G_{K,S}$: 
\begin{align*}
	\mathfrak{Z}&=\textrm{the ray class group of $K$ modulo $\mathfrak{C}p^\infty$,}\\
	Z^-&=\textrm{the maximal quotient of $\mathfrak{Z}$ where complex conjugation acts as $-1$},\\ 
	Z_p^-&=\textrm{the maximal $p$-profinite quotient of $Z^-$}.
\end{align*}

Let $\mathcal{K}_{\mathfrak{C}p^\infty}$ be the ray class field of $K$ modulo $\mathfrak{C}p^\infty$. Let $K_{\mathfrak{C}p^\infty, p}^-/K$ denote the maximal pro-$p$ anti-cyclotomic subextension of $\mathcal{K}_{\mathfrak{C}p^\infty}/K$, so that the Artin map supplies canonical isomorphisms 
\[
\mathfrak{Z}\cong{\rm Gal}(\mathcal{K}_{\mathfrak{C}p^\infty}/K),\quad 
Z_p^-\cong{\rm Gal}(K_{\mathfrak{C}p^\infty,p}^-/K)  
\]
We also let $\Gamma_K^-\cong\bZ_p$ be the maximal torsion-free quotient of $Z_p^-$, and let $K^-_\infty/K$ be the corresponding $\Z_p$-extension. 

Let $\F'$ be the subfield of $\F$ generated by the values of $\bar\psi^-$, and 
denote by $\psi^-:G_{K,S} \rightarrow W'^\times$ the Teichm\"uller lift of $\bar{\psi}^-$, where $W' := W(\F')$. Then $\psi^-$ factors through a character on the quotient $Z^{(p)}_-:=Z^-/Z^-_p$ (a direct factor of $Z^-$), hence defining a projection 
\[
	\pi_{\psi^-}:W'\lb Z^-\rb\longrightarrow W'\lb Z_p^-\rb
\]
sending a group-like element $(z_p,z^{(p)})\in Z^-\subset W'\lb Z^-\rb^\times$ to $\psi^-(z^{(p)})z_p\in W'\lb Z_p^-\rb$. In the following, we let
\begin{equation}
	\label{eq: ac group rings}
	\til\Lambda_{W'}^- := W'\lb Z_p^-\rb,\quad 
	\Lambda_{W'}^-:= W'\lb\Gamma_K^-\rb 
\end{equation}
denote the isotypical components of $W'\lb Z^-\rb$ via $\pi_{\psi^-}$, and via $\pi_{\psi^-}$ composed with the natural projection $Z_p^-\twoheadrightarrow\Gamma_K^-$, respectively. Let $\cI := \ker(\til\Lambda_{W'}^- \rsurj \Lambda_{W'}^-)$ be the kernel of the natural projection. 

\smallskip
\noindent
\textbf{Notation:} For the rest of \S\ref{sec: ac iwasawa} we drop the subscript $W'$ in $\til\Lambda_{W'}^-, \Lambda_{W'}^-$, but we resume this outside \S\ref{sec: ac iwasawa}.

A choice of section $s: \Gamma_K^- \rinj Z_p^-$ endows $\til\Lambda^-$ with the structure of an augmented $\Lambda^-$-algebra. Moreover, it is free of finite rank, receiving a natural isomorphism
\begin{equation}
	\label{eq: tilde presentation}
	\til \Lambda^- \cong \Lambda^-\lb{\rm Gal}(H_s/K)\rb,
\end{equation}
where $H_s/K$ is the finite $p$-primary unramified extension of $K$ cut out by the quotient $Z_p^- \rsurj Z_p^-/\Gamma_K^-$. 

Let
\begin{equation}
	\label{eq: AC characters}
	\begin{gathered}
		\dia_- : G_{K,S} \ra (\Lambda^-)^\times, \qquad \til\dia_- : G_{K,S} \ra (\til\Lambda^-)^\times %\\
	\end{gathered}
\end{equation}
be the canonical characters arising from the projection from the group rings \eqref{eq: ac group rings}, and denote by $\Lambda^-_\dia$ (resp. $\til\Lambda^-_\dia$) the free $\Lambda^-$-module (resp. $\til\Lambda^-$-module) of rank $1$ on which $G_{K,S}$ acts via $\dia_-$ (resp. $\til\dia_-$). 
In particular, the residual character in both cases is $\bar\psi^- : G_{K,S} \ra \F'^\times$. 

The following extension fields of $K$ are cut out by the characters 
$\psi^-$, $\dia_-$, and $\til\dia_-$, respectively:
\begin{align*}
	K(\psi^-)&=\overline{\Q}^{\ker(\psi^-)},\\
	K_\infty^-(\psi^-)&=\textrm{the composite $K_\infty^-K(\psi^-)$},\\
	\cK_\infty^-(\psi^-)&=\textrm{the composite $K_{\mathfrak{C}p^\infty, p}^-K(\psi^-)$}.
\end{align*}

\subsection{Anti-cyclotomic Katz $p$-adic $L$-functions}
\label{subsec: katz}

We briefly recall Katz's $p$-adic $L$-functions attached to $K$. In this section we write $\sW$ for the Witt ring $W(\overline{\mathbb{F}}_p)$ of an algebraic closure of $\mathbb{F}_p$.

For any prime-to-$p$ ideal $\mathfrak{C} \subset \cO_K$, Hida--Tilouine \cite{HT-ENS}, following work of Katz \cite{katz-CM} in the case $\fC = 1$, produced an element 
\[
\mu_\p\in \sW\lb\mathfrak{Z}\rb
\]
(denoted $\mu_p(\mathfrak{C}\p^{*\infty})$ in \cite[Thm.~II.4.14]{deShalit}) characterized by an interpolation property of critical values of the complex  $L$-functions attached to certain Hecke characters of $K$ modulo $\mathfrak{C}p^\infty$. Taking $\mathfrak{C}$ to be the prime-to-$p$ conductor of $\bar\psi^-$, we shall be concerned with the projection 
\[
\mathfrak{L}_p^-(\psi^-) \in \sW\lb Z_p^-\rb \cong \til\Lambda^- \hat\otimes_{W'} \sW 
\]
of $\mu_\p$ via the composite of the natural projection $\sW\lb\fZ\rb \rsurj \sW\lb Z^-\rb$ with $\pi_{\psi^-}$. 

By the Weierstrass preparation theorem, we may and do fix a choice of $\til{L}_p^-(\psi^-)\in\til{\Lambda}^-$ such that
\[
(\til{L}_p^-(\psi^-) \otimes 1)=(\mathfrak{L}^-_p(\psi^-))
\] 
as ideals in $\til\Lambda^-\hat\otimes_{W'} \sW$, and write $L_p^-(\psi^-) \in \Lambda^-$ for its further specialization to $\Lambda^-$. Finally, when $\Spec(\bI) \subset \Spec(\til\Lambda^-)$ is some irreducible component, we denote by $\til L^-_p(\psi^-)_\bI$ the specialization of $\til L_p^-(\psi^-)$  to $\bI$. 

The same constructions apply when $\p$ is replaced by $\p^\ast$ (i.e., starting with $\mu_{\p^*}$), yielding  $\til{L}_p^-(\psi^-)^*\in\til\Lambda^-$, etc. Altogether we obtain the following avatars of the Katz $p$-adic $L$-functions that we will consider: 
\begin{equation}
	\label{eq: Lambda ac-katz}
	\begin{gathered}
		\til L_p^-(\psi^-) , \quad \til L_p^-(\psi^-)^\ast  \in \til\Lambda^-;\\
		L_p^-(\psi^-), \quad L_p^-(\psi^-)^\ast \in \Lambda^-;\\
		\til L_p^-(\psi^-)_\bI, \quad \til L_p^-(\psi^-)_\bI^\ast \in \bI.
	\end{gathered}
\end{equation}

Since we impose condition (4), the following result gives us that the $\mu$-invariants of these $p$-adic $L$-functions (when the coefficient ring is a domain) vanish. 

\begin{prop}[{Finis \cite{finis2006}, Hida \cite{hida2010}}]
	\label{prop: mu=0}
	The $\mu$-invariants of $L_p^-(\psi^-)$, $L_p^-(\psi^-)^\ast$, $\til L_p^-(\psi^-)_\bI$, and $\til L_p^-(\psi^-)_\bI^\ast$ are zero. 
\end{prop}

\begin{rem}
	Each $\bI$ is abstractly isomorphic to $W'[\mu_{p^n}]\lb t\rb$ for some $n$, where $\mu_{p^n}$ denotes a $p^n$-th root of unity. 
\end{rem}

\subsection{Anti-cyclotomic Iwasawa class groups}
\label{subsec: ac Iwasawa objects}

Consider the following metabelian field extensions of $K$:
\begin{align*}
	M_\infty^-&=\textrm{the maximal $\p$-ramified pro-$p$ abelian extension of $K_\infty^-(\psi^-)$},\\
	\cM_\infty^-&=\textrm{the maximal $\p$-ramified pro-$p$ abelian extension of $\cK_\infty^-(\psi^-)$}, \\
%	\end{align*}
%	\begin{align*}
	L_\infty^-&=\textrm{the maximal unramified pro-$p$ abelian extension of $K_\infty^-(\psi^-)$},\\
	\cL_\infty^-&=\textrm{the maximal unramified pro-$p$ abelian extension of $\cK_\infty^-(\psi^-)$}.
\end{align*}
We have Iwasawa modules coming from Galois groups of these extensions, along with the following integral units in these fields: 
\begin{align*}
	Y_\infty^-&={\rm Gal}(M_\infty^-/K_\infty^-(\psi^-)),\\
	\cY_\infty^-&={\rm Gal}(\cM_\infty^-/\cK_\infty^-(\psi^-)),\\
	X_\infty^-&={\rm Gal}(L_\infty^-/K_\infty^-(\psi^-)),\\
	\cX_\infty^-&={\rm Gal}(\cL_\infty^-/\cK_\infty^-(\psi^-)),\\	
%	E_\infty^-&=\textrm{the group of global units in $K_\infty^-(\psi^-)$},\\
	\cE_\infty^-&=\textrm{the group of global units in $\cK_\infty^-(\psi^-)$},\\
%	U_\infty^-&=\textrm{the group of local $1$-units in the completion of $K_\infty^-(\psi^-)$ above $\p$},\\
	\cU_\infty^-&=\textrm{the group of local $1$-units in the completion of $\cK_\infty^-(\psi^-)$ above $\p$}.
\end{align*}

We note that $Y_\infty^-, X_\infty^-$ are naturally modules over $\mathbb{Z}_p\lb{\rm Gal}(K_\infty^-(\psi^-)/K)\rb$, while $\cE_\infty^-, \cU_\infty^-, \cY_\infty^-, \cX_\infty^-$ are naturally modules over $\Z_p\lb\!\Gal(\cK_\infty^-(\psi^-)/K)\rb$. In either case, we append $(\psi^-)$, e.g.\ $Y_\infty^-(\psi^-)$, to denote their $\psi^-$-isotypical components. Thus $Y_\infty^-(\psi^-),X_\infty^-(\psi^-)$ are $\Lambda^-$-modules and $\cE_\infty^-(\psi^-),\cU_\infty^-(\psi^-),  \cY_\infty^-(\psi^-), \cX_\infty^-(\psi^-)$ are $\til\Lambda^-$-modules, and of all of these are known to be finitely generated. They are related by isomorphisms 
\[
Y_\infty^-(\psi^-) \cong \cY_\infty^-(\psi^-)/\cI\cY_\infty^-(\psi^-), \qquad 
X_\infty^-(\psi^-) \cong \cX_\infty^-(\psi^-)/\cI\cX_\infty^-(\psi^-). 
\]

Class field theory then yields the ``fundamental'' exact sequence of $\til\Lambda^-$-modules
\begin{equation}
	\label{eq:CFT1}
	0\longrightarrow \cE_\infty^-(\psi^-)\longrightarrow \cU_\infty^-(\psi^-)
	\longrightarrow \cY_\infty^-(\psi^-)\longrightarrow \cX_\infty^-(\psi^-)\longrightarrow 0.
\end{equation}

\begin{prop}[Anti-cyclotomic main conjecture \cite{rubin1991, HT1994, hida-quadratic}]
	\label{prop: AC main conj}
	The characteristic ideal in $\Lambda^-$ of $Y^-_\infty(\psi^-)$ is generated by $L^-_p(\psi^-)$, and the characteristic ideal of $\cY^-_\infty(\psi^-)_\bI$ is generated by $\til L^-_p(\psi^-)_\bI$. In particular, 
	\[
	\cY^-_\infty(\psi^-) = 0 \iff Y^-_\infty(\psi^-) = 0 \iff L_p^-(\psi^-) \in (\Lambda^-)^\times.
	\]
\end{prop} 

We apply the main conjecture toward the control of $\cX^-_\infty(\psi^-)$. 
\begin{prop}
	\label{prop: X control}
	The following equivalences hold.
	\begin{enumerate}[label=(\roman*)]
		\item $X(\psi^-) = 0 \iff X^-_\infty(\psi^-) = 0 \iff \cX^-_\infty(\psi^-) = 0$, and this is implied by
		at least one of $L_p^-(\psi^-)$ and $L_p^-(\psi^-)^\ast$ being a unit in $\Lambda^-$. 
		\item $\cX^-_\infty(\psi^-)$ is infinite if and only if there exists some irreducible component $\Spec(\bI) \subset \Spec(\til\Lambda^-)$ such that $\til L_p^-(\psi^-)_\bI$ and $\til L_p^-(\psi^-)^\ast_\bI$ have a non-trivial common prime factor $P \subset \bI$ of characteristic zero. 
	\end{enumerate}
\end{prop}

\begin{proof}
	The equivalences of (i) (and the leftmost equivalence of Proposition~\ref{prop: AC main conj}) follow from Nakayama's lemma. For example, $X(\psi^-) \cong \cX^-_\infty(\psi^-)/\m_{\til\Lambda^-}\cX^-_\infty(\psi^-)$. The relation of the vanishing of $X^-_\infty(\psi^-)$ to the $L$-functions in the statement of (i) follows from Proposition \ref{prop: AC main conj} and its variant for the module $\mathcal{Y}^-_\infty(\psi^-)^*$ obtained by swapping the roles of $\p$ and $\p^*$. 
	
	To prove (ii), for convenience write $\cY$ (resp. $\cY^*$) for $\cY^-_\infty(\psi^-)$ (resp. $\cY^-_\infty(\psi^-)^*$), $Y$ for $Y^-(\psi^-)$, and $\cX$ for $\cX^-_\infty(\psi^-)$. Because $\cX$ is a quotient of $\cY$, and we know from Proposition \ref{prop: mu=0} that the $\mu$-invariant of $Y$ is zero, Lemma~\ref{lem: p-infinite} below implies that $\cX$ has a non-zero $p$-torsion-free quotient.
	
	Therefore $\cX[1/p]$ is a non-zero $\til\Lambda[1/p]$-module. By examining a choice of presentation \eqref{eq: tilde presentation}, we see that 
	\[
	\til\Lambda[1/p] \lrisom \bigoplus_\bI (\til\Lambda/\bI)[1/p]
	\]
	is a regular ring. Therefore $\cX[1/p]$ is supported at some maximal ideal of $(\til\Lambda/\bI)[1/p]$ for some choice of irreducible component $\Spec(\bI) \subset \Spec(\til\Lambda^-)$. Since we know that $\cX$ is a quotient of both $\cY$ and $\cY^*$ (whose characteristic ideals on each $\bI$ are associated to $\til L_p^-(\psi^-)_\bI$ by Proposition \ref{prop: AC main conj}), this means that $\mathrm{Char}_\bI(\cY_\bI)$ and $\mathrm{Char}_\bI(\cY^\ast_\bI)$ have a common factor. By Proposition \ref{prop: AC main conj} this is a common factor of $\til L_p^-(\psi^-)_\bI$ and $\til L_p^-(\psi^-)^\ast_\bI$ as well. 	
\end{proof}

\begin{lem}
	\label{lem: p-infinite}
	Let $\cZ$ be a finitely generated $\til\Lambda^-$-module. If $\cZ$ is infinite and $p$-power torsion, then the $\mu$-invariant of $Z := \cZ \otimes_{\til\Lambda^-} \Lambda^-$ as a $\Lambda^-$-module is positive. 
\end{lem}

\begin{proof}
	Because $\cZ$ is finitely generated and $p$-power torsion, there exists some $t \in \Z_{\geq 1}$ such that $p^t \cdot \cZ = 0$. Because of the surjections $\cdot p^s : \cZ/p \rsurj p^s\cZ/p^{s+1}\cZ$, the infinitude of $\cZ$ implies that $\cZ/p$ is infinite. Because $\til\Lambda^-/p$ is generated over $\Lambda^-/p$ by adjoining finitely many nilpotents (via a choice of presentation \eqref{eq: tilde presentation}), the same argument implies that $Z/p$ is infinite. As $Z$ is supported on $\Spec(\Lambda^-/p) \subset \Spec(\Lambda^-)$, this means that the $\mu$-invariant of $Z$ as a $\Lambda^-$-module is positive. 
\end{proof}

\subsection{Galois cohomology with support, and duality}

In this section, we compute some Galois cohomology groups often known as ``Iwasawa cohomology,'' relating them to the Iwasawa-theoretic objects defined in \S\ref{subsec: ac Iwasawa objects}. We follow the approach of \cite[\S6]{WWE1} and parts of \cite[\S2]{WWE2}, using the notation for Galois cohomology established in \S\ref{subsec: notation}. 

We will make use of the modules $\Lambda^-_\dia$, $\til\Lambda^-_\dia$ equipped with the canonical characters defined in \eqref{eq: AC characters}, and respectively denote by
\[
\Lambda^-_\#,\qquad \til\Lambda^-_\#
\]
the same underlying modules equipped with the inverse of those characters.

Let $S' \subset S_K$ denote some subset of places of $K$. We will study the cohomology of a $G_{K,S}$-module $T$ with support in $S'$, denoted $H^i_{(S')}(O_K[1/pN], T)$, which is defined to be the cohomology of the cone of the morphism of complexes
\[
C^\bullet_{(S')}(O_K[1/pN], T) := \mathrm{Cone}\left(C^\bullet(O_K[1/pN], T) \lra \bigoplus_{s \in S'} C^\bullet(K_s, T)\right).
\]
This gives rise to the standard long exact sequence in cohomology, whose terms in a single degree are
\begin{equation}
\label{eq:LES cone}
 H^i_{(S')}(O_K[1/pN],T) \lra H^i(O_K[1/pN],T) \lra \bigoplus_{s \in S'} H^i(K_s,T) 
 \end{equation}
We see that we have $H^i_{(\emptyset)} \cong H^i$. 

The following module-theoretic version of global Tate duality will be useful. 
\begin{prop}
\label{prop: duality}
Let $T$ a free module of finite rank over a complete local Noetherian $\Z_p$-algebra $R$ that is Gorenstein. Equip $T$ with an $R$-linear action of $G_{K,S}$. Let $V$ denote a finitely generated $R$-module (with a trivial $G_{K,S}$-action). Then there is a spectral sequence
\[
E^{i,j}_2 = \Ext^i_{R}(H^{3-j}_{(S')}(O_K[1/pN],T^*(1)), V) \Rightarrow H^{i+j}_{(S_K\setminus S')}(O_K[1/pN], T \otimes_R V),
\]
where $T^*$ denotes the $R$-linear dual module with the contragredient $G_{K,S}$-action. 
\end{prop}

\begin{proof}
This follows directly from \cite[Prop.\ 2.2.1]{WWE2} when $R$ is regular and $S' \in \{S_K, \emptyset\}$. We explain how to adapt the proof of \textit{loc.\ cit.}\ to prove this proposition. 

The generalization to an arbitrary subset $S' \subset S$ follows from the fact that classical Poitou--Tate duality (i.e.\ for $T$ a finite abelian group and $T^*$ its Pontryagin dual) holds for an arbitrary $S' \subset S$. For this, see e.g.\ \cite[Thm.\ B.1]{GV2018}. 

The first part of the proof of \cite[Prop.\ 2.2.1]{WWE2} reduces to the case $V = R$. It relies on a particular case of \cite[Prop.\ 5.4.3]{nekovar2006}, which is an expression of this duality in the derived category of $R$-modules. In this setting, $T$ may be a bounded complex and $T^*$ is a bounded complex representing $\mathrm{RHom}_R(T, \omega_R)$, where $\omega_R$ is a dualizing complex for $R$. In our statement, $R$ is assumed to be Gorenstein (thus one may let $\omega_R$ be $R[0]$) and $T$ is $R$-free, so we may use the standard $R$-linear dual module $T^*$.

The second part of the proof of \cite[Prop.\ 2.2.1]{WWE2} uses \cite[Prop.\ 3.1.3]{LS2013}, and there is no difference in its application. 
\end{proof}

\subsection{Kummer theory for anti-cyclotomic Iwasawa cohomology}

We are interested in Galois cohomology with coefficients in $T = \til\Lambda^-_\#(1)$, which, in view of the review of Iwasawa cohomology in \cite[\S6.1]{WWE1}, is the case of Kummer theory. 
 
\begin{prop}[Kummer theory]
\label{prop: Kummer}
The long exact sequence \eqref{eq:LES cone} where $T = \til\Lambda^-_\#(1)$ and $S' = \{\p\}$, namely, 
\begin{align*}
0& = H^1_{(\p)}(O_K[1/Np], \til\Lambda^-_\#(1))  \ra H^1(O_K[1/Np], \til\Lambda^-_\#(1))  \ra H^1(K_{\p}, \til\Lambda^-_\#(1)) \\
 &\ra H^2_{(\p)}(O_K[1/Np], \til\Lambda^-_\#(1)) \ra H^2(O_K[1/Np], \til\Lambda^-_\#(1)) \ra H^2(K_{\p}, \til\Lambda^-_\#(1)) = 0,
\end{align*}
is canonically isomorphic to the fundamental exact sequence \eqref{eq:CFT1}. In particular, we have isomorphisms
\begin{align}
\label{eq: GC Y}
H^2_{(\p)}(\cO_K[1/pN], \til\Lambda^-_\#(1)) \cong \cY_\infty^-(\psi^-), \\ 
\label{eq: GC X}
H^2(\cO_K[1/pN], \til\Lambda^-_\#(1)) \cong \cX_\infty^-(\psi^-).
\end{align}
\end{prop}

The proof technique is similar to that of \cite[\S6]{WWE1}, which applies when $\Q$ is replaced by $K$. 

\begin{lem}
\label{lem:global kummer}
There are canonical isomorphisms
\[
H^1(O_K[1/pN],\til\Lambda^-_\#(1)) \cong \cE^-_\infty(\psi^-) 
\]
and \eqref{eq: GC X}. 
\end{lem}

\begin{proof}
The isomorphism with $\cE^-_\infty(\psi^-)$ appears in \cite[Cor.\ 6.1.3]{WWE1}. The isomorphism \eqref{eq: GC X} follows just as in the proof of \cite[Cor.\ 6.3.1]{WWE1}. Namely, because $\psi^-$ is non-trivial at all primes of $K$ dividing $N$, and is clearly not congruent modulo $p$ to $\Z_p(1)$, taking the $\psi^-$-component of the long exact sequence appearing in the statement of \cite[Cor.\ 6.1.3]{WWE1} results in the desired isomorphism. 
\end{proof}

Similarly, we have the Kummer isomorphism
\[
H^1(K_{\p}, \til\Lambda^-_\#(1)) \cong \cU^-_\infty,
\]
with respect to which the natural maps $H^1(O_K[1/pN],\til\Lambda^-_\#(1)) \ra H^1(K_{\p}, \til\Lambda^-_\#(1))$ and $\cE^-_\infty \rinj \cU^-_\infty$ are compatible. Because $H^0(K_{\p}, \til\Lambda^-_\#(1)) = 0$, it follows from \eqref{eq:LES cone} that $H^1_{(\p)}(O_K[1/pN], \til\Lambda^-_\#(1)) = 0$. By local Tate duality (``derived'' as in Proposition \ref{prop: duality}, which can be applied with $R = \til\Lambda^-$ since this ring is a complete intersection, given its presentation \eqref{eq: tilde presentation}), the vanishing of $H^2(K_{\p}, \til\Lambda^-_\#(1))$ follows from the fact that $H^0(K_{\p}, \til\Lambda^-_\dia/I) = 0$ for all ideals $I \subset \til\Lambda^-$. 

It remains to establish \eqref{eq: GC Y} compatibly with the isomorphisms we have already drawn. 
Using the proof of \cite[Prop.\ 5.3.3(b)]{lim2012} (which is written for $S' = S_K$, but applies to any choice of $S'$, such as $S' = \{\p\}$), we find that
\[
H^2_{(\p)}(O_{K}[1/Np], \til\Lambda^-_\#(1)) \cong \varprojlim_r H^2_{(\p)}(O_{K_r}[1/Np], (\psi^-)^{-1}(1)),
\]
where $K_{\mathfrak{C}p^\infty, p}^- \supset K_r$ is a sequence of $p$-abelian extensions of $K$ cut out by a fundamental system of open neighborhoods of the identity in $Z_p^-$, and the maps of the limit are corestrictions. We use classical Poitou--Tate duality to draw a canonical isomorphism to
\[
\varprojlim_{r,m} \Hom(A_{r,m}, \Q_p/\Z_p), \quad \text{where } A_{r,m} := H^1_{(N\p^*)}(O_{K_r}[1/Np], \psi^- \otimes_{\Z_p} \Z/p^m\Z).
\]
Because $\psi^-$ has order prime to $p$ and is non-constant on $G_\mathfrak{q}$ for all primes $\mathfrak{q}$ of $K_r$ dividing $N\p^*$, we deduce
\[
A_{r,m} \cong H^1_{(N\p^*)}(O_{K_rK(\psi^-)}[1/Np], \Z/p^m\Z)^{\psi^-}
\]
from the analogous isomorphisms for the cohomology theories $H^1(O_{K_r}[1/Np], -)$ or $H^1(K_\mathfrak{q},-)$ replacing $H^1_{(N\p^*)}(O_{K_r}[1/Np], -)$. Because taking the $\psi^-$-part kills the contribution of the cokernel of 
\[
H^0(O_{K_rK(\psi^-)}[1/Np], \Z/p^m\Z) \ra \prod_{\mathfrak{q'}\mid \mathfrak{q} \mid N\p^*} H^0((K_r K(\psi^-))_{\mathfrak{q'}}, \Z/p^m \Z),
\]
to $H^1_{(N\p^*)}(O_{K_rK(\psi^-)}[1/Np], \Z/p^m\Z)$, we know that $A_{r,m}$ is canonically isomorphic to the group of $\psi^-$-equivariant homomorphisms from the absolute Galois group of $K_rK(\psi^-)$ to $\Z/p^m\Z$ that are trivial on $G_{\mathfrak{q'}}$ for $\mathfrak{q'} \mid N\p^*$. 

We observe that $H^1(K_\mathfrak{q}, \bar\psi^-)=0$ for $\mathfrak{q} \mid N$ follows from assumption (4); likewise, $\ker(H^1(K_{\p^*}, \bar\psi^-) \ra H^1(K_{\p^*}^\mathrm{unr}, \bar\psi^-)) = 0$ follows from assumption (2). It follows that triviality of an element of $A_{1,1} = H^1(O_K[1/Np], \bar\psi^-)$ at the decomposition group at $\mathfrak{q} \mid N\p^*$ is equivalent to being trivial on the inertia group at $\mathfrak{q}$. It is straightforward to generalize this conclusion to general $K_r$ and $m \geq 1$ from this base case ($K_1 = K$ and $m=1$), as $K_r/K$ is ramified only at $\p$. By definition of $\cY^-_\infty(\psi^-)$, we deduce a canonical isomorphism
\[
A_{r,m} \cong \Hom_{\Z_p}(\cY^-_\infty(\psi^-) \otimes_{\til\Lambda^-} W'\lb \Gal(K_r/K)\rb, \Z/p^m\Z).
\]
Applying this isomorphism to the limits over $m$ and $r$ above, we deduce \eqref{eq: GC Y}.

To complete the proof of Proposition \ref{prop: Kummer}, it remains to check that the connecting map in \eqref{eq:LES cone} is compatible with the map $\cU^-_\infty(\psi^-) \ra \cY^-_\infty(\psi^-)$ coming from the Artin symbol, and that the map from $H^2_{(\p)}$ to $H^2$ in \eqref{eq:LES cone} is compatible with $\cY^-_\infty(\psi^-) \rsurj \cX^-_\infty(\psi^-)$. This is standard, so we omit it. 

\section{Residually CM Hecke algebras}
\label{sec: CM and non-CM}

Continuing from \S\ref{sssec: CM-hida}, we apply (0)--(4) of \S\ref{subsec: setup} to describe the structure of $\bT$. 

\subsection{CM Hecke algebras and associated Galois representations}
\label{subsec: CM Hecke Gal reps}
The point of this section is to study the structure of the CM Hecke algebra $\bT^\CM$, a quotient of $\bT$ which we defined in \S\ref{sssec: CM-hida}. This will mainly be applied in \S\ref{sec: resultant}. We do this by understanding the relation of $\bT^\CM$ to Galois representations. 

Recall that ${\rm Spec}(\bT^{\rm CM})\subset{\rm Spec}(\bT)$ is the minimal closed subscheme containing all of the irreducible components of $\mathbb{T}$ with CM by $K$, and $\rho_{\CM}=\rho_\bT\otimes_\bT\bT^{\CM}$ denotes the restriction of $\rho_\bT$ to this CM locus. Recall that $\frc \subset O_K$ denotes the prime-to-$p$ Artin conductor of $\psi : G_{K,S} \ra W^\times$. 

We will also use the notation for anti-cyclotomic Iwasawa theory established at the beginning of \S\ref{subsec: AC extns}. We add to it the following definitions. Let $\cK_{\frc\p^\infty}$ denote the ray class field of $K$ modulo $\frc\p^\infty$, with ray class group $Z$. Let $Z_p$ denote the maximal pro-$p$ quotient of $Z$, which is also naturally a direct factor. Also let $\Gamma_K^\p \simeq \Z_p$ be the maximal torsion-free quotient of $Z_p$. 

We see that $\psi$ factors through a character on the quotient $Z^{(p)} := Z/Z_p$, resulting in a projection 
\[
\pi_\psi : W\lb Z \rb \rsurj W\lb Z_p\rb
\]
sending a group-like element $(z_p, z^{(p)}) \in Z$ to $\psi(z^{(p)})z_p \in W\lb Z_p\rb$. In the following, we let
\[
\til \Lambda := W\lb Z_p\rb, \quad \Lambda := W\lb \Gamma^\p_K\rb,
\]
which are equipped with a canonical surjection $\til\Lambda \rsurj \Lambda$. 

Similarly to \eqref{eq: AC characters}, we denote by 
\[
\til \dia : G_{K,S} \ra \til \Lambda^\times, \qquad 
\dia : G_{K,S} \ra \Lambda^\times
\]
the natural characters arising from projection $G_{K,S} \rsurj Z$ and $\pi_\psi$ (resp.\ also via $\til\Lambda \rsurj \Lambda$). Each of $\til\Lambda$ and $\Lambda$ are complete local Noetherian $W$-algebras with residue field $\F$, and these two characters are residually equal to $\bar\psi$. 

Similarly to Definition \ref{defn: Dord}, a deformation $\psi_A$ of $\bar\psi$ to $A \in \mathrm{CNL}_W$ is called \emph{minimal} at a prime $\frq$ of $K$ if reduction modulo $\m_A$ induces an isomorphism $\psi_A(I_\frq) \risom \bar\psi(I_\frq)$. It is standard (see e.g.\ \cite[\S1.4]{mazur1989}) that $\til\Lambda$ with $\til\dia$ is a universal deformation of $\bar\psi$ as follows.
\begin{lem}
\label{lem: CM char def}
There is a canonical isomorphism $R_{\bar\psi} \risom \til\Lambda$, where $R_{\bar\psi}$ represents deformations $\psi_A : G_{K,S} \ra A^\times$ of $\bar\psi$ to $A \in \mathrm{CNL}_W$ that are minimal outside $\p$. 
\end{lem}

\begin{prop}
\label{prop: CM Galois}
Assume (0)--(4) of \S\ref{subsec: setup}. Induction $\Ind_K^\Q$ produces an isomorphism $\til \Lambda \risom \bT^\CM$, arising from the isomorphism
\[
\rho_\CM \simeq \Ind_K^\Q \til\dia. 
\]
In particular, $\bT^\CM$ is a reduced complete intersection. 
\end{prop}
\begin{proof}
As pointed out in the proof of \cite[Prop.\ 5.7(2)]{hida2015}, since we are working in the minimal case (the tame level of our forms is equal to the prime-to-$p$ conductor of $\bar\rho$) this claim follows immediately from Lemma \ref{lem: CM char def} as long as $\bar\rho$ is induced only from $K$ among all quadratic fields. By Proposition~5.2(2) in \emph{loc.\ cit.}, assumption (3) of \S\ref{subsec: setup} implies this. 
\end{proof}

There is a notion of a Zariski-closed \emph{maximal induced locus for $\Ind_K^\Q$} in $\Spec R$, where $R \in \mathrm{CNL}_W$ supports a Galois representation $\rho_R : G_{\Q,S} \ra \GL_2(R)$ deforming $\bar\rho = \Ind_K^\Q \bar\psi$. (See, for example, \cite{DW2018}.) 
\begin{cor}
\label{cor: induced locus}
The kernel $I_\CM$ of the canonical surjection $\bT \rsurj \bT^\CM$ cuts out the maximal induced locus for $\rho_\bT : G_{\Q,S} \ra \GL_2(\bT)$. 
\end{cor}

\begin{proof}
By Theorem \ref{thm: Rord=Tord} and the proof of Lemma \ref{lem: GE Gal reps}, any Zariski-closed locus in $\Spec(\bT)$ is determined by the Galois deformations it supports. Thus the corollary follows from Proposition \ref{prop: CM Galois} and the fact that the CM condition of \eqref{eq: CM condition} is equivalent to the induced condition: $R_{\bar\psi}$ parameterizes all characters $\psi_A$ such that $\Ind_K^\Q \psi_A$ is $p$-ordinary of tame level $N$, and injects into $\bT^\CM$. 
\end{proof}

Proposition \ref{prop: CM Galois} also allows us to the study weight map $\Lambda_\Q \ra \bT^\CM \cong \til\Lambda$. 
\begin{lem}
\label{lem: diaQ and diaK}
The composite map $\beta$ of $\Lambda_\Q \ra \bT \rsurj \bT^\CM \cong \til\Lambda$ satisfies
\begin{equation}
\label{eq: det of CM}
\beta \circ \dia_\Q\vert_{G_{K,S}} = \til\dia \cdot \til\dia^c.
\end{equation}
Also, $\beta$ is an isomorphism if and only if $p \nmid h_K$. 
\end{lem}

\begin{proof}
The first statement follows from \eqref{eq: det of ord}, as Proposition \ref{prop: CM Galois} tells us that $\rho_\CM\vert_{G_{K,S}} \simeq \til\dia \oplus \til\dia^c$. 

A presentation of $\Lambda_\Q$ as a power series ring $W\lb t\rb$ arises from $t \mapsto \lr{\gamma}_\Q-1$, where $\gamma$ is any element of $I_p$ that projects to a generator of the Galois group of the maximal cyclotomic $\Z_p$-extension of $\Q$. From the presentation of $\til\Lambda$ given above, and the equality \eqref{eq: det of CM}, we see that $\Lambda_\Q \ra \til \Lambda$ is an isomorphism if and only if $\gamma-1 \in \Lambda_\Q$ maps to a power series generator of $\til\Lambda$ if and only if $\gamma$ maps to a generator of $Z_p$. This is the case if and only if $I_p \risom I_\p \subset G_{K,S}$ surjects onto $Z_p$, which is equivalent to $p \nmid h_K$. 
\end{proof}

\subsection{Congruence module of the CM locus}
\label{subsec: CM and nCM}

We recall Hida's determination of the characteristic ideal of the congruence module of the CM  locus ${\rm Spec}(\bT^{\CM})\subset{\rm Spec}(\bT)$. 

For this, and for the further study of non-induced deformations of induced representations in \S\ref{sec: conormal}, we identify how  anti-cyclotomic objects over $\til\Lambda^-_{W'}$ set up in \S\ref{sec: ac iwasawa} (like $\til L_p^-(\psi^-)$) are presented over $\til\Lambda$. 

\smallskip
\noindent
\textbf{Notation.} 
In \S\ref{sec: ac iwasawa} only, we denoted $\til\Lambda^-_{W'}, \Lambda^-_{W'}$ without the subscript. Elsewhere, the relationship between the two notations is 
\[
\til\Lambda^- := \til\Lambda_{W'}^- \otimes_{W'} W, \quad \Lambda^- := \Lambda_{W'}^- \otimes_{W'} W,
\]
as in \eqref{eq: Z to Z-minus}. We mildly abuse notation by continuing to use $\til\dia_-$ (resp.\ $\dia_-$) for the base change of this character (as defined in \S\ref{subsec: AC extns}) via $\otimes_{\til\Lambda_{W'}^-} \til\Lambda^-$ (resp.\ $\otimes_{\Lambda_{W'}^-} \Lambda^-$).

These anti-cyclotomic Iwasawa algebras $\til\Lambda^-$ and $\Lambda^-$ are domains of isomorphisms 
\begin{equation}
\label{eq: Z to Z-minus}
\til \delta: \til\Lambda^-  \lrisom \til\Lambda  , \qquad 
\delta : \Lambda^-  \lrisom \Lambda 
\end{equation}
that are characterized by inducing the equality of $\til\Lambda$ (resp.\ $\Lambda$)-valued characters 
\[
\til\delta \circ \til\dia_- \cong  \til\dia \cdot (\til\dia^c)^{-1}, \quad \text{ resp.\ }  
\delta \circ \dia_- \cong  \dia \cdot (\dia^c)^{-1}.
\]
They are induced by the canonical isomorphism $\iota:Z_p\cong Z_p^-$ of \cite[pg.~636]{hida2015}.

Because $\bT$ and $\bT^\CM$ are reduced under our running hypotheses (see Lemma \ref{lem: reduced}, Proposition \ref{prop: CM Galois}), there is a unique algebra decomposition of total fraction fields 
\[
{\rm Frac}(\bT)\simeq{\rm Frac}(\bT^{\rm CM})\oplus X.
\]  
Letting $\bT^{\rm nCM}$ be the projected image of $\bT$ in $X$, we have $I_\CM\hookrightarrow\bT^{\rm nCM}$ and $\bT^\mathrm{nCM}$ is $\Lambda_\Q$-torsion-free. The quotient $\bT^{\rm nCM}/I_\CM$ is the \emph{congruence module}, in the sense of e.g.\ \cite[\S{5.3.3}]{hida-MFG}, between the two components ${\rm Spec}(\bT^{\rm nCM})$ and ${\rm Spec}(\bT^{\rm CM})$ of ${\rm Spec}(\bT)$.

\begin{thm}[Hida]
\label{thm: CM and non-CM components}
Assume conditions (0)--(4) of \S\ref{subsec: setup}. Then 
\[
\bT^{\rm nCM}/I_\CM\simeq \til\Lambda/(\til L_p^-(\psi^-)).
\] 
Moreover, we have the following commutative diagram with exact rows and columns:
\[
\xymatrix{
& I_\mathrm{nCM} \ar[r]^-\sim \ar[d] & (\til L_p^-(\psi^-)) \ar[d] \\
I_\CM \ar[r] \ar[d]^\wr & \bT \ar[r] \ar[d] & \bT^\CM \ar[d]\cong \til \Lambda \\
	I_\CM \ar[r] & \bT^\mathrm{nCM} \ar[r] &  \til\Lambda/(\til L_p^-(\psi^-)).
}
\]
\end{thm}

\begin{proof}
This is shown in \cite[Thm.~7.2]{hida2015}, building on the proof originating from \cite{MT} of the anti-cyclotomic main conjecture (Proposition~\ref{prop: AC main conj}). There we find the additional assumption that $\bar\psi$ is ramified at $\p$ and $p \nmid \phi(N)$. However, the first assumption is used only in order to apply \cite[Thm.\ 7.1]{hida2015} and ensure that $\bT$ is a Gorenstein ring. In our setting, this follows from Theorem \ref{thm: Rord=Tord}. The assumption $p \nmid \phi(N)$ is used to rule out the failure of minimality of CM families, but our assumptions guarantee minimality. 
\end{proof}

\section{Computation of conormal modules using Shapiro's lemma}
\label{sec: conormal}

In this section, we give an explicit interpretation of the conormal module of the closed CM locus inside the $p$-ordinary (resp.\ $p$-locally split) locus. From this, we deduce the main theorem (Theorem \ref{thm:main}) in \S\ref{subsec: proof of main theorems}. 

\subsection{Conormal modules}

Assume (0)--(4) of \S\ref{subsec: setup} in all that follows. We will study the conormal modules of the closed subspaces
\begin{enumerate}
\item $\Spec(\bT^\CM) \subset \Spec (R^\ord) \cong \Spec (\bT)$, and 
\item $\Spec(\bT^\CM) \subset \Spec (R^\spl)$ 
\end{enumerate}

We establish notation 
\[
J := I_\CM = \ker\left(\bT \rsurj \bT^\CM\right), \qquad 
J^s := \ker\left(R^\spl \rsurj \bT^\CM\right),
\]
so that these conormal modules may be denoted 
\[
(1) \ J/J^2 \quad \textrm{ and } \quad (2)\ J^s/(J^s)^2,
\]
respectively. For convenience, we will use the canonical isomorphism $\til \Lambda \cong \bT^\CM$ of Lemma \ref{prop: CM Galois} and write $\til \Lambda$ in the place of $\bT^\CM$ throughout this section, studying $J/J^2$ and $J^s/(J^s)^2$ as $\til{\Lambda}$-modules. 

We also let $\rho$ represent a member of the strict equivalence class (the equivalence relation defining $D^\spl(\til\Lambda)$; see \S\ref{subsec: spl def ring}) of $\rho_\CM$ characterized by demanding that
\[
\rho(c) = \ttmat{0}{1}{1}{0} \quad  \text{ and } \quad \rho\vert_{G_{K,S}} = \ttmat{\dia}{0}{0}{\dia^c}.
\]
Indeed, the left equality fixes a basis up to ordering and scaling, and the second condition fixes the order. 

Let $\til\Lambda[V]$ denote $\til\Lambda \oplus V$ as a square-zero augmented $\til\Lambda$-algebra, so $V^2=0$. For $R^\ast \in \{R^\ord, R^\spl, \til\Lambda\}$, let $\Hom_\CM(R^\ast,\til\Lambda[V])$ denote the fiber of 
\begin{equation}
\label{eq: CM condition over tL}
\Hom_{\Lambda_\Q}(R^\ast, \til\Lambda[V]) \lra \Hom_{\Lambda_\Q}(R^\ast, \til\Lambda)
\end{equation}
over the canonical $\Lambda_\Q$-algebra homomorphism $\phi_\rho : R^\ast \rsurj \til\Lambda$ induced by $\rho$. Here we use the isomorphism $\til \Lambda \cong R_{\bar\psi}$ of Lemma~\ref{lem: CM char def} to speak of the identity automorphism of $\til\Lambda$ induced by $\rho$. Note that $\Hom_\CM(R^\ast, \til\Lambda[V])$ has a natural $\til\Lambda$-module structure coming from the second argument. 

In what follows, we will use, without further comment, the following concrete interpretation of $\Hom_\CM(R^\ast, \til\Lambda[V])$ as a modified deformation functor $D^*_\rho$. 
\begin{lem}
\label{lem: concrete def module}
Let $D^* \in \{D^\ord, D^\spl, D_\psi\}$ be the deformation problem represented by $R^\ast$. There is a canonical bijective correspondence between $\Hom_\CM(R^\ast, \til\Lambda[V])$ and the subset $D^*_\rho(\til\Lambda[V]) \subset D^\ast(\til\Lambda[V])$ consisting of the image of strict equivalence classes within the set of homomorphisms $\rho_V: G_{\Q,S} \ra \GL_2(\til\Lambda[V])$ such that $\rho_V \pmod{V} = \rho$ and $\det \rho_V = \dia_\Q \otimes_{\Lambda_\Q} \til \Lambda$. 
\end{lem}

\begin{rem}
Strict equivalence classes within $D^*_\rho$ amount to conjugacy classes by $1 + M_2(V) \subset \GL_2(\til\Lambda[V])$, which is why it is non-trivial to take the image in $D^*$. 
\end{rem}

\begin{proof}
Let $\rho_V$ represent a strict equivalence class in $D^*$ that is the image of a strict equivalence classes in $D^*_\rho$. Then $\rho_V \pmod{V} \simeq \rho$ and $\det \rho_V = \dia_\Q$. The first condition is equivalent to the map $\phi_{\rho_V} : R^* \ra \til\Lambda[V]$ being induced by $\rho_V$ composing with $\til\Lambda[V] \rsurj \til\Lambda$ to produce $\phi_\rho$. By examining \eqref{eq: det of ord}, we see that the second condition is equivalent to $R^* \ra \til\Lambda$ being a $\Lambda_\Q$-algebra homomorphism. Conversely, any strict equivalence class in $D^*(\til \Lambda[V])$ that satisfies both conditions contains a representative $\rho_V$ of a strict equivalence class in $D^*_\rho(\til\Lambda[V])$, and it is clear that such a class is unique. 
\end{proof}

We also record the relationship between the $\Hom_\CM(R^\ast, \til\Lambda[V])$, which follows directly from the surjections $R^\ord \rsurj R^\spl \rsurj \til\Lambda$. 
\begin{prop}
\label{prop: interpret conormal modules}
The conormal modules are characterized as $\til\Lambda$-modules by 
\begin{align*}
\Hom_{\til\Lambda}(J/J^2, V) \cong \Hom_\CM(R^\ord, \til\Lambda[V])/\Hom_\CM(\til\Lambda, \til\Lambda[V]), \\
\Hom_{\til\Lambda}(J^s/(J^s)^2, V) \cong \Hom_\CM(R^\spl, \til\Lambda[V])/\Hom_\CM(\til\Lambda, \til\Lambda[V]),
\end{align*}
for all finitely generated $\til\Lambda$-modules $V$. 
\end{prop}

\noindent
\textbf{Notation.} We will write $\rho_V$ for a homomorphism
\[
\rho_V : G_{\Q,S} \ra \GL_2(\til\Lambda[V])\quad \text{such that}\quad\rho_V\;({\rm mod}\;{V}) = \rho \text{ and } \det\rho_V = \dia_\Q.
\]
That is, $\rho_V$ is a representative of $D^*_\rho(\til\Lambda[V])$. We also mildly abuse terminology by speaking of a deformation $\rho_V$, when really this is the strict equivalence class of $\rho_V$, and refer to $\rho_V$ as an element of $D^\ast_\rho(\til\Lambda[V])$ for $D^\ast_\rho \in \{D^\ord_\rho, D^\spl_\rho, D_{\bar\psi,\rho}\}$. 

Next we find these $\rho_V$ as elements of an $\Ext^1$-module.

\begin{lem}
\label{lem: injection of con module}
For any finitely generated $\til\Lambda$-module $V$ and $R^\ast \in \{R^\ord, R^\spl, \til\Lambda\}$, there exists a $\til\Lambda$-linear injection of 
\[
D^*_\rho(\til\Lambda[V]) = \Hom_\CM(R^\ast, \til\Lambda[V]) \rinj \Ext^1_{\til\Lambda[G_{\Q,S}]}(\rho, \rho \otimes_{\til\Lambda} V)
\]
determined by sending any $\rho_V \in D^*_\rho(\til\Lambda[V])$ to the extension class determined by the surjection
\[
\rho_V \rsurj \rho_V \otimes_{\til\Lambda[V]}  \til\Lambda = \rho.
\]
\end{lem}
\begin{proof}
The condition $\rho_V \in D^*_\rho(\til\Lambda[V])$ implies that $\rho_V \otimes_{\til\Lambda[V]}  \til\Lambda = \rho$.  One may then readily check that the kernel of $\rho_V \rsurj \rho$ is isomorphic to $\rho \otimes_{\til\Lambda} V$ (where $V$ has a trivial $G_{\Q,S}$-action). Then the map to $\Ext^1$ is injective because strict equivalence in $D^*_\rho$ amounts to conjugation by $1 + M_2(V)$. The fact that this map is $\til\Lambda$-linear is a functorial (in $V$) version of the standard fact (see e.g.\ \cite[pg.\ 399]{mazur1989}) that the tangent space of a deformation ring $R_\rho$ with residue field $k$ is given, as a $k$-vector space, to $\Hom(R, k[\epsilon]/\epsilon^2)$, and admits a canonical isomorphism of $k$-vector spaces to $\Ext^1_{k[G_{\Q,S}]}(\rho,\rho)$. 
\end{proof}

\subsection{Local conditions} 

Next we address the local conditions that define the deformation problems $D^\ord, D^\spl$, thereby determining the images of the injections of Lemma \ref{lem: injection of con module}. We will decompose the condition on the constancy of the determinant of Lemma \ref{lem: concrete def module} into a sum of local inertial conditions. 

First we address conditions at $p$. As we have seen, $\rho\vert_{G_{K,S}} \simeq \psi \oplus \psi^c$. Because $p$ splits in $K$ (and recall that we have designated $\p$ such that $G_\p \risom G_p$), we also have this decomposition of $\rho\vert_{G_p}$. The characters remain distinct after restriction to both $G_{K,S}$ and $G_\p$ because $\bar\psi\vert_{G_\p} = \bar\chi_1 \neq \bar\chi_2 = \bar\psi^c\vert_{G_\p}$, by the assumptions of \S\ref{subsec: setup}. Therefore, restriction to $G_{K,S}$ induces a canonical map 
\begin{equation}
\label{eq: key ext map}
\sigma^p : \Ext^1_{\til\Lambda[G_{\Q,S}]}(\rho, \rho \otimes V)
\ra
\begin{pmatrix}
\Ext^1_{\til\Lambda[G_p]}(\til\dia, \til\dia \otimes V) & 
\Ext^1_{\til\Lambda[G_p]}(\til\dia^c, \til\dia \otimes V) \\
\Ext^1_{\til\Lambda[G_p]}(\til\dia, \til\dia^c \otimes V) &
\Ext^1_{\til\Lambda[G_p]}(\til\dia^c, \til\dia^c\otimes V) 
\end{pmatrix}
\end{equation}
(where the matrix stands for the direct sum of its entries). For $1 \leq i,j \leq 2$, write $\sigma^p_{i,j}$ for the projection to the $(i,j)$-th coordinate of the target of $\sigma^p$. Likewise, write $\tau^p_{i,i}$ for the composition of $\sigma^p_{i,i}$ with 
\[
\Ext^1_{\til\Lambda[G_p]}(\til\dia^{c^{i+1}}, \til\dia^{c^{i+1}}\otimes V)  \ra \Ext^1_{\til\Lambda[I_p]}(\til\dia^{c^{i+1}}, \til\dia^{c^{i+1}}\otimes V).
\]

\begin{lem}
\label{lem: adjoint ordinary}
Let $V$ be a finitely generated $\til\Lambda$ module. 
\begin{enumerate}
\item The ordinary condition and $I_p$-constant determinant condition on the target of $\sigma^p$ are cut out by the kernel of $\sigma^p_{2,1} \oplus \tau^p_{1,1} \oplus \tau^p_{2,2}$. 
\item The split condition and $I_p$-constant determinant condition on the target of $\sigma^p$ are cut out by the kernel of $\sigma^p_{2,1} \oplus \sigma^p_{1,2} \oplus \tau^p_{1,1} \oplus \tau^p_{2,2}$.
\end{enumerate}
\end{lem}
\begin{proof}
This computation of the ordinary condition amounts to the study of ordinary deformation rings appearing in \cite[\S1.7, pg.\ 401]{mazur1989}, and a straightforward generalization to $D^\spl$. We provide more detail, and address the inertial determinant condition. 

A choice of $V$-valued cocycles $e = \sm{a}{b}{c}{d}$ representing a cohomology class in the codomain of $\sigma^p$ may be represented as 
\[
\begin{pmatrix}
a \in Z^1(\Q_p, V) &  b \in Z^1(\Q_p, \til\Lambda_\dia^- \otimes V) \\
c \in Z^1(\Q_p, \til\Lambda^-_\# \otimes V) & d \in Z^1(\Q_p, V)
\end{pmatrix},
\] 
where $\til\Lambda_\dia^- \otimes V$ is short for $\til\Lambda_\dia^- \otimes_{\Lambda^-} V$, and where $V$ is made to be a $\til\Lambda^-$-module via the homomorphism $\til\Lambda^- \ra \til\Lambda^-_{W'}  \otimes_{W'} W \risom \til\Lambda$ found in \eqref{eq: Z to Z-minus}. This data determines a homomorphism
\begin{equation}
\label{eq: rho_e}
\rho_e := \begin{pmatrix}
\til\dia \cdot (1 +  a) &  \til\dia^c b \\
\til\dia c & \til\dia^c\cdot (1 + d)
\end{pmatrix} : G_p \lra \GL_2(\til\Lambda[V]).
\end{equation}
Conjugation by $1 + M_{2\times 2}(V) \subset \GL_2(\til\Lambda[V])$ moves $e$ within its cohomology class. Therefore, a deformation of $\rho\vert_{G_p}$ to $\til\Lambda[V]$ satisfies the conditions of the deformation functor $D^\ord$ if and only if $c$ is a coboundary and $d\vert_{I_p} = 0$. The additional condition that the deformation of the determinant is trivial on $I_p$ is equivalent to $(a+d)\vert_{I_p} = 0$, so we must have $a\vert_{I_p} = 0$ as well. 

Similarly, a deformation of $\rho$ to $\til\Lambda[V]$ restricting to $G_p$ as $\rho_e$ determines an element of $D^\spl(\til\Lambda[V])$ with a trivial deformation of the determinant on $I_p$ if and only if both $b$ and $c$ are coboundaries and $d\vert_{I_p} = a\vert_{I_p} = 0$. 
\end{proof}

Next we address the conditions at primes $\ell \mid N$. This is fairly simple, as we have noted that the off-diagonal cohomology is trivial at $\ell$ in the proof of Lemma \ref{lem: local reduction}. We set up the maps $\sigma^\ell$, $\sigma^\ell_{i,j}$, and $\tau^\ell_{i,j}$ just as for the prime $p$, above. 

\begin{lem}
\label{lem: adjoint ell-constant}
Let $\ell \mid N$ be a prime. The condition of minimality at $\ell$ is cut out by the kernel of $\tau^\ell_{1,1} \oplus \tau^\ell_{2,2}$. 
\end{lem}
\begin{proof}
This condition is part (iii) of Definition \ref{defn: Dord}. As the codomains of $\sigma_{i,j}^\ell$ are zero for $(i,j) \in \{(1,2), (2,1)\}$, only the conditions cut out by $\tau^\ell_{1,1}$, $\tau^\ell_{2,2}$ remain. 
\end{proof}

Thus we have determined the image of the injections of Lemma \ref{lem: injection of con module}.
\begin{cor}
\label{cor: adjoint Selmer}
Let $V$ be a finitely generated $\til\Lambda$-module.
\begin{enumerate}
\item The image of 
\[
\Hom_{\rm CM}(R^\ord,  \til\Lambda[V]) \rinj \Ext^1_{\til\Lambda[G_{\Q,S}]}(\rho, \rho \otimes_{ \til\Lambda} V)
\]
is the kernel of $\sigma^p_{2,1} \oplus \bigoplus_{v \mid Np} \left(\tau^v_{1,1} \oplus \tau^v_{2,2}\right) $.
\item The image of 
\[
\Hom_{\rm CM}(R^\spl,  \til\Lambda[V]) \rinj \Ext^1_{\til\Lambda[G_{\Q,S}]}(\rho, \rho \otimes_{ \til\Lambda} V)
\]
is the kernel of $\sigma^p_{1,2} \oplus \sigma^p_{2,1} \oplus \bigoplus_{v \mid Np} \left(\tau^v_{1,1} \oplus \tau^v_{2,2}\right) $. 
\end{enumerate}
\end{cor}

\subsection{An explicit form of Shapiro's lemma}

Because $\rho \cong \Ind_K^\Q \til\dia$ (see Proposition~\ref{prop: CM Galois}), we can apply Shapiro's lemma to the domain of \eqref{eq: key ext map} to yield that
\[
\Ext^1_{\til\Lambda[G_{\Q,S}]}(\rho, \rho \otimes V) \cong 
\Ext^1_{\til\Lambda[G_{K,S}]}(\til\dia\oplus\til\dia^c, (\til\dia\oplus \til\dia^c) \otimes V). 
\]
We need to relate this isomorphism to \eqref{eq: key ext map}. For this, we develop, in this section, an explicit version of Shapiro's lemma for this particular case. 

In order to state it, we use the notation $(-)^c$ on an extension class as follows, extending the notation for representations of $G_K$ established in \S\ref{sssec: resid CM}: When $\rho_1, \rho_2$ are representations of $G_K$ and $e \in \Ext^1_{G_K}(\rho_2, \rho_1)$ is an extension class represented by the short exact sequence
\[
0 \lra \rho_1 \lra \rho_e \lra \rho_2 \lra 0,
\]
then we write $e^c \in \Ext^1_{G_K}(\rho_2^c, \rho_1^c)$ for the extension class of
\[
0 \lra \rho_1^c \lra \rho_e^c \lra \rho_2^c \lra 0.
\]
Using the canonical isomorphism between these $\Ext$-groups and group cohomology, we also use the notation $(-)^c$ for the map
\[
H^1(O_K[1/Np], \rho_2^* \otimes \rho_1) \lra H^1(O_K[1/Np], {\rho_2^c}^* \otimes \rho_1^c)
\]
induced by the map on $\Ext$-groups. 

Similarly, choosing matrix-valued representatives for the $\rho_i$ and choosing some cocycle $a \in Z^1(O_K[1/Np], \rho_2^* \otimes \rho_1)$, we may use the notion of $(-)^c$ that applies to homomorphisms: 
\[
a^c(\gamma) = a(c\gamma c)\; \text{ for } \gamma \in G_{K,S}. 
\]
We next show that these are compatible. 
\begin{lem}
With notation as above, if we write $\rho_a$ for the extension of $\rho_2$ by $\rho_1$ induced by the cohomology class of $a$, then the cohomology class of $a^c$ corresponds to the extension class of $\rho^c_a$. 
\end{lem}
\begin{proof}
Using the matrix valued representatives, we can write $\rho_a$ as a homomorphism 
\[
\begin{pmatrix}
\rho_1 & \rho_1 \cdot a \\
& \rho_2
\end{pmatrix}
\]
and observe that $\rho_a^c$ is represented by the homomorphism 
\[
\begin{pmatrix}
\rho_1^c & \rho_1^c \cdot a^c \\
& \rho_2^c
\end{pmatrix}. \qedhere
\]
\end{proof}

For notational convenience, in the statement of Proposition \ref{prop: explicit Shapiro} we use $\lr{}$ in place of $\til\dia$. 
\begin{prop}
\label{prop: explicit Shapiro}
The natural map 
\begin{equation}
\label{eq: key ext map K}
\sigma^K : \Ext^1_{\til\Lambda[G_{\Q,S}]}(\rho, \rho \otimes V)
\rightarrow
\begin{pmatrix}
\Ext^1_{\til\Lambda[G_{K,S}]}(\lr{}, \lr{} \otimes V) & 
\Ext^1_{\til\Lambda[G_{K,S}]}(\lr{}^c, \lr{} \otimes V) \\
\Ext^1_{\til\Lambda[G_{K,S}]}(\lr{}, \lr{}^c \otimes V) &
\Ext^1_{\til\Lambda[G_{K,S}]}(\lr{}^c, \lr{}^c\otimes V) 
\end{pmatrix}
\end{equation}
is injective, and its image is given by
\[
\left\{
\begin{pmatrix}
a & b \\
c & d
\end{pmatrix} \in 
\begin{pmatrix}
\Ext^1_{\til\Lambda[G_{K,S}]}(\lr{}, \lr{} \otimes V) & 
\Ext^1_{\til\Lambda[G_{K,S}]}(\lr{}^c, \lr{} \otimes V) \\
\Ext^1_{\til\Lambda[G_{K,S}]}(\lr{}, \lr{}^c \otimes V) &
\Ext^1_{\til\Lambda[G_{K,S}]}(\lr{}^c, \lr{}^c\otimes V) 
\end{pmatrix} \bigg\vert\ a^c = d, b^c = c
\right\}.
\]
\end{prop}
\begin{proof}
Shapiro's lemma tells us that $\sigma^K$ is injective. 

Choose $e = \sm{a}{b}{c}{d}$ in the group of cocycles whose cohomology class lies in the codomain of $\sigma^K$; for example, $b \in Z^1(O_K[1/Np], \til\dia \cdot (\til\dia^c)^{-1} \otimes V)$. This is a function $e : G_{K,S} \ra M_{2 \times 2}(V)$ that determines the homomorphism $\rho_e : G_{K,S} \ra \GL_2(\til\Lambda[V])$ (similar to \eqref{eq: rho_e}) given by
\[
\rho_e := \begin{pmatrix}
\til\dia (1+  a) &  \til\dia^c b \\
\til\dia c & \til\dia^c(1 + d)
\end{pmatrix} : G_{K,S} \lra \GL_2(\til\Lambda[V]).
\]
It extends to a function on $G_{\Q,S} = G_{K,S} \coprod G_{K,S}c$ that we denote by $\tilde \rho_e^C$, given by 
\[
\tilde \rho_e^C : G_{K,S}c  \ni \gamma c \mapsto \rho_e(\gamma) \cdot C\in \GL_2(\til\Lambda[V]) 
\]
(so, in particular, $\tilde\rho_e^C(c) = C$), where $C \in \GL_2(\til\Lambda[V])$ has order 2 and satisfies 
\[
C \equiv \sm{}{1}{1}{} \pmod{V}. 
\]

We observe that the set of lifts of $\rho$ to $\til\Lambda[V]$ is in bijection with the set of pairs $(e,C)$ such that $\tilde \rho_e^C$ is a homomorphism. We break the determination of the homomorphism condition on $\tilde \rho_e^C$ into cases. 

\smallskip
\noindent
\textbf{Case $C = \sm{}{1}{1}{}$.} When $C = \sm{}{1}{1}{}$, we claim that $\tilde \rho_e^C$ is a homomorphism if and only if $a^c = d$ and $b^c = c$, as cocycle functions $G_{K,S} \ra V$. 

We want to verify that $\tilde\rho_e^C(\gamma'' \gamma') = \tilde\rho_e^C(\gamma'')\tilde\rho_e(\gamma')$ for all $\gamma'', \gamma' \in G_{\Q,S}$. A brief computation reduces this verification to the case where $\gamma' \in G_{K,S}$ and also $\gamma'' = \gamma c$ for some unique $\gamma \in G_{K,S}$. In this case, rewrite $\gamma'' \gamma' = (\gamma c) \gamma'$ as $\gamma (c \gamma' c) c$, observing that the desired equality holds if and only if 
\[
\rho_e(c \gamma' c) = \sm{}{1}{1}{} \rho_e(\gamma') \sm{}{1}{1}{}.
\]
This condition holds if and only if $a^c = d$ and $b^c = c$, proving the claim. 

\smallskip
\noindent
\textbf{Case of general $C$.} The set of all possible elements $\GL_2(\til\Lambda[V])$ satisfying the conditions demanded of $C$ are in bijection with 
\[
\cC = \left\{
\ttmat{v_{11}}{v_{12}}{v_{12}}{v_{22}} \in M_{2 \times 2}(V) \mid v_{11} + v_{22} = v_{12} + v_{21} = 0\right\}
\]
via $C \mapsto C - \sm{}{1}{1}{}$. For the moment, fix $(v_{i,j})$ so that it equals $C-\sm{}{1}{1}{}$. The function arising from conjugating $\tilde\rho_e^C$ by $1 + C' := 1+\sm{-v_{12}/2}{-v_{11}/2}{v_{11}/2}{v_{12}/2}$ satisfies 
\[
(1+C')\tilde \rho_e^C(c)(1-C') = \sm{}{1}{1}{}.
\]
Thus we may reduce to Case $C = \sm{}{1}{1}{}$. 

In order to carry out this reduction, we need a bit of additional notation. Write $\partial$ for the boundary map $\partial : C^0(O_K[1/Np], M_{2\times 2}(V)) \ra C^1(O_K[1/Np], M_{2\times 2}(V))$, and write $\partial = \sm{\partial_{11}}{\partial_{12}}{\partial_{21}}{\partial_{22}}$ for its decomposition into matrix coordinates. Then we apply the case $C = \sm{}{1}{1}{}$ and observe that $1+e$ is fixed by conjugation by $(1+C')$ to deduce that $\tilde\rho_e^C$ is a homomorphism if and only if 
\[
a^c = d, \quad b^c = c. 
\]
A complement to $\cC \subset \frp\frg\frl_2 \otimes V$ is $\sm{0}{w}{w}{0}$. Conjugating $\tilde\rho_e^C$ by $1+\sm{0}{1}{1}{0}w$ fixes $\tilde\rho_e^C(c) = C$, fixes $a$ and $d$, and sends 
\[
b \mapsto b-\partial_{12}(w), \quad c \mapsto c - \partial_{21}(w), 
\]
which maintains the equality $b^c = c$. 

Altogether, we have calculated that lifts of $\rho$ to $\til\Lambda[V]$ are in bijection with the $\til\Lambda$-module
\[
(a,b,v_{11},v_{12}) \in Z^1(O_K[1/Np], V) \oplus Z^1(O_K[1/Np], \til \Lambda^-_\dia \otimes V) \oplus V^{\oplus 2}
\]
via $(a,b,v_{11},v_{12}) \mapsto \tilde\rho_e^C$, where $e$ and $C$ are defined as 
\[
e = \ttmat{a}{b}{b^c}{a^c}, \qquad C = \ttmat{}{1}{1}{} + \ttmat{v_{11}}{v_{12}}{-v_{12}}{-v_{11}}.
\]
The action of conjugation by $1 + M_{2 \times 2}(V) \subset \GL_2(\til\Lambda[V])$ on the lifts of $\rho$ to $\til\Lambda[V]$, under this bijection, amounts to translation by the sub-$\til\Lambda$-module
\[
B^1(O_K[1/Np], V) \oplus B^1(O_K[1/Np], \til\Lambda^-_\dia \otimes V) \oplus V^{\oplus 2}.
\]
The quotient is naturally isomorphic to the claimed image of $\sigma^K$. 
\end{proof}

Using the foregoing expression of Shapiro's lemma, we calculate $\Hom_\CM(R^\ast, \til\Lambda[V])$. Write $H_p$ for the $p$-primary summand of the ideal class group of $K$. 
\begin{prop}
\label{prop: tangent spaces}
For any finitely generated $\til\Lambda$-module $V$, there are isomorphisms 
%\begin{enumerate}
%\item 
%$\displaystyle
\[
\Hom_\CM(R^\ord,  \til\Lambda[V]) \lrisom \Hom_{\Z_p}(H_p, V) \oplus H^1_{(N\p^*)}(O_K[1/Np], \til\Lambda^-_\dia \otimes V),
%$
\]
%\item
%$\displaystyle 
\[
\Hom_\CM(R^\spl,  \til\Lambda[V]) \lrisom \Hom_{\Z_p}(H_p, V) \oplus H^1_{(Np)}(O_K[1/Np], \til\Lambda^-_\dia \otimes V),
\]
\[
\Hom_\CM(\til\Lambda,  \til\Lambda[V]) \lrisom \Hom_{\Z_p}(H_p, V).
\]
%$
%\end{enumerate}
\end{prop}

\begin{proof}
We apply throughout the interpretation of $\Hom_\CM(R^\ast,  \til\Lambda[V])$ in Lemma~\ref{lem: concrete def module}. Thus our goal is to calculate the image of the injections of Lemma \ref{lem: injection of con module}, which are determined by Corollary \ref{cor: adjoint Selmer}. So it remains is to interpret the conclusion of Corollary \ref{cor: adjoint Selmer} in terms of Proposition \ref{prop: explicit Shapiro}. 

We use the notation of Galois cohomology instead of $\Ext^1$. For convenience, when $v$ is a rational prime dividing $Np$ and $*=ij$ for $i, j\in\{1,2\}$, we use the natural extensions of $\sigma_*^v$ and $\tau_*^v$ to the codomain of $\sigma^K$: these are $\sigma_*^\frv$, $\tau_*^\frv$, where $\frv$ is the prime over $v$ distinguished by the embeddings of \S\ref{sssec: the question}. 
\begin{align*}
\sigma_{1,2}^\p : H^1(O_K[1/Np], \til\Lambda_\dia^- \otimes_{\til\Lambda} V) \lra H^1(K_\p, \til\Lambda_\dia^- \otimes_{\til\Lambda} V) \\
\sigma_{2,1}^\p : H^1(O_K[1/Np], \til\Lambda_\#^- \otimes_{\til\Lambda} V) \lra H^1(K_\p, \til\Lambda_\#^- \otimes_{\til\Lambda} V) \\
\tau_{i,i}^\frv : H^1(O_K[1/Np], V) \lra H^1(K^\mathrm{unr}_\frv, V), \quad i = 1,2.
\end{align*}
We also use the isomorphism of Shapiro's lemma as given by the top row of $\sigma^K$:
\begin{equation}
\label{eq: S lem}
\Ext^1_{\til\Lambda[G_{\Q,S}]}(\rho, \rho \otimes V)
\lrisom
H^1(O_K[1/Np], V) \oplus H^1(O_K[1/Np], \til\Lambda_\dia^- \otimes_{\til\Lambda} V).
\end{equation}

The map $\bigoplus_{v \mid Np} \left(\tau_{1,1}^\frv \oplus \tau_{2,2}^\frv\right)$ factors through the summand $H^1(O_K[1/Np], V)$ of the codomain of \eqref{eq: S lem},  yielding 
\begin{align*}
H^1(O_K[1/Np], V) &\lra \bigoplus_{v \mid Np} \left(H^1(K_\frv^\mathrm{unr}, V) \oplus H^1(K_\frv^\mathrm{unr},V)\right)
\\
a &\mapsto \big((a\vert_{I_\frv}, a^c\vert_{I_\frv}) \mid \text{ primes } v \mid Np\big).
\end{align*}
Using the equivalence $a^c\vert_{I_\p}=0 \Longleftrightarrow a\vert_{I_{\p^*}} = 0$, we find that these are $V$-valued homomorphisms factoring through $H_p$. This establishes the final claimed isomorphism, as deformations induced from $K$ are split upon restriction to $K$. 

For the first claimed isomorphism, we calculate the ordinary case. Similarly to the previous paragraph, $\sigma_{2,1}^\p$ factors through the summand $H^1(O_K[1/Np], \til\Lambda_\dia^- \otimes_{\til\Lambda} V)$ of the codomain of \eqref{eq: S lem}, yielding
\begin{align}
\label{eq: H ord non-induced tangent cond}
\begin{split}
H^1(O_K[1/Np], \til\Lambda_\dia^- \otimes_{\til\Lambda} V) &\lra 
H^1(K_\p, \til\Lambda_\#^- \otimes_{\til\Lambda} V) \\
b &\mapsto b^c\vert_{G_\p}
\end{split}
\end{align}
Let $\mathfrak{l}$ be a prime of $K$ over $N$. It follows from the cohomology calculation in the proof of Lemma \ref{lem: local reduction} that $H^i(K_{\mathfrak{l}}, \bar\psi^-) = 0$ for all $i \geq 0$. Therefore, the local factors over $N$ of the long exact sequence in cohomology \eqref{eq:LES cone} arising from the cone construction (with $S'$ the set of primes of $K$ dividing $N \p^\ast$ and $T = \til\Lambda_\#^-$) are trivial. Likewise, for the local factors over $p$, we have $H^0(K_\p,\bar\psi^-) = H^0(K_{\p^\ast},\bar\psi^-) = 0$, so there are no local terms in degree zero in this long exact sequence. Also, $b^c\vert_{G_\p} = 0$ if and only if $b\vert_{G_{\p^\ast}} = 0$. Therefore, the kernel of \eqref{eq: H ord non-induced tangent cond} is canonically isomorphic to $H^1_{(N\p^\ast)}(O_K[1/Np], \til\Lambda^-_\dia \otimes V)$.

Recalling the decomposition \eqref{eq: S lem}, we conclude that $\sigma_{2,1}^\p \oplus \bigoplus_{v \mid Np} \big(\tau_{1,1}^\frv \oplus \tau_{2,2}^\frv\big)$ has kernel naturally isomorphic to the direct sum of the two kernels above. This gives the first isomorphism. 

The argument for the second is essentially identical. We replace $\sigma_{2,1}^\p$ with $\sigma_{1,2}^\p \oplus \sigma_{2,1}^\p$, which also factors through the summand $H^1(O_K[1/Np], \til\Lambda_\dia^- \otimes_{\til\Lambda} V)$ of the codomain of \eqref{eq: S lem}. This factorization is 
\begin{align*}
H^1(O_K[1/Np], \til\Lambda_\dia^- \otimes_{\til\Lambda} V) &\lra 
H^1(K_\p, \til\Lambda_\dia^- \otimes_{\til\Lambda} V) \oplus
H^1(K_\p, \til\Lambda_\#^- \otimes_{\til\Lambda} V) \\
b &\mapsto (b\vert_{G_{\p}}, b^c\vert_{G_{\p}}).
\end{align*}
Therefore the kernel of $\sigma_{1,2}^\p \oplus \sigma_{2,1}^\p \oplus \bigoplus_{v \mid Np} \big(\tau_{1,1}^\frv \oplus \tau_{2,2}^\frv\big)$ is naturally isomorphic to the direct sum of the two kernels from the factorization. Then, \eqref{eq:LES cone} computes this group by the same argument as before, where $S'$ is now the set of primes of $K$ dividing $Np$. 
\end{proof}

Now we can interpret maps out of the conormal modules of the CM locus in the ambient ordinary or split deformation space. 
\begin{cor}
\label{cor: normal maps}
For any finitely generated $\til\Lambda$-module $V$, we have canonical isomorphisms 
\[
\Hom_{\til\Lambda}(J/J^2, V) \lrisom H^1_{(N\p^*)}(O_K[1/Np], \til\Lambda^-_\dia \otimes V),
\]
%\item
%$\displaystyle 
\[
\Hom_{\til\Lambda}(J^s/(J^s)^2, V) \lrisom H^1_{(Np)}(O_K[1/Np], \til\Lambda^-_\dia \otimes V)
\]
that are functorial in $V$. 
\end{cor}
\begin{proof}
We claim that the injections 
\[
\Hom_\CM(\til\Lambda, \til\Lambda[V]) \rinj \Hom_\CM(R^\ast, \til\Lambda[V]), \qquad \ast \in \{\ord,\spl\}
\]
induced by the canonical surjections $R^\ord \rsurj R^\spl \rsurj \til\Lambda$ are compatible with the direct sum decompositions in the statement of Proposition \ref{prop: tangent spaces}. This follows from the fact that the image of these injections, say on an element $a \in \Hom_{\Z_p}(H_p, V)$, corresponds exactly to $\Ind_K^\Q \til\dia\cdot (1+a)$. By Lemma \ref{lem: CM char def}, induced deformations of $\bar\rho$ are exactly those that arise from homomorphisms out of $\til\Lambda$. Hence the statement follows from Proposition \ref{prop: interpret conormal modules}. 
\end{proof}

\subsection{Interpretation as class groups}

We arrive at the identification of the conormal modules. We apply the map $\til\delta$ of \eqref{eq: Z to Z-minus}, usually restricting it from its domain $\til\Lambda_{W'}^- \otimes_{W'} W$ to its subring $\til\Lambda_{W'}^- \otimes 1 \cong \til\Lambda_{W'}^-$. 

\begin{thm}
\label{thm: conormal spaces}
We have isomorphisms
\begin{enumerate}[label=(\roman*)]
\item $\displaystyle \cY^-_\infty(\psi^-) \otimes_{\til\Lambda_{W'}^-, \til\delta} \til\Lambda \lrisom J/J^2$ and
\item $\displaystyle \cX^-_\infty(\psi^-)\otimes_{\til\Lambda_{W'}^-, \til\delta} \til\Lambda \lrisom J^s/(J^s)^2$,
\end{enumerate}
compatibly with the natural surjections $J/J^2 \rsurj J^s/(J^s)^2$ and $\cY^-_\infty(\psi^-) \rsurj \cX^-_\infty(\psi^-)$. 
\end{thm}

\begin{rem}
Case (i) was originally proved by Hida; indeed, it follows immediately from the computation of $\Hom_\CM(R^\ord,  \til\Lambda[V])$ in \cite[Prop.~3.89, Thm.~5.33]{hida-HMF} combined with the argument establishing Corollary \ref{cor: adjoint Selmer}.
\end{rem}

\begin{proof}
Let $V$ be a finitely generated $\til\Lambda$-module. Since $\til\Lambda$ is a complete intersection (see Proposition \ref{prop: CM Galois}), we may apply global Tate duality in the form of Proposition \ref{prop: duality}.  Since  $H^i_{(\p)}(O_K[1/Np],\til\Lambda^-_\#(1)) = 0$ for $i \neq 2$ according to Proposition \ref{prop: Kummer}, the application to $T = \til\Lambda^-_\#(1)$ of the global Tate duality spectral sequence of Proposition \ref{prop: duality} degenerates. This yields 
\[
\Hom_{\til\Lambda}(H^2_{(\p)}(O_K[1/Np], \til\Lambda^-_\#(1)), V) \lrisom H^1_{(N\p^\ast)}(O_K[1/Np], \til\Lambda^-_\dia \otimes V).
\]

Because $\til\Lambda \otimes_{\til\delta,\til\Lambda^-_{W'}} (\til\Lambda_{W'}^-)_\#(1) \cong W \otimes_{W'} (\til\Lambda_{W'}^-)_\#(1)$, Proposition \ref{prop: Kummer} allows us to  replace $H^2_{(\p)}(O_K[1/Np], \til\Lambda^-_\#(1))$ by $\cY^-_\infty(\psi^-) \otimes_{\til\Lambda^-_{W'},\til\delta} \til\Lambda$. Corollary \ref{cor: normal maps} canonically identifies $\Hom_{\til\Lambda}(J/J^2,-)$ with
\[
H^1_{(N\p^\ast)}(O_K[1/Np], \til\Lambda^-_\dia \otimes -)
\]
as functors on finitely generated $\til\Lambda$-modules. Because both $J/J^2$ and $\cY^-_\infty(\psi^-) \otimes_{\til\Lambda^-_{W'}} \til\Lambda$ are finitely generated as $\til\Lambda$-modules, Yoneda's lemma implies the result (i). 

The proof of (ii) is essentially the same. Because $H^i(O_K[1/Np], \til\Lambda^-_\#(1)) = 0$ for $i > 2$, the duality spectral sequence of Proposition \ref{prop: duality} yields 
\[
\Hom_{\til\Lambda}(H^2(O_K[1/Np], \til\Lambda^-_\#(1)), V) \lrisom H^1_{(Np)}(O_K[1/Np], \til\Lambda^-_\dia \otimes V).
\]
By Proposition \ref{prop: Kummer}, we can replace $H^2(O_K[1/Np], \til\Lambda^-_\#(1))$ by $\cX^-_\infty(\psi^-) \otimes_{\til\Lambda_{W'}^-, \til\delta} \til\Lambda$. The rest of the proof proceeds as in the proof of (i). 
\end{proof}

\subsection{Proofs of main theorems}
\label{subsec: proof of main theorems}

In this section, we deduce the main result (Theorem \ref{thm:main}), and also Theorems \ref{thm: equiv main} and \ref{thm: equiv main2}, from the following main technical result. We resume writing $\bT^\CM$ in place of $\til\Lambda$. 
\begin{thm}
\label{thm: main R=T}
Assume conditions (0)--(4) of \S\ref{subsec: setup}. Then the surjection $R^\spl \rsurj \bT^\CM$ is an isomorphism if and only if $X(\psi^-) = 0$. 
\end{thm}

\begin{proof}
We know that $X(\psi^-) = 0$ if and only if $\cX^-_\infty(\psi^-) = 0$ by Proposition \ref{prop: X control}(i). Thus Theorem \ref{thm: conormal spaces} implies the theorem, as long as we know that $J^s = 0 \iff J^s/(J^s)^2 = 0$. This follows from Nakayama's lemma, as $J^s$ is contained in the maximal ideal of the complete Noetherian local ring $R^\spl$. 
\end{proof}

The main theorem now follows. 
\begin{proof}[{Proof of Theorem \ref{thm:main}}]
The conclusion of Theorem \ref{thm:main} is equivalent to the set 
\[
\Spec R^\spl(\oQ_p) \smallsetminus \Spec \bT^\CM(\oQ_p).
\]
being empty. When $X(\psi^-) = 0$, this immediately follows from Theorem \ref{thm: main R=T}. 
\end{proof}

Now we deduce Theorems \ref{thm: equiv main} and \ref{thm: equiv main2} from Theorem \ref{thm: main R=T} and the background in \S\ref{sec: interpolation}. 

\begin{proof}[{Proof of Theorems \ref{thm: equiv main} and \ref{thm: equiv main2}}]
It follows from Proposition \ref{prop: Rsplit valid} that the $p$-locally split condition is well-defined on the Galois representations associated to generalized eigenforms $g'$, $\bar g'$, even though their coefficient rings are not domains. Thus condition (c) of the theorems is equivalent to the map $\bT \ra A_{\bar g'}$ (resp.\ $\bT \ra A_{g'}$) factoring through $\bT \rsurj R^\spl$. 

Similarly, as we have noted that the CM condition is well-defined on generalized eigenforms in \S\ref{subsec: weights and GE}, the ``not CM'' condition (b) of both theorems is equivalent to the map $\bT \ra A_{\bar g'}$ (resp.\ $\bT \ra A_{g'}$) not factoring through $\bT \rsurj \bT^\CM$. 

\smallskip

\noindent
\textbf{Case of Theorem \ref{thm: equiv main2}.} 
Assume that $\cX^-_\infty(\psi^-)$ is infinite, which is equivalent to $\cX := \cX^-_\infty(\psi^-) \otimes_{\til\Lambda_{W'}^-, \til\delta} \til\Lambda$ being infinite. Then as a $\Lambda_\Q$-module (where this module structure arises from  $\beta: \Lambda_\Q \ra \til \Lambda$ discussed in Lemma \ref{lem: diaQ and diaK}), $\cX$ has support on some height 1 prime $P \subset \Lambda_\Q$. By Proposition \ref{prop: X control}(ii), $P$ has characteristic zero; hence $P = P_{k, \chi'}$ for some $p$-adic weight $(k,\chi')$. 

Let $E = E_{k,\chi'}$ denote the residue field of $P_{k,\chi'}$, which is a finite extension of $\Q_p$. We now consider the surjection with square-zero kernel 
\[
(R^\spl/(J^s)^2) \otimes_{\Lambda_\Q} E \rsurj \bT^\CM \otimes_{\Lambda_\Q} E. 
\]
By Theorem \ref{thm: conormal spaces}, its kernel surjects onto $\cX \otimes_{\Lambda_\Q} E$, which is non-zero. Because $\bT^\CM \otimes_{\Lambda_\Q} E$ is a finite product of finite extension fields over $E$, it has some factor $E_x = (\bT^\CM \otimes_{\Lambda_\Q} E)/\m_x$ with the following property: letting $\m'_x$ be the kernel of the surjection from $(R^\spl/(J^s)^2) \otimes_{\Lambda_\Q} E$ to $E_x$, $\cX \otimes_{\Lambda_\Q} E$ does not vanish under its natural map to $\m'_x/{\m'_x}^2$.

 Choose some $E_x$-1-dimensional quotient $\cX'$ of $\cX \otimes_{\Lambda_\Q} E_x$ and let $A_x := E_x[\cX'] \simeq E_x[\epsilon]/(\epsilon^2)$ be the corresponding square-zero extension of $E_x$. Then we may factor $(R^\spl/(J^s)^2) \otimes_{\Lambda_\Q} E \rsurj E_x$ through $A_x \rsurj E_x$. 

We now recall the discussion of generalized eigenforms and their attached Galois representations from \S\ref{subsec: weights and GE}. The composite $\bT \rsurj R^\spl \rsurj A_x$ corresponds (via the duality of Lemma \ref{lem: classical duality}) to a $p$-adic $p$-ordinary generalized eigenform $g'$ of $p$-adic weight $(k,\chi')$ with eigensystem corresponding to the composite $\bT \rsurj A_x \rsurj E_x$. The corresponding Galois representation $\rho_{g'} : G_{\Q,S} \ra \GL_2(A_x)$ arising as $\rho_{g'} := \rho_\bT \otimes_{\bT} A_x$ has the following properties:
\begin{enumerate}[label=(\alph*)]
\item The eigensystem induced by $\bT \ra E_x$ has CM and is congruent to $\bar f$, because it factors through $\bT \rsurj \bT^\CM$.
\item $g'$ does not have CM, because $\bT \ra A_x$ cannot factor through $\bT^\CM$: indeed, by Theorem \ref{thm: conormal spaces}, if it did factor, then $\cX$ must vanish when projected to $A_x$. But $\bT \ra A_x$ has been constructed so that it does not have this property. 
\item $\rho_{g'}$ is $p$-locally split, because $\bT \cong R^\ord \ra A_x$ factors through $R^\ord \rsurj R^\spl$. 
\end{enumerate}
These are the properties (a), (b), and (c) of Theorem \ref{thm: equiv main2}. We have also arranged for $A_x \simeq E_x[\epsilon]/(\epsilon^2)$, as claimed. 

For the converse, note that if $g'$ inducing $\bT \ra A_{g'}$ arises from the action on a generalized eigenform with properties (a), (b), and (c), then 
\begin{enumerate}[label=(\alph*)]
\item implies that the composite map $\bT \ra A_{g'} \ra E_{g'} \cong A_{g'}/\m_{g'}$ to the residue field of $A_{g'}$ amounts to an eigensystem that has CM,
\item implies that this map does not factor through $\bT \rsurj \bT^\CM$, and
\item implies that this map does factor through $\bT \rsurj R^\spl$. 
\end{enumerate}
Consider the image $A \subset A_{g'}$ of $R^\spl$, which is a local ring that is not a field (by (a) and (b)). Writing $\m_A \subset A$ for its maximal ideal, we consider the induced map $R^\spl \rsurj A/\m_A^2$. Its restriction to $J^s$ factors through $J^s/(J^s)^2$, and (b) implies that its image is non-zero. Since this image is a $\Z_p$-submodule of a $\Q_p$-vector space, we deduce from Theorem \ref{thm: conormal spaces} that $\cX^-_\infty(\psi^-)$ is infinite.

\smallskip

\noindent
\textbf{Case of Theorem \ref{thm: equiv main}.} 
The proof of this case is essentially the same. The only difference is that $\F$ plays the role of both $E$ and $E_x$, while $\bT^\CM \otimes_{\Lambda_\Q} \F$ is an Artinian local $\F$-algebra. Then the surjection of Artinian local algebras $R^\spl \otimes_{\Lambda_\Q} \F \rsurj \bT^\CM \otimes_{\Lambda_\Q} \F$ induces a surjection of the square-zero extension quotients. By Theorem \ref{thm: conormal spaces} and by letting $V = \F$ in Proposition \ref{prop: tangent spaces}, this surjection is 
\[
\F[X(\psi^-) \oplus (H_p \otimes_{\Z_p} \F)] \rsurj \F[H_p \otimes_{\Z_p} \F]
\]
(in the notation of Proposition \ref{prop: tangent spaces}). It is straightforward to deduce the result from here, using arguments analogous to the case of Theorem \ref{thm: equiv main2}. 
\end{proof}

\section{Commutative algebra}
\label{sec: resultant}

In this section, we set up a proposition from commutative algebra and deduce Theorem \ref{thm:main2}. 

\subsection{A proposition using the resultant}

The following lemma summarizes the theory of the resultant that we will require. 
\begin{lem}
\label{lem: resultant}
Let $R$ be a domain, and let $F(y), G(y) \in R[y]$ be polynomials. There is a resultant $\pi \in R$ of $F(y)$ and $G(y)$ with the following properties.
\begin{enumerate}
\item $\pi = 0$ if and only if $F(y)$ and $G(y)$ have a non-constant common factor. 
\item $\pi \in R \subset R[y]$ is an $R[y]$-linear combination of $F(y)$ and $G(y)$, i.e.
\[
\pi \cdot \frac{R[y]}{(F(y), G(y))} = 0. 
\]
\end{enumerate}
\end{lem}

In the following proposition, we refer to the \emph{generic rank} of a module $M$ over a domain $R$. This is defined to be the $\Frac(R)$-dimension of $M \otimes_R \Frac(R)$. 

\begin{prop}
\label{prop: GR 1}
Let $R$ be a complete Noetherian regular local ring. Let $S$ be an augmented reduced local $R$-algebra that is finitely generated and torsion-free as an $R$-module. Let $T$ be an augmented local $R$-algebra quotient of $S$, and denote by $K$ the kernel of $T \rsurj R$. 

Assume that $K/K^2$ is supported in codimension at least $2$ as an $R$-module. Then $T$ has generic rank equal to 1. 
\end{prop}

\begin{proof}
For this proof, given an augmented $R$-algebra $R \rinj A \rsurj R$, we denote by $A^c$ the $R$-module complement to the summand $R \subset A$ determined by the augmented $R$-algebra structure. That is, we have a canonical isomorphism of $R$-modules $A \cong R \oplus A^c$. We note that $A$ has generic rank 1 if and only if $A^c$ is $R$-torsion; we will implicitly use this equivalence in this proof. 

Denote by $J$ the kernel of $S \rsurj R$, and choose a minimal set $\cG$ of generators for the ideal $J$, which is also a minimal set of generators for $S$ as an $R$-algebra. Choose an element $y \in \cG$ and write $S'_y \subset S$, $T'_y \subset T$ for the sub-$R$-algebras generated by $y$. We observe that $S'_y \ra T'_y$ is a morphism of augmented $R$-algebras. 

We claim that it suffices to prove that $T'_y$ has generic rank 1 for all $y \in \cG$. Indeed, consider these product algebras with an augmented $R^\cG$-algebra structure
\[
R^\cG \ra \prod_{y \in \cG} T'_y \ra R^\cG \rsurj R,
\]
where the additional rightmost arrow is the diagonal projection homomorphism. We also have a natural map
\[
\prod_{y \in \cG} T'_y \rsurj T
\]
lying over the diagonal projection, inducing a surjection of $R$-modules
\[
\bigoplus_{y \in \cG} (T'_y)^c \rsurj T^c. 
\] 
Thus we observe that $T$ has generic rank 1 if and only if $T'_y$ has generic rank 1 for all $y \in \cG$. 

Having reduced to the case that $\#\cG = 1$, we render $S$ and $T$ as 
\begin{gather*}
S = \frac{R[y]}
{(y \cdot F_1(y), \dotsc, y \cdot F_n(y))}, \\
T = \frac{R[y]}
{(y \cdot F_1(y), \dotsc, y \cdot F_n(y), y \cdot G_1(y), \dotsc, y \cdot G_r(y))}. 
\end{gather*}
Now we have $J = (y)$. Note that $J/J^2$ is a torsion $R$-module generated by $y \pmod{J^2}$. Indeed, if this were not the case, let $m \geq 2$ be minimal such that $J^m/J^{m-1}$ is $R$-torsion. If $P(y) \in R[y]$ is a monic polynomial of minimal degree satisfied by $y$, then $y^m \mid P(y)$ because $J^i/J^{i+1}$ is free of rank one for $i < m$. Thus $y \cdot P(y)$ is a nilpotent element of $S$, contradicting our assumption that $S$ is reduced. 

Observe that $J/J^2$ is a cyclic $R$-module, generated by $y$, and isomorphic as an $R$-module to 
\[
J/J^2 \lrisom \frac{R}{(F_1(0), \dotsc, F_n(0))}.
\]
Likewise, its quotient $K/K^2$ is generated by the image $y'$ of $y$ in $T$ and is isomorphic as an $R$-module to 
\[
K/K^2 \lrisom \frac{R}{(F_1(0), \dotsc, F_n(0), G_1(0), \dotsc, G_r(0))}. 
\]
We claim that there exist a pair of polynomials $F(y), G(y)$ in the set 
\[
\{F_1(y), \dotsc, F_n(y), G_1(y), \dotsc, G_r(y)\}
\]
such that $R/(F(0), G(0))$ is supported in codimension 2. This follows directly from the assumption that $K/K^2$ is supported in codimension 2. 

We note that 
\[
\frac{R[y]}
{(y \cdot F(y))}, \qquad 
\frac{R[y]}
{(y \cdot F(y), y \cdot G(y))}
\]
are naturally augmented local $R$-algebras with augmentation ideal generated by $y$, and with a surjective augmented $R$-algebra map to $S$ and $T$, respectively. Therefore, it suffices to replace $S$ and $T$ with these algebras. Indeed, having done this, we observe that $J/J^2$ is torsion and $K/K^2$ is supported in codimension 2. We define
\[
T' := \frac{R[y]}
{(F(y), y \cdot G(y))},
\]
the quotient of $T$ by $(F(y))$, but note that $T'$ is not an augmented $R$-algebra. Because the kernel of $T \rsurj T'$ is a cyclic $R$-module (generated by $F(y)$), and we know that $T$ has generic rank at least 1, it will suffice to show that $T'$ is a torsion $R$-module. 

Let $\pi \in R$ be the resultant of the polynomials $F(y), y \cdot G(y) \in R[y]$. By Lemma \ref{lem: resultant}(2), we have
\[
\pi \cdot T' = 0.
\]
Thus we want to show that $\pi \neq 0$. By Lemma \ref{lem: resultant}(1), it suffices to prove that $F(y)$ and $y \cdot G(y)$ do not have any non-constant common factors. Assume, for the sake of contradiction, that there exists such a divisor $H(y) \in R[y]$. We may assume that $H(y)$ is irreducible and monic, since both $F(y)$ and $y \cdot G(y)$ are monic. We see that $H(y) \neq y$, because $F(0) \neq 0$. Next, note that $H(0)$ is not a unit in $R$, because if $H(y) \mid F(y)$ with quotient $Q(y)$, then $S \cong R[y]/(y \cdot H(y) \cdot Q(y))$ would not be a local ring (consider $S/\m_R S$). Then $H(0) \mid F(0)$ and $H(0) \mid G(0)$. This contradicts the fact that $R/(F(0), G(0))$ is finite, as it surjects onto the non-finite $R/(H(0))$. 
\end{proof}

\subsection{Proof of Theorem \ref{thm:main2}}

We will apply Proposition \ref{prop: GR 1} to $R^\spl$ in order to prove Theorem \ref{thm:main2}.

\begin{lem}
\label{lem: GR 1 assumptions}
Assume (0)--(4). Also assume that $p\nmid h_K$ and that $X^-_\infty(\psi^-)$ has finite cardinality. Then $R^\spl$ has generic rank 1 as a $\Lambda_\Q$-module. 
\end{lem}

\begin{proof}
We see that the conclusion of the lemma will follow from verifying that the assumptions of Proposition \ref{prop: GR 1} about $(R,S,T, K)$ are satisfied by 
\[
(R,S,T,K) = (\Lambda_\Q, R^\ord \cong \bT, R^\spl, J^s),
\]
where the augmented $\Lambda_\Q$-algebra structure of $R^\ord \cong \bT$ is understood to be defined by the ideal $J \cong I^\CM$.

Recall from Lemma \ref{lem: diaQ and diaK} the sequence of homomorphisms
\[
\Lambda_\Q \ra \bT \rsurj R^\spl \rsurj  \bT^\CM \risom \til\Lambda \rsurj \Lambda. 
\]
There, we see that these induce isomorphisms $\Lambda_\Q \risom \til\Lambda \risom \Lambda$ if and only if $p \nmid h_K$. Thus we apply the assumption $p \nmid h_K$ and identify $\Lambda_\Q \risom \bT^\CM \cong \til\Lambda$, treating $\bT \rsurj R^\spl$ as a morphism of augmented $\Lambda_\Q$-algebras. 

All of the assumptions of Proposition \ref{prop: GR 1}, except the one that $J^s/(J^s)^2$ is supported in codimension at least 2, are satisfied by the properties of $\bT$ checked in \S2, especially Lemma \ref{lem: reduced}. We will show that the remaining property follows from the assumption that $X^-_\infty(\psi^-)$ is finite in cardinality.

For $R = \Lambda_\Q$, an $R$-module is supported in codimension 2 if and only if it has finite cardinality. By Theorem \ref{thm: conormal spaces}, there is an isomorphism $\cX^-_\infty(\psi^-) \otimes_{\til\Lambda^-_{W'}} \til\Lambda\cong J^s/(J^s)^2$. When $p \nmid h_K$, $\cX^-_\infty(\psi^-) = X^-_\infty(\psi^-)$. Then the tensor product operation $\otimes_{\til\Lambda^-_{W'}} \til\Lambda$ preserves the finite cardinality property of these modules. 
\end{proof}

\begin{proof}[{Proof of Theorem \ref{thm:main2}}]
By Lemma \ref{lem: GR 1 assumptions}, we know that the assumptions of Theorem \ref{thm:main2} imply that $R^\spl$ has generic rank 1 as a $\Lambda_\Q$-module. 

Because the locus $\Spec(\bT^\CM) \subset \Spec (\bT)$ parameterizes exactly the CM $p$-adic eigenforms congruent to $\bar f$, it follows from the constructions of \S\ref{subsec: CM and nCM} that the map $x_g : \bT \ra \oQ_p$ of Lemma \ref{lem: classical duality} corresponding to a $p$-adic eigenform $g$ (congruent to $\bar f$) factors through $\bT \rsurj \bT^\mathrm{nCM}$ if $g$ does not have CM. We also know that $\rho_g$ is $p$-locally split if and only if $x_g$ factors through $\bT \rsurj R^\spl$. Thus it will suffice to show that 
\[
R^\mathrm{sn} := \bT^\mathrm{nCM} \otimes_\bT R^\spl
\]
is torsion as a $\Lambda_\Q$-module. 

Since we have already deduced that $R^\spl$ has generic rank 1, it suffices to show that the kernel of 
\[
R^\spl \rsurj R^\mathrm{sn}
\]
has generic rank 1. In view of Theorem \ref{thm: CM and non-CM components}, we want to show that the kernel $I_\mathrm{nCM}\subset\bT$ of $\bT \rsurj \bT^\mathrm{nCM}$, injects into $R^\spl$ under $\bT \rsurj R^\spl$. But this follows from the same theorem, as we see there that $I_\mathrm{nCM}$ injects under the composite quotient map $\bT \rsurj R^\spl \rsurj \bT^\CM \cong \Lambda_\Q$, with torsion cokernel. 
\end{proof}

\begin{rem}
\label{rem: thm main2 and GV}
The main result of Ghate--Vatsal \cite{GV2004} establishes the conclusion of Theorem \ref{thm:main2} upon assumptions (1')--(3') of \S\ref{subsec: setup}. The additional assumptions we rely on to prove Theorem \ref{thm:main2} are (0), (4), and the finiteness of $X^-_\infty(\psi^-)$. There, the authors use the fact that the ideal of $(\ast) \subset \bT$ generated by the image of $G_p$ under the ``$\ast$'' of \eqref{eq: ord form of rho_T} cuts out the quotient $\bT \rsurj R^\spl$. Our method hinges on the study of maximal square-zero augmented $\bT^\CM$-algebra quotients of $\bT$ (resp.\ $R^\spl$) over $\Lambda_\Q$. We found in Theorem \ref{thm: conormal spaces} that this maximal quotient is $\bT \rsurj \til\Lambda[\cY^-_\infty(\psi^-)]$ (resp.\ $R^\spl \rsurj \til\Lambda[\cX^-_\infty(\psi^-)]$), and that the image of $G_\p$ cuts out the quotient $\cY^-_\infty(\psi^-) \rsurj \cX^-_\infty(\psi^-)$. So our method relies on detecting ``$\ast$'' in the conormal module $I_\CM/I_\CM^2 \cong \cY^-_\infty(\psi^-)$. 
\end{rem}

%When to do markleft vs. markright depends upon page parity! 

\markleft{HARUZO HIDA}

\newpage

\markright{APPENDIX: LOCAL INDECOMPOSABILITY}

\setcounter{subsection}{0}

%\begin{appendix}
%\documentclass[oneside,draft,10pt]{amsart}
% \usepackage{amsbsy}
% \usepackage{amsmath}
% \usepackage{amstext}
% \usepackage{amsthm}
% \usepackage{amscd}
% \usepackage{amssymb}
%\newcommand{\Qb}{{\mathbb Q}}
%
%\setlength{\oddsidemargin}{0.25in}

% Set width of the text - What is left will be the right margin.
% In this case, right margin is 8.5in - 1.25in - 6in = 1.25in.
%\setlength{\textwidth}{6in}

% Set top margin - The default is 1 inch, so the following 
% command sets a 0.75-inch top margin.
%\setlength{\topmargin}{-0.30in}

% Set height of the text - What is left will be the bottom margin.
% In this case, bottom margin is 11in - 0.75in - 9.5in = 0.75in
%\setlength{\textheight}{9.0in}

\newcommand\h{{\operatorname{h}}}

\newcommand\GG{{\mathrm {G}}}
\newcommand\HH{{\mathrm {H}}}
\newcommand\TT{{\mathrm {T}}}

\newcommand\gA{{\mathfrak A}}
\newcommand\gV{{\mathfrak V}}
\newcommand\gd{{\mathfrak d}}  
\newcommand\gN{{\mathfrak N}}
\newcommand\gn{{\mathfrak n}}
\newcommand\gq{{\mathfrak q}}
\newcommand\gm{{\mathfrak m}}
\newcommand\gmsmall{{\mathfrak m}}
\newcommand\gmsmallprime{{\mathfrak m}}
\newcommand\f{{\mathbf f}}
\newcommand\e{{\mathbf e}}
\newcommand\PP{{\tiny P}}
\newcommand\gP{{\mathfrak P}}
\newcommand\gQ{{\mathfrak Q}}
\newcommand\gp{{\mathfrak p}}
\newcommand\gl{{\mathfrak l}}
\newcommand\gpsmall{{\mathfrak p}}
\newcommand\gqsmall{{\mathfrak q}}
\newcommand\hG{{\mathfrak h}}
\newcommand\gG{{\mathfrak g}}

\newcommand\Nm{{\mathrm {N}}}

\newcommand\tensor{{\> \otimes \> }}
\newcommand\Tr{{\mathrm {Tr}}}
\newcommand\barQ{{\overline{\Q}}}
\newcommand\divides{{\> \big| \>}}
\newcommand\notdivides{{\> \not\big| \>}}
\newcommand\modulo{{\hbox{mod }}}
\newcommand\card{{\mathrm {card}}}
\newcommand\iso{{\> \stackrel{\sim}{\longrightarrow} \> }} 
\newcommand\isom{{\> \stackrel{\sim}{=} \> }}
\newcommand\Isom{{\mathrm {Isom}}}
\newcommand\image{{\mathrm {image}}}

\newcommand\frob{{{ \Big[ \frac{L/K}{\gp} \Big]}}}

%\newtheorem{lemma}{Lemma}
%\newtheorem{cor}[lemma]{Corollary}
%\newtheorem{thm}[lemma]{Theorem}
%\newtheorem{defn}[lemma]{Definition}
%\newtheorem{prop}[lemma]{Proposition}

%\theoremstyle{plain}
%\newtheorem{thm}{Theorem}[section]
%\newtheorem{cor}[thm]{Corollary}
%\newtheorem{conj}[thm]{Conjecture}
%\newtheorem{lem}[thm]{Lemma}
%\newtheorem{prop}[thm]{Proposition}
%\newtheorem{ax}{Axiom}
%\newtheorem{qst}[thm]{Exercise}
%\newtheorem{defn}[thm]{Definition}
%\newtheorem{rem}[thm]{Remark}
%\newtheorem{expl}[thm]{Example}

%\theoremstyle{remark}
%\newtheorem{qst}{Exercise}[section]
%\newtheorem{notation}{Notation}
%\renewcommand{\thenotation}{}  
%\numberwithin{equation}{section}

\newcommand{\thmref}[1]{Theorem~\ref{#1}}
\newcommand{\secref}[1]{Section~\ref{#1}}
\newcommand{\lemref}[1]{Lemma~\ref{#1}}
\newcommand{\corref}[1]{Corollary~\ref{#1}}
\newcommand{\conjref}[1]{Conjecture~\ref{#1}}
\newcommand{\propref}[1]{Proposition~\ref{#1}}
\newcommand{\remref}[1]{Remark~\ref{#1}}
\newcommand{\explref}[1]{Example~\ref{#1}}
\newcommand{\qstref}[1]{Exercise~\ref{#1}}
\newcommand{\defnref}[1]{Definition~\ref{#1}}

%       Math definitions

%\newcommand{\Tate}{\operatorname{Tate}}
%\newcommand{\Ann}{\operatorname{Ann}}
%\newcommand{\Grass}{\operatorname{Grass}}
%\newcommand{\Hilb}{\operatorname{Hilb}}
%\newcommand{\Log}{\operatorname{Log}}
%\newcommand{\Inf}{\operatorname{Inf}}
%\newcommand{\Pic}{\operatorname{Pic}}
%\newcommand{\Div}{\operatorname{Div}}
%\newcommand{\codim}{\operatorname{codim}}
%\newcommand{\diag}{\operatorname{diag}}
%\newcommand{\hdim}{\operatorname{hdim}}
%\newcommand{\depth}{\operatorname{depth}}
%\newcommand{\Ext}{\operatorname{Ext}}
%\newcommand{\Ind}{\operatorname{Ind}}
%\newcommand{\res}{\operatorname{res}}
%\newcommand{\Sup}{\operatorname{Sup}}
%\newcommand{\rank}{\operatorname{rank}}
%\newcommand{\corank}{\operatorname{corank}}
%\newcommand{\IM}{\operatorname{Im}}
%\newcommand{\RE}{\operatorname{Re}}
%\newcommand{\id}{\operatorname{id}}
%\newcommand{\Id}{\operatorname{Id}}
%\newcommand{\Tor}{\operatorname{Tor}}
%\newcommand{\Ord}{\operatorname{Ord}}
%\newcommand{\ord}{\operatorname{ord}}
%\newcommand{\Reg}{\operatorname{Reg}}
%\newcommand{\Isog}{\operatorname{Isog}}
\newcommand{\Ker}{\operatorname{Ker}}
\newcommand{\Coker}{\operatorname{Coker}}
%\newcommand{\interval}[1]{\mathinner{#1}}
%\newcommand{\bs}{\backslash}

%@@@@@@@@@@@@@@@@@@@@@@@@@@@@@@@@@@@@@@@@
%\newcommand{\lra}{\longrightarrow}

%\newcommand{\p}{\parallel}

\newcommand{\hra}{\hookrightarrow}
\newcommand{\twra}{\twoheadrightarrow}
\newcommand{\xra}{\xrightarrow}
\newcommand{\xla}{\xleftarrow}
\newcommand{\wot}{\widehat\otimes}
\newcommand{\wh}{\widehat}

\newcommand{\ov}{\overline}
\newcommand{\ul}{\underline}
\newcommand{\vs}{\varsigma}
\newcommand{\vep}{\varepsilon}
\newcommand{\kp}{\kappa}
\newcommand{\al}{\alpha}
\newcommand{\s}{\sigma}
\newcommand{\LLm}{{\boldsymbol\Lambda}}
\newcommand{\llm}{{\boldsymbol\lambda}}
\newcommand{\Gam}{{\boldsymbol\Gamma}}
\newcommand{\Del}{{\boldsymbol\Delta}}
\newcommand{\mmu}{{\boldsymbol\mu}}
\newcommand{\vpb}{{\boldsymbol\varphi}}
\newcommand{\deltab}{{\boldsymbol\delta}}
\newcommand{\omegab}{{\boldsymbol\omega}}
\newcommand{\phib}{{\boldsymbol\phi}}
\newcommand{\etab}{{\boldsymbol\eta}}
\newcommand{\taub}{{\boldsymbol\tau}}
\newcommand{\pib}{{\boldsymbol\pi}}
%@@@@@@@@@@@@@@@@@@@@@@@@@@@@@@@@@@@@@@@@
\newcommand{\Abb}{{\mathbb A}}
\newcommand{\Cbb}{{\mathbb C}}
\newcommand{\Dbb}{{\mathbb D}}
\newcommand{\Eb}{{\mathbb E}}
\newcommand{\fb}{{\mathbb F}}
\newcommand{\Hb}{{\mathbb H}}
\newcommand{\Ib}{{\mathbb I}}
\newcommand{\jb}{{\mathbb J}}
\newcommand{\kb}{{\mathbb K}}
\newcommand{\Lb}{{\mathbb L}}
\newcommand{\Nb}{{\mathbb N}}
\newcommand{\pbb}{{\mathbb P}}
\newcommand{\sbb}{{\mathbb S}}
\newcommand{\Tbb}{{\mathbb T}}
\newcommand{\Vb}{{\mathbb V}}
\newcommand{\Wbb}{{\mathbb W}}
\newcommand{\Xb}{{\mathbb X}}
\newcommand{\Zb}{{\mathbb Z}}
\newcommand{\zb}{{\mathbb Z}}
\newcommand{\Zp}{{\mathbb Z}_p}
\newcommand{\Zl}{{\mathbb Z}_\ell}
\newcommand{\Qp}{{\mathbb Q}_p}
\newcommand\A{{\mathbb A}}
\newcommand\C{{\mathbb C}}
\newcommand\E{{\mathbb E}}
\newcommand\G{{\mathbb G}}
\newcommand\N{{\mathbb N}}
\newcommand\Pbb{{\mathbb P}}
\newcommand\R{{\mathbb R}}
\newcommand\T{{\mathbb T}}
\newcommand{\oF}{\overline{\mathbb F}}

\newcommand{\Hbf}{{\mathbf H}}
\newcommand{\Ibf}{{\mathbf I}}
\newcommand{\cb}{{\mathbf c}}
\newcommand{\Cb}{{\mathbf C}}
\newcommand{\Gb}{{\mathbf G}}
\newcommand{\Wb}{{\mathbf W}}
\newcommand{\Vbb}{{\mathbf V}}
\newcommand{\ib}{{\mathbf i}}
\newcommand{\Jbf}{{\mathbf J}}
\newcommand{\hb}{{\mathbf h}}
\newcommand{\xb}{{\mathbf x}}
\newcommand{\ab}{{\mathbf a}}
\newcommand{\bb}{{\mathbf b}}
\newcommand{\eb}{{\mathbf e}}
\newcommand{\Pb}{{\mathbf P}}
\newcommand{\Mb}{{\mathbf M}}
\newcommand{\fbb}{{\mathbf f}}
\newcommand{\gbb}{{\mathbf g}}
\newcommand{\pbof}{{\mathbf p}}
\newcommand{\Sbb}{{\mathbf S}}
\newcommand{\Ebb}{{\mathbf E}}
\newcommand{\0}{{\mathbf 0}}
\newcommand{\Tb}{{\mathbf T}}
\newcommand{\Db}{{\mathbf D}}
\newcommand{\Lbd}{{\mathbf L}}
\newcommand{\vb}{{\mathbf v}}
\newcommand{\db}{{\mathbf d}}
\newcommand{\Xbb}{{\mathbf X}}
\newcommand{\sbf}{{\mathbf s}}

\newcommand{\Ac}{{\mathcal A}}
\newcommand{\Bc}{{\mathcal B}}
\newcommand{\Cc}{{\mathcal C}}
\newcommand{\dc}{{\mathcal D}}
\newcommand{\Ec}{{\mathcal E}}
\newcommand{\fc}{{\mathcal F}}
\newcommand{\gc}{{\mathcal G}}
\newcommand{\hc}{{\mathcal H}}
\newcommand{\Ic}{{\mathcal I}}
\newcommand{\Jc}{{\mathcal J}}
\newcommand{\Kc}{{\mathcal K}}
\newcommand{\lc}{{\mathcal L}}
\newcommand{\mc}{{\mathcal M}}
\newcommand{\nc}{{\mathcal N}}
\newcommand{\Oc}{{\mathcal O}}
\newcommand{\Pc}{{\mathcal P}}
\newcommand{\qc}{{\mathcal Q}}
\newcommand{\Rc}{{\mathcal R}}
\newcommand{\Sc}{{\mathcal S}}
\newcommand{\Tc}{{\mathcal T}}
\newcommand{\Uc}{{\mathcal U}}
\newcommand{\Vc}{{\mathcal V}}
\newcommand{\Wc}{{\mathcal W}}
\newcommand{\Xc}{{\mathcal X}}
\newcommand{\Yc}{{\mathcal Y}}
\newcommand{\Zc}{{\mathcal Z}}
\newcommand{\Fc}{{\mathcal F}}
\newcommand\sA{{{\mathcal A}}}
\newcommand\sF{{{\mathcal F}}}
\newcommand\sD{{{\mathcal D}}}
\newcommand\sE{{{\mathcal E}}}
\newcommand\sO{{{\mathcal O}}}
\newcommand\sK{{{\mathcal K}}}
\newcommand\sL{{{\mathcal L}}}
\newcommand\sM{{{\mathcal M}}}
\newcommand\sN{{{\mathcal N}}}
\newcommand\sQ{{{\mathcal Q}}}
\newcommand\sR{{{\mathcal R}}}
\newcommand\sS{{{\mathcal S}}}
\newcommand\sT{{{\mathcal T}}}
\newcommand\sV{{{\mathcal V}}}
\newcommand\sZ{{{\mathcal Z}}}

\newcommand{\EG}{{\mathfrak E}}
\newcommand{\eG}{{\mathfrak e}}
\newcommand{\WG}{{\mathfrak W}}
\newcommand{\AG}{{\mathfrak A}}
\newcommand{\aG}{{\mathfrak a}}
\newcommand{\fG}{{\mathfrak f}}
\newcommand{\lG}{{\mathfrak l}}
\newcommand{\IG}{{\mathfrak I}}
\newcommand{\iG}{{\mathfrak i}}
\newcommand{\qG}{{\mathfrak q}}
\newcommand{\pG}{{\mathfrak p}}
\newcommand{\PG}{{\mathfrak P}}
\newcommand{\oG}{{\mathfrak o}}
\newcommand{\OG}{{\mathfrak O}}
\newcommand{\mG}{{\mathfrak m}}
\newcommand{\zG}{{\mathfrak z}}
\newcommand{\MG}{{\mathfrak M}}
\newcommand{\nG}{{\mathfrak n}}
\newcommand{\QG}{{\mathfrak Q}}
\newcommand{\CG}{{\mathfrak C}}
\newcommand{\LG}{{\mathfrak L}}
\newcommand{\dG}{{\mathfrak d}}
\newcommand{\ogg}{{\mathfrak o}}
\newcommand{\SG}{{\mathfrak S}}
\newcommand{\sG}{{\mathfrak s}}
\newcommand{\rG}{{\mathfrak r}}
\newcommand{\RG}{{\mathfrak R}}
\newcommand{\ZG}{{\mathfrak Z}}
\newcommand{\HG}{{\mathfrak H}}
\newcommand{\FG}{{\mathfrak F}}
\newcommand{\UG}{{\mathfrak U}}
\newcommand{\uG}{{\mathfrak u}}
\newcommand{\yG}{{\mathfrak y}}
\newcommand{\XG}{{\mathfrak X}}
\newcommand{\xG}{{\mathfrak x}}
\newcommand{\tG}{{\mathfrak t}}
\newcommand{\KG}{{\mathfrak K}}
%@@@@@@@@@@@@@@@@@@@@@@@@@@@@@@@@@@@@@@@@@@@@@
\newcommand{\Gm}{{\G}_m}
\newcommand{\wT}{{\widehat T}}
\newcommand{\Inv}{\rm Inv}
\newcommand{\pw}{{\parallel\!\!\wr}}
\newcommand{\orho}{{\overline\rho}}
\newcommand{\od}{{\overline\delta}}
\newcommand{\okb}{{\overline\kb}}
\newcommand{\oep}{{\overline\varepsilon}}
\newcommand{\ophi}{{\overline\varphi}}
\newcommand{\oa}{{\overline\alpha}}
\newcommand{\oP}{\overline{P}}
\newcommand{\vr}{\varrho}
\newcommand{\vp}{\varphi}
\newcommand{\lm}{\lambda}
\newcommand{\Lm}{\Lambda}
\newcommand{\vpi}{\varpi}
%@@@@@@@@@@@@@@@@@@@@@@@@@@@@@@@@@@@@@@@@@@@@

 \newcommand{\BG}{{\mathfrak B}}
 \newcommand{\TG}{{\mathfrak T}}
\newcommand{\DG}{{\mathfrak D}}
\newcommand{\rhob}{{\boldsymbol\rho}}
\newcommand{\epsilonb}{{\boldsymbol\epsilon}}
\newcommand{\vepb}{{\boldsymbol\vep}}
\newcommand{\psib}{{\boldsymbol\psi}}
\newcommand{\betab}{{\boldsymbol\beta}}
\newcommand{\kpb}{{\boldsymbol\kappa}}
\newcommand{\lan}{\langle}
\newcommand{\ran}{\rangle}
\newcommand{\Fb}{{\mathbf F}}
\newcommand{\cha}{\operatorname{char}}
\newcommand{\edim}{\operatorname{edim}}
\newcommand{\gr}{\operatorname{gr}}
\newcommand{\Coind}{\operatorname{C5ind}}

%\begin{document}

%\appendix

%\section{Local indecomposablity via a presentation of the Hecke algebra, by Haruzo Hida}
\section*{Appendix. \for{toc}{Local indecomposablity via a presentation of the Hecke algebra, by Haruzo Hida}\except{toc}{\ \\Local indecomposability via a presentation of the Hecke algebra\\ by Haruzo Hida}}
\renewcommand{\theequation}{A.\arabic{equation}}
\renewcommand{\thesubsection}{A.\arabic{subsection}}

\subsection{Summary}
Let  $p\ge5$  be a prime.
In this appendix,
we give a proof of Greenberg's conjecture ((CG) in the main text) of local indecomposability 
of a non-CM residually CM Galois representation based on the presentation of the universal ring
given in \cite{hida-cyclicity} (so, the proof is different from the one given in the main text).
We impose an extra assumption (H3-4) in addition to the set of the assumptions made in the main text
(we list our set of assumptions as (H0--4)  below).
We use the notation introduced in the main text.
For each Galois representation  $\rho$  of  $G_K$,
we write  $K(\rho)=\barQ^{\Ker(\rho)}$  for the splitting field of  $\rho$.
We fix an algebraic closure  $\ov\fb$  of  $\fb$  and write  $\WG$  for the Witt vector ring  $W(\ov\fb)$.

A deformation  $\rho_A:G_K\ra\GL_2(A)$  for an algebra  $A$  in  CNL$_{W}$  of a the representation  $\orho=\Ind_K^\Q\ov\psi:G_\Q\ra\GL_2(\fb)$  as in \S1.2.2
is said to be minimal if  $\rho_A(I_l)\cong\orho(I_l)$
by the reduction map for all primes  $l|N$ \cite[\S3.1, pg.~715]{DFG2004}.
By an $R=\T$ theorem (e.g., \cite[Thm.~2.3]{diamond1997}),
we have a local ring  $\T$  of the ordinary Hecke algebra  
and its Galois representation  $\rho_\T:G_\Q\ra\GL_2(\T)$  giving a universal ordinary pair with
$\T$  being naturally an algebra over the weight Iwasawa algebra  $\Lambda:=W[[(1+p\Zp)]]\cong W[[T]]$.
We assume that  $\Spec(\T)$  contains a non-CM component $\Spec(\T^\text{\rm nCM})$.
We made the following assumptions in \cite{hida-cyclicity} 
to prove a presentation of  $\T$  over  $\Lambda$: 
\begin{enumerate}
  \item[(H0)]  $\psi^-|_{G_p}\ne1$  (a local condition),
  \item[(H1)]  $\psi$  has conductor  ${\frc'}$  such that  ${\frc'}+{\frc'}^c=O_K$  and  $\pG^*\nmid{\frc'}$, 
\item[(H2)]   the character
$\psi^-$  has order at least  $3$  (a global condition),
\item[(H3)]  the class number  $h_K$  of  $K$  is prime to  $p$,
\item[(H4)]  the class number  $h_{K(\psi^-)}$  of the splitting field  $K(\psi^-)=\oQ^{\Ker(\psi^-)}$  of  $\psi^-$  is prime to  $p$.
\end{enumerate}
Assuming  $\T\ne\Lambda$,
the minimal presentation we found in \cite{hida-cyclicity} has the following form: 
\begin{equation}\label{presen}
\T\cong\Lambda[[T_-]]/(T_-S_+).
\end{equation}
Here the ring  $\Lambda[[T_-]]$  is the one variable power series ring over  $\Lambda$  with variable  $T_-$  
and  $S_+$  is a power series in  $\Lambda[[T_-]]$  prime to  $T_-$.
We have an involution  $\s$  over  $\Lambda$  acting on  $\T$   
corresponding to the operation  $\rho\mapsto\rho\otimes\chi$  for  $\chi:=\left(\frac{K/\Q}{}\right)$.
Non-triviality of  $\s$  is equivalent to the existence of a non CM component of  $\Spec(\T)$.
This involution  $\s$  extends to an involution  $\s_\infty$  
of  $\Lambda[[T_-]]$  so that  $\s_\infty(T_-)=-T_-$  and  $\s_\infty(S_+)=S_+$.
To prove the presentation, we made in \cite{hida-cyclicity} some extra conditions 
whose removal will be discussed in \S\ref{univR}.
To have one-variable presentation in \eqref{presen},
we need to assume  $p\nmid h_K$  (otherwise, we could have variables fixed by  $\s_\infty$  in the presentation).

Let  $\T_+$  be the subring of  $\T$  fixed by  $\s$.
Let  $\T^\text{\rm nCM}:=\Lambda[[T_-]]/(S_+)$  and  $\T^\text{\rm CM}:=\Lambda[[T_-]]/(T_-)=\Lambda$,
and write  $\Theta$  for the image of  $T_-$  in  $\T$.
Since the CM Galois deformation  $\rho_{\T^\text{CM}}$  into  $\GL_2(\T^\text{CM})$  is induced from  $K$,
the involution  $\s$  is trivial on  $\T^\text{CM}$; so, the image  of  $T_-$  with  $\s_\infty(T_-)=-T_-$
vanishes in  $\T^\text{CM}$; so,  $\Theta$  lives in  $(0\times\T^\text{\rm nCM})\cap\T$
(this is also clear from  $\T^\text{CM}=\Lambda[[T_-]]/(T_-)$).
This $\Theta$  plays the role of  $L_p^-(\psi^-)$  in the main text 
in the sense that  $\T^\text{nCM}/(\Theta)\cong\T^\text{nCM}\otimes_\T\T^\text{CM}\cong\Lambda/(L_p^-(\psi^-))$
(the identity of the congruence modules) even if  $\Theta$  lives in  $\T^\text{nCM}$  while  $L_p^-(\psi^-))\in\T^\text{CM}=\Lambda$.
Then  $\T\hra\T^\text{\rm nCM}\times\T^\text{\rm CM}$  whose cokernel is isomorphic to  $\T^\text{\rm nCM}/(\Theta)$  
as  $\T$-modules and  $(\Theta)=(0\times\T^\text{\rm nCM})\cap\T$.
The congruence module  $\T^\text{\rm nCM}/(\Theta)$  after extending scalars to
$\WG$  is isomorphic to  $(\T^\text{\rm CM}\wot_W \WG)/(\lc^-_p(\psi^-))$
for the anticyclotomic Katz $p$-adic L-function  $\lc^-_p(\psi^-)$  (of branch character  $\psi^-$, denoted in the main text as $\frL^-_p(\psi^-)$; see \corref{Ycor}); so,
$\Theta$  is a generator of  $I_\text{\rm CM}$  and in this sense, we regard  $\Theta\in\T^\text{\rm nCM}$.

Let  $\PG$  be a prime factor of  $\pG$  in  $K(\orho)$  (the splitting field of  $\orho$).
Write the image of  $U(p)$
in  $\T$ as  $u$.
Writing the local Artin symbol  $[x,K_\pG]$  (identifying  $K_\pG=\Qp$),
for the residual degree  $f$  of  $\PG$,
the semi-simplification of
$\rho_\T([p,K_\pG]^f)$  is a conjugate of  $\left(\begin{smallmatrix}u^{-f}&0\\0&u^f\end{smallmatrix}\right)$
as  $\det(\rho_\T([p,\Qp]^f))=1$.
Note here  $\psi^-([p,K_\pG]^f)=1$  and  $u^{2f}\equiv\psi^-([p,K_\pG]^f)=1\mod\mG_\T$  (as  $u\equiv\psi([p,K_\pG])\mod\mG_\T$).
Put $a=u^{2f}-1\in\mG_\T$,  and
for the $\Zp$-subalgebra  $W_1$  of  $W$  generated by the values of  $\psi^-$  over  $G_p$,
define  $\Lambda_1:=W_1[[T,a]]$  to be the subalgebra
of  $\T$  topologically generated over  $W_1[[T]]\subset\Lambda$  by  $a$.
 
\begin{thm}
	\label{thm: app main}
Let the notation be as above.
Assume  {\em(H0--4)}  and  $\s\ne\id$  on  $\T$.
Let  $\ov I_\pG$  be the wild $\pG$-inertia subgroup of  $\Gal(K(\rho_\T)/\Q)$ for the splitting field  $K(\rho_\T)$  of  $\rho_\T$.
Then we have a decomposition  $\ov I_\pG=\Uc\rtimes\Gal(\Q_\infty/\Q)$  for the $\Zp$-extension  $\Q_\infty/\Q$,
where  $\Uc$  is an abelian group mapped by  $\rho_\T$  into the unipotent radical of a Borel subgroup in  $\GL_2(\T)$  whose logarithmic image  $\uG= Lie(\Uc)$
(in the nilpotent Lie $\Lambda$-algebra  $\T$)  is equal to  $\Theta\cdot\Lambda_1$.
In short, we have an isomorphism  $\rho_\T(\ov I_\pG)\cong\left\{\left(\begin{smallmatrix}t^{\Zp}&\Theta\Lambda_1\\
0&1\end{smallmatrix}\right)\right\}\subset\GL_2(\T)$, where  $t=1+T\in\Lambda$.
\end{thm}

%\medskip
This theorem supplies us with a very explicit unipotent element  $\left(\begin{smallmatrix}1&\Theta\\0&1\end{smallmatrix}\right)$
in the image of  $\rho_\T$  with  $(\Theta\T\wot_W \WG)\cap\Lambda_\WG=(\lc_p^-(\psi^-))$;
therefore, we can answer the question of Greenberg:

%\medskip\noindent
%{\bf Corollary~1:} {\em 
\begin{cor}
	\label{cor: app1}
	Assume  {\em(H0--4)}  and  $\s\ne\id$  on  $\T$.
	For all prime divisors  $P\in\Spec(\T^\text{\rm nCM})$ with associated Galois representation  $\rho_P$, the following conditions are equivalent:
	\begin{enumerate}
		\item the Galois representation  $\rho_P$  is completely reducible over the inertia group  $I_p$  at  $p$,
		\item  $P\in\Spec(\T^\text{\rm nCM})\cap\Spec(\T^\text{\rm CM})$,
		\item  $P|(\lc^-_p(\psi^-)\Lambda_\WG\cap\Lambda)$.
\end{enumerate}
\end{cor}

%\medskip
As described in the main text, from \cite{emerton2004} and \cite[Prop.~11]{ghate2005},
the above corollary implies:

%\medskip\noindent
%{\bf Corollary 2:} (Coleman's question): {\em 
\begin{cor}[Coleman's question]
	\label{cor: app2}
	Assume  {\em(H0--4)}.
For every classical modular form  $f$  of weight  $k\ge2$  and of level  $N$
with residual representation  $\orho$, 
write  $g$  for the $p$-critical stabilization of the primitive form associated to  $f$.
Then  $g$ is in the image of  $(q\frac{d{}}{dq})^{k-1}$  if and only if  $f$  has complex multiplication.
\end{cor}

\subsection{Presentation of a Galois deformation ring}\label{univR}
For a set  $Q$  of Taylor--Wiles primes satisfying the conditions (Q0--10) in \cite[\S\S3-4]{hida-cyclicity},
we write  $K(\orho)^{(pQ)}$  for the maximal $p$-profinite extension of  $K(\orho)$  unramified outside  $\{p\}\sqcup Q$.
We simply write  $K(\orho)^{(p)}$  for  $K(\orho)^{(pQ)}$  if  $Q=\emptyset$.
Let  $G_Q:=\Gal(K(\orho)^{(pQ)}/\Q)$  and  $H_Q:=\Gal(K(\orho)^{(pQ)}/K)$  with  $G=G_\emptyset$  and  $H=H_\emptyset$.
We first note that  $G_Q=\Gal(K(\orho)^{(pQ)}/K(\orho))\rtimes\Gal(K(\orho)/\Q)$  and
$H_Q=\Gal(K(\orho)^{(pQ)}/K(\orho))\rtimes\Gal(K(\orho)/K)$  as  $p>2$  and  $p\nmid[K(\orho):\Q]$.
We fix such a decomposition; so, $\Gal(K(\orho)/\Q)\cong\Delta_G$
for a subgroup  $\Delta_G$  of  $\Gal(K(\orho)^{(p)}/\Q)$.
Write  $\Delta\subset\Delta_G$  for the subgroup isomorphic to  $\Gal(K(\orho)/K)$; so,
$[\Delta_G:\Delta]=2$.

Let  $N=DN_{K/\Q}({\frc'})$.
Let  $\hb^Q$  be the big Hecke algebra described in \cite[\S1]{hida-adjoint} for each  $Q$.
We have a local ring  $\T^Q$  of  $\hb^Q$  whose residual representation is isomorphic to  $\orho$.
Let  $\rhob^Q:G_Q\ra\GL_2(\T^Q)$  be the Galois representation of  $\T^Q$  such that  $\Tr(\rhob^Q(\Frob_l))$  for primes  $l$  outside  $\{l|Np\}\sqcup Q$
is given by the image in  $\T^Q$  of the Hecke operator  $T(l)$.
On  $\T^Q$,
we have an involution  $\s$  with the property that  $(\rhob^Q)^\s\cong\chi\otimes\rhob^Q$
for the quadratic character  $\chi=\left(\frac{K/\Q}{}\right)$.
Put  $\T^Q_\pm:=\{h\in\T^Q|\s(h)=\pm h\}$.
Let  $\Ib^Q:=\T^Q(\s-1)\T^Q=\T^Q\T^Q_-$ (the $\s$-different) and  $\T^Q_\text{\rm CM}:=\T^Q/\Ib^Q$.
It is known that  $\T^Q$  and  $\T^Q_\text{\rm CM}$  are reduced algebras finite flat over  $\Lambda$.
Further we have an algebra decomposition  $\T^Q\otimes_\Lambda\Frac(\Lambda)=\Frac(\T^Q_\text{\rm CM})\times\Frac(\T^Q_\text{\rm nCM})$
for  $\T^Q_\text{\rm nCM}\cong\T^Q/(\Frac(\T^Q_\text{\rm CM})\times0)\cap\T^Q$.  
In the above notation,
if  $Q=\emptyset$,
we remove the superscript or subscript  $Q$  from the notation.
If  $\s$  is the identity on  $\T$,
we have  $\T^\text{\rm nCM}=0$.
Otherwise the subring  $\T^\text{\rm nCM}_+$  fixed by  $\s$  is a non-trivial $\Lambda$-algebra.
The theorem proven in \cite[Thms.~B and 5.4]{hida-cyclicity} is:

\begin{thm}\label{TWinvthm}
Assume {\em (H0-H4)},  $\s\ne\id$  on  $\T$  and that  $p$  splits in  $K$.
Let  $\Spec(\T)$  be a connected component of  $\Spec(\hb)$
associated to the induced Galois representation $\orho=\Ind_K^\Q\ov\psi$  
for the reduction  $\ov\psi$  of  $\psi$ modulo  $\mG_W$
for the maximal ideal  $\mG_W$  of  $W$.
Then the following assertions hold:
\begin{enumerate}
  \item  We have presentations  $\T\cong\Lambda[[T_-]]/(T_-S_+)$,
$\T_+\cong\Lambda[[T_-^2]]/(T_-^2S_+)$, $\T^\text{\rm nCM}\cong\Lambda[[T_-]]/(S_+)$
and  $\T_+^\text{\rm nCM}\cong\Lambda[[T_-^2]]/(S_+)$  such that the involution $\s_\infty:T_-\mapsto-T_-$  over  $\Lambda$
fixes the power series  $S_+\in \Lambda[[T_-^2]]$  and induces  $\s$  on  $\T$.
  \item  The rings   $\T$,  $\T_+$, $\T^\text{\rm nCM}$, $\T_+^\text{\rm nCM}$  
are all local complete intersections free of finite rank over $\Lambda$.
  \item  The $\T^\text{\rm nCM}$-ideal $\Ib=\T(\s-1)\T\subset\T^\text{\rm nCM}$  is principal
and is generated by the image  $\Theta$  of  $T_-$
with $\theta:=\Theta^2 \in \T_+$, and  $\Theta$  is not a zero divisor.
The element $\Theta$ generates the $\T_+^\text{\rm nCM}$-module $\T_-^\text{\rm nCM}$  which is free over  $\T_+^\text{\rm nCM}$, and
$\T^\text{\rm nCM}= \T_+^\text{\rm nCM}[\Theta]$  is free of rank $2$ over $\T_+^\text{\rm nCM}$.
\end{enumerate}
\end{thm}
\begin{proof}
The result \cite[Thm.\ 4.10 and Prop.\ 6.2]{hida-cyclicity} asserts that  $\T=\Lambda[\Theta]$  with  $\s(\Theta)=-\Theta$; 
so, we have a surjection  $\pi:\Lambda[[T_-]]\twra\T$  with  $\pi(T_-)=\Theta$,  
and \cite[Thms.~A and B]{hida-cyclicity} asserts that  $\T$  is a local complete intersection over  $\Lambda$.
Thus  $\T\cong\Lambda[[T_-]]/(S_-)$  for a power series  $S_-\in\Lambda[[T_-]]$.
By the construction of  $\pi$  of \cite[\S4]{hida-cyclicity}    via a Taylor--Wiles patching argument,
we have an involution  $\s_\infty$  of $\Lambda[[T_-]]$  lifting  $\s$  such that  $\s(T_-)=-T_-$  and  $\s(S_-)=-S_-$;
so, we have  $T_-|S_-$  and hence  $S_-=T_-S_+$.
Since  $\T$  is reduced,  $T_-$  and  $S_+$  are co-prime in  $\Lambda[[T_-]]$.
This shows the assertion (1).
The assertions (2) and (3) follow from \cite[Thm.~B]{hida-cyclicity}.

Strictly speaking, the patching argument is given in \cite{hida-cyclicity} under the following extra assumptions:
\begin{enumerate}
  \item[(h2)]  $N:=DN_{K/\Q}({\frc'})$  for an  $O$-ideal  ${\frc'}$  prime to  $D$  
with square-free  $N_{K/\Q}({\frc'})$ 
(so, $N$  is cube-free),
  \item[(h3)]  $p$  is prime to  $N\prod_{l|N}(l-1)$  for prime factors  $l$  of  $N$.
\end{enumerate}
Here is the reason why we can remove these two assumptions:
We studied the minimal deformation problem in \cite{hida-cyclicity} over the absolute Galois group  $G_\Q$,
but as was explained in \cite[pg.~717]{DFG2004}, under the condition that  $p\nmid|\orho(I_l)|$  (which holds in our case),
all minimal deformations factor through  $G$, and
considering the deformation problem over  $\{G_Q\}_Q$  for appropriate sets  $Q$  of Taylor--Wiles primes
satisfying \cite[\S3 (Q0--8)]{hida-adjoint},
every argument in the proof of \cite[Thm.~5.4]{hida-cyclicity} goes through for the above choice of  $\T^Q$ (as easily checked),
and thus we obtain the theorem.
Indeed,
we used (h3) in \cite{hida-cyclicity} just because
the universal minimal ordinary Galois representation of prime-to-$p$ conductor  $N$ 
(considered in \cite{hida-cyclicity}) factors through  $G$; so,
just imposing deformations to factor through  $G$  the arguments simply work; so,
we do not need to assume (h3).
The condition (h2) is assumed to guarantee the big Hecke algebra 
is reduced, but again, all deformations over  $G$  has prime-to-$p$ conductor equal to  $N$  which is equal to the prime-to-$p$ conductor of its determinant
(the Neben character).
Then, by the theory of new forms,
the Hecke algebra is reduced if its tame character has conductor equal to the tame level;
so, we do not need (h2).
\end{proof}
Since  $\s$  acts trivially on  $\T^\text{\rm CM}=\T/(\Theta)$,
writing  $\rho:=(\rho_\T\ \modulo(\Theta))$,
we find  $\rho\cong\rho\otimes\chi$  for  $\chi=\left(\frac{K/\Q}{}\right)$.
Note that  $\rho$  is a minimal deformation of  $\orho$; so, 
it factors through  $G$.
Thus by \cite[Lem.~3.2]{doi-hida-ishii}  applied to  $\gc=G$  and  $\hc=H$  (under the notation of the lemma), 
we find  $\rho\cong\Ind_K^\Q\Psi$  for a character  $\Psi:H\ra\T^{\text{\rm CM},\times}$  unramified outside  ${\frc'}\pG$  deforming  $\ov\psi$.
Let  $\Gamma_\pG$  be the Galois group over  $K(\orho)$  of the maximal $p$-abelian extension of  $K$  inside  $K(\orho)^{(p)}$
unramified outside  $\pG$.
By  $p\nmid h_K$,
$\Gamma_\pG\cong O_\pG^\times\otimes_\Z\Zp$,
and hence  $W[[\Gamma_\pG]]\cong\Lambda$ canonically via  $\Zp^\times= O_\pG^\times$.
We identify the two rings.
Since  $p\nmid[K(\orho):K]$,
there exists a class field  $K(\pG)/K$  in  $K(\orho)^{(p)}$  with  $\Gal(K(\pG)/K)\cong \Gamma_\pG$  by Artin symbol.
Define a character  $\Phi:G_K\ra W[[\Gamma_\pG]]^\times=\Lambda^\times$
given by  $\Phi(\tau)=\psi(\tau)\tau|_{K(\pG)}$.
Then  $\Phi$  factors through  $H$.
Since  $(\Lambda,\Phi)$  for the character  $\Phi:H\ra\Lambda^\times$
is a universal pair for the deformation problem of  $\ov\psi$  unramified outside  $\pG{\frc'}$  over the group  $H$,
we have a canonical surjective algebra homomorphism  $\Lambda\twra\T^\text{\rm CM}$  inducing  $\Psi$.
By the same argument which proves \cite[Cor.~2.5]{hida-cyclicity}, this is an isomorphism.
We record this fact as
\begin{cor}\label{Lambdavep}
We have isomorphisms  $\T_+/(\theta)\cong\T/(\Theta)=\T^\text{\rm CM}\cong\Lambda$,
where  $\theta=\Theta^2\in\T_+^\text{\rm nCM}$.
\end{cor}

Recall  $G=\Gal(K(\rho)^{(p)}/\Q)$  and  $H=\Gal(K(\rho)^{(p)}/K)$.  
Let  $\rho_A:G_K\ra\GL_2(A)$  be a minimal $p$-ordinary deformation of  $\orho$  
for a $p$-profinite local $W$-algebra  $A$  with residue field  $\ov\fb$.
The representation  $\rho_A$  factors through  $G$  by minimality  (so, hereafter, we consider the deformation problem over  $G$).
By $p$-ordinarity,
we have
$$\rho_A|_{G_p}\cong\left(\begin{smallmatrix}\epsilon_A&*\\0&\delta_A\end{smallmatrix}\right)
\ \text{ with  $(\delta_A\ \modulo\mG_A)=\ov\psi^c$,}$$
where  $\mG_A$  is the maximal ideal of the local ring  $A$.
This gives rise to an exact sequence  $\epsilon_A\hra\rho_A\twra\delta_A$.
Realize  $\sG\lG_2(A)$  inside the $A$-linear endomorphism algebra  $\End_A(\rho_A)$,
and write  $F_+(\rho_A)$  the subspace of  $\{T\in\sG\lG_2(A)|T(\epsilon)=0\}=\Hom_A(\delta_A,\epsilon_A)$  
on which  $Ad(\rho_A)$  acts by the character  $\epsilon_A/\delta_A$
(the upper nilpotent Lie subalgebra if  $\rho_A|_{G_p}$  has upper triangular form as above).
Write  $Ad(\rho_A)^*$  for the Galois module  $Ad(\rho_A)\otimes_AA^\vee$  for the Pontryagin dual  $A^\vee$  of  $A$,
where  $G_\Q$  acts on the factor  $Ad(\rho_A)$.
Similarly we put  $F_+(\rho_A)^*:=F_+(\rho_A)\otimes_AA^\vee$  which is a $p$-local Galois module.
Then we define
\begin{equation}\label{sel}
\Sel_\Q(Ad(\rho_A)):=\Ker(H^1(G,Ad(\rho_A)^*)\ra H^1(I_l,\frac{Ad(\rho_A)^*}{F_+(\rho_A)^*})
\times\prod_{l|N}H^1(I_l,Ad(\rho_A)^*))
\end{equation}
for the product of restriction maps to the inertia group  $I_l\subset G$  of  $l$.
In the Galois group  $G$,
for  $l\nmid N$,  $I_l$  is trivial  (as  $K(\orho)^{(p)}/\Q$  is unramified outside  $Np$); so,
in the right-hand-side of the above definition,
$H^1(I_l,Ad(\rho_A)^*))$  for  $l\nmid N$  does not show up.
We write  $M^\vee$  for the Pontryagin dual of a module  $M$.

\medskip
Recall  $K^-_\infty/K$  which is the maximal sub-extension of  $K(\orho)^{(p)}$   $p$-abelian and anticyclotomic over  $K$,
where the word ``anti-cyclotomic"  means complex conjugation $c$  acts on  $\tau\in\Gal(K^-_\infty/K)$  by  $c\tau c^{-1}=\tau^{-1}$.
Lifting  $\tau\in\Gal(K_C/K)$  to   $h\in H$  and restricting  $h$  to  $K^-_\infty$,
we have an isomorphism  $\Gamma_\pG=\Gal(K(\pG)/K)\cong\Gal(K^-_\infty/K)$  (see \cite[pg.~636]{hida2015} and the main text \S3).
Recall:
\begin{defn}\label{Yc}{\em 
Let  $\phi:G_K\ra W^\times$  be a character of order prime to  $p$
whose image generates  $\Zp[\phi]$  in $W$  over  $\Zp$.
Let  $Y_\infty^-$  
be the Galois group over $K^-_\infty(\phi)$ 
of the maximal $p$-abelian extension of  $K^-_\infty(\phi)$  unramified outside  $\pG$  
and totally split at  $\pG^*$.
Regarding  $\Gal(K(\phi)/K)$  as a subgroup of  $\Gal(K^-_\infty(\phi)/K)\cong\Gal(K(\phi)/K)\times\Gal(K^-_\infty/K)$,
define 
$Y_\infty^-(\phi):=Y_\infty^-\otimes_{\Zp[\Gal(K(\phi)/K)],\phi}\Zp(\phi).$
Here  $\Zp(\phi)$  is the  $\Zp[\phi]$-module free of rank  $1$  on which  $\Gal(K(\phi)/K)$  acts by  $\phi$.}
\end{defn}
\begin{cor}\label{Ycor}
We have canonical isomorphisms of  $\T$-modules
\begin{equation*}\begin{split}
&\Sel_\Q(Ad(\Ind_K^\Q\Phi))\cong(Y_\infty^-(\psi^-)\otimes_{\Zp[\psi^-]}W)^\vee,\\
&\Sel_\Q(Ad(\Ind_K^\Q\Phi))^\vee
\cong(\Theta)/(\Theta)^2\cong\T_-/\theta\T_-\cong\T^\text{\rm nCM}/(\Theta)\cong\T^\text{\rm nCM}_+/(\theta),\\
&\T^\text{\rm nCM}_+/(\theta)\wot_W\WG\stackrel{(*)}{\cong}\Lambda_\WG/(\lc^-_p(\psi^-)).
\end{split}
\end{equation*}
\end{cor}
\begin{proof}
By the decomposition  $Ad(\Ind_K^\Q\Phi)\cong\chi\oplus\Ind_K^\Q\psi^-$  for  $\chi=\left(\frac{K/\Q}{}\right)$
combined with the functoriality of Greenberg's Selmer group,
we have  $\Sel_\Q(Ad(\Ind_K^\Q\Phi))\cong\Sel_\Q(\chi)\oplus\Sel_\Q(\Ind_K^\Q\psi^-)$.
The first isomorphism is \cite[Thm.~5.33]{hida-HMF},
where we get   $\Sel_\Q(\Ind_K^\Q\psi^-)=Y_\infty^-(\psi^-)^\vee$.  
Note that $\Sel_\Q(\chi)$  vanishes by  $p\nmid h_K$.
The second follows from cyclicity over  $\Lambda$  
proven in \cite[Thm.~B]{hida-cyclicity} and \thmref{TWinvthm}.
The third identity ($*$) follows from the proof of the anticyclotomic main conjecture shown by Rubin and Mazur--Tilouine:  
$\cha_{\Lambda_\WG}(Y_\infty^-(\psi^-))=(\lc^-_p(\psi^-))$ (see \cite{rubin1991}, \cite{rubin1994}, \cite{tilouine1989}, \cite{MT})
combined with the first two identities.
\end{proof}

\subsection{Modular Cayley--Hamilton representations}\label{modrep}
We introduce representations 
with values in a generalized matrix algebra (GMA) as in \cite{BC2009}, \cite{chen2014}  and \cite{WE2018}.
We refer to \cite[\S5.9]{WWE1} for the notion of ordinarity over  $\Q$  for GMA representations (not treated in \cite{BC2009} and \cite{chen2014}).
Since we have two conjugacy classes of $p$-decomposition groups  $D_\pG$  and  $D_{\pG^*}$,
we modify the definition (see below) of ordinarity depending on each factor  $\pG$  and  $\pG^*$.     
We follow  \cite[\S1.3]{BC2009}  to define a GMA $A$-algebra  $E$. 
Let $A$  be a commutative ring and $E$ an A-algebra. 
We say
that $E$ is a generalized matrix algebra (GMA) of type $(d_1,\dots , d_r)$ if $R$ is equipped
with:
\begin{itemize}
  \item  a family orthogonal idempotents $\Ec=\{e_1,\dots , e_r\}$  with  $\sum_ie_i=1$,
\item for each $i$, an $A$-algebra isomorphism  $\psi_i : e_iEe_i \xra{\sim} M_{d_i}(A)$,
such that the trace map $T : R \ra A$, defined by $T(x) :=\sum_i\Tr(\psi_i(e_ixe_i))$
satisfies $T(xy) = T(yx)$ for all $x, y \in E$. 
We call $\Ec = \{e_i,  \psi_i, i = 1, \dots , r\}$ the
data of idempotents of $E$.
\end{itemize}
In this appendix,
we assume that  $r=2$  and  $d_1=d_2=1$;
so, we can forget about  $\psi_i$  as an $A$-algebra automorphism of  $A$  is unique.
Once we have  $\Ec$,
we identify  $e_iEe_i=A$  and put  $B=e_1Ee_2$  and  $C=e_2Ee_1$.
Then a generalized matrix algebra over 
$A$ is a pair of an associative $A$-algebra  $E$  and  $\Ec$.
It is isomorphic to  $A\oplus B\oplus C\oplus A$  as  $A$-modules; so,
we write instead  $(E,\Ec)=\left(\begin{smallmatrix}A&B\\
C&A\end{smallmatrix}\right)$
which we call a GMA structure. 
There is an A-linear map $B\otimes_A C \ra A$ such that the multiplication in  $E$ is given by 2-by-2
matrix product. In this case, $A$ is called the scalar subring of $(E,\Ec)$ and $(E,\Ec)$ is called an $A$-GMA.
A Cayley--Hamilton representation with coefficients in $A$ and residual representation  
$\left(\begin{smallmatrix}\ov\psi&0\\0&\ov\psi^c\end{smallmatrix}\right)$  
(with this order  $\ov\psi$  at the top)
is a homomorphism
$\rho : H \ra E^\times$, such that $(E,\Ec)$ is an $A$-GMA, and such that in matrix coordinates, $\rho$ is given by
$
\s\mapsto
\left(\begin{smallmatrix}\rho^\Ec_{11}(\s)&\rho^\Ec_{12}(\s)\\
\rho^\Ec_{21}(\s)&\rho^\Ec_{22}(\s)\end{smallmatrix}\right)
$
with  $(\rho_{11}(\s)\ \modulo\mG_A) = \ov\psi(\s)$, $(\rho_{22}(\s)\ \modulo\mG_A) = \ov\psi^c(\s)$, 
and  $\rho_{12}(\s)\rho_{21}(\s) \equiv 0\mod\mG_A$. 
For a given  $\rho$,
if we change the set $\Ec$  of idempotents,
the matrix expression changes; so, we added the superscript  $\Ec$  to the matrix entries  $\rho_{ij}^\Ec$
to indicate its dependence on  $\Ec$.
If the input of  $\Ec$  is clear from the context,
we omit the superscript  $\Ec$.
%Given such a $\rho$, there is an induced $A$-valued pseudo character, denoted  $T_\rho : H \ra A$, given
%by $T_\rho = \rho_{11} + \rho_{22}$ and $\det(\rho) = \rho_{11}\rho_{22} - \rho_{12}\rho_{21}$, cf. \cite[Prop. 2.2.3]{WE15}.

In  $H$,
we have two conjugacy classes of the $p$-decomposition groups depending on prime factors of  $p$  in  $K$.
Fix a decomposition subgroup  $D_\pG\subset H$  for  $\pG$  and put  $D_{\pG^*}$  for  $\pG^*$.
We define $\pG$-ordinarity (resp.  $\pG^*$-ordinarity) 
of  $\rho$  to have  $\Ec$  (resp.  $\Ec^*$)  such that  $\rho^\Ec_{12}(\s)=0$  for all  $\s\in D_\pG$  
and  $\rho^\Ec_{22}(I_\pG)=1$  (resp. $\rho^{\Ec^*}_{21}(\s)=0$  for all  $\s\in D_{\pG^*}$  and  $\rho^{\Ec^*}_{11}(I_{\pG^*})=1$).
We say  $\rho$  is {\em ordinary} if it is $\pG$  and  $\pG^*$-ordinary at the same time.
This definition does not depends on the choice of  $D_\pG$  and  $D_{\pG^*}$.
For example, if we replace  $D_\pG$  by  $\s D_\pG\s^{-1}$,
$(E,\rho(\s)\Ec\rho(\s)^{-1})$  satisfies the required conditions.

If  $(E,\Ec)$  can be embedded into the matrix algebra  $M_2(\wt A)$  for a complete local $W$-algebra  $\wt A$ with residue field  $\fb$
containing  $A$,
the Cayley--Hamilton representation  $\rho:H\ra E^\times$  can be regarded as a representation into  $\GL_2(\wt A)$.
Since  $\orho=\Ind_K^\Q\ov\psi$  is irreducible over  $G$,
we may have an extension  $\wt\rho$  of the GMA representation  $\rho$  to  $G$.
If an extension  $\wt\rho$  exists, 
the extension is a usual representation into  $\GL_2(\wt A)$.
As usual, we call  $\wt\rho$  $p$-ordinary if $\wt\rho|_{G_p}\cong\left(\begin{smallmatrix}\epsilon&*\\0&\delta\end{smallmatrix}\right)$
with unramified $\delta\equiv\psi^c\mod\mG_{\wt A}$.  
The ordering of the residual representation  $\left(\begin{smallmatrix}\ov\psi&0\\0&\ov\psi^c\end{smallmatrix}\right)$  
(with this order  $\ov\psi$  at the top)  is fixed; so, plainly, to have compatibility of ordinarity of  $\rho$ over  $H$
and $\Q$-ordinarity of  $\wt\rho$ (and
to preserve residual order of the characters  $\ov\psi$  and  $\ov\psi^c$),
we need to define  $\pG^*$-ordinarity to have a set of idempotent  $\Ec^*$  so that  $\rho^{\Ec^*}|_{D_\pG^*}$
in the lower triangular form.
Indeed, if  $\wt\rho(c)=\left(\begin{smallmatrix}0&1\\ 1&0\end{smallmatrix}\right)$,
$\rho$  is  $\pG$-ordinary for  $\Ec$  if and only if
$\rho$  is  $\pG^*$-ordinary for the same  $\Ec$  by choosing  $D_{\pG^*}=cD_\pG c^{-1}$.
As we describe in the following proposition,
this phenomenon occurs if we take  $\rho:=\rho_\T|_H$  for  $A=\T_+$  and  $\wt A=\T$.
Details of the deformation theory of  $\orho$  in the category of representations over  $G$  and in the category of Cayley--Hamilton representations
over  $H$  will be discussed in a forthcoming paper \cite{hida2019-ord-real}.

\begin{prop}\label{+rep}
The Galois representation  $\rhob=\rho_{\T}|_H$  associated to  $\T$   
restricted to  $H$  is an ordinary Cayley-Hamilton representation with values in the following $\T_+$-GMA  
$$(\E,\Ec=\Ec^*)=\left(\begin{smallmatrix}\T_+&B_+\\C_+&\T_+\end{smallmatrix}\right)
\cong\left(\begin{smallmatrix}\T_+&\T_-\\\T_-&\T_+\end{smallmatrix}\right)\
\text{ with  $B_+\otimes_{\T_+}C_+\cong\T_-\otimes_{\T_+}\T_-\ra\T_+$}$$
given by  $\Theta b\otimes\Theta c\mapsto\theta bc$  for  $\theta=\Theta^2$ (the product in  $\T$).
\end{prop}
\begin{proof}
Recall  $\T_-:=\{x\in\T|\s(x)=-x\}$.
Then  $\T_-=\Theta\T_+$, and  $\Theta\in\T^\text{\rm nCM}$  under the inclusion $\T\hra\T^\text{\rm CM}\oplus\T^\text{\rm nCM}$; 
so, $\Theta$  is a zero-divisor in  $\T$  but is not a zero-divisor in $\T^\text{\rm nCM}$.
Similarly  $\theta\in\T_+^\text{\rm nCM}$.
Extend the character  $\ov\psi$  to a function on  $G$  just by  $0$  outside  $H$,
and decompose  $G=H\sqcup c H$.
Then we have the following standard realization of the induced representation:
$$\orho(\tau)=\left(\begin{smallmatrix}\ov\psi(\tau)&\ov\psi(\tau c)\\
\ov\psi(c^{-1}\tau)&\ov\psi(c^{-1}\tau c)\end{smallmatrix}\right).$$
Then if  $\chi(\tau)=-1$  ($\Leftrightarrow\tau\not\in H$),
we have  
$$(\orho\otimes\chi)(\tau)=\left(\begin{smallmatrix}0&-\ov\psi(\tau c)\\
-\ov\psi(c^{-1}\tau)&0\end{smallmatrix}\right)=j\orho(\tau)j^{-1}$$
for  $j:=\left(\begin{smallmatrix}1&0\\
0&-1\end{smallmatrix}\right)$.
If  $\chi(\tau)=1$  ($\Leftrightarrow\tau\not\in G_K$),
$\orho(\tau)$  is diagonal commuting with  $j$; so,
$$(\orho\otimes\chi)(\tau)=\orho(\tau)=j\orho(\tau)j^{-1}.$$
Thus we conclude  $\orho\otimes\chi=j\orho j^{-1}$.

The deformation functor represented by  $\T$  is given by:
$$D(A):=\{\rho:G\ra \GL_2(A):\text{$p$-ordinary }|\ (\rho\ \modulo\mG_A=\orho\}/\approx,$$
where  ``$\approx$"  is the strict equivalence 
(i.e., conjugation by  $1+M_2(\mG_A)$).
Thus we can let  $\chi$  act on  $D$  by  
$$\rho\mapsto j(\rho\otimes\chi)j^{-1}=\rho^\s.$$
Since  $j\left(\begin{smallmatrix}a&b\\c&d\end{smallmatrix}\right)j^{-1}=\left(\begin{smallmatrix}a&-b\\-c&d\end{smallmatrix}\right)$
and  $(\rho_\T|_H\modulo(\Theta))=\Phi\oplus\Phi^c$  is diagonal,
we have  $uj(\rho_\T\otimes\chi)(uj)^{-1}=\rho_\T^\s$  with  $u\in1+\Theta M_2(\T)$.
Write  $U=uj$.
Applying  $\s$,
we get  $U^\s(\rho_\T^\s\otimes\chi)U^{-\s}=\rho_\T$; so, we have
$$U\rho_\T U^{-1}=U(\rho_\T\otimes\chi)U^{-1}\otimes\chi=\rho_\T^\s\otimes\chi=U^{-\s}\rho_\T U^\s.$$
Thus  $ju^\s ju=U^\s U=z\in Z:=1+\Theta\T$.
Since  $1+\Theta M_2(\T)$  is  $p$-profinite, letting  $\s$  act on  $1+\Theta M_2(\T)$  by  $x\mapsto x^{\wt\s}:=jx^\s j$,
we can thus write  $u=v^{\wt\s-1}\in (1+\Theta M_2(\T))/Z$  for  $v\in 1+\Theta M_2(\Theta)$.
Thus replacing  $\rho_\T|_H$  by  $\rhob:=v^{-1}j\rho_\T jv|_H$,
we find  $j\rhob j^{-1}=\rhob^\s$.
In other words,
$\rhob$  has values in  $\E=\left(\begin{smallmatrix}\T_+&\T_-\\\T_-&\T_+\end{smallmatrix}\right)$, as desired

Since  $\psi^-|_{D_\pG}\ne1$ (H0),
we can choose first $\tau\in\Delta$  with  $\psi(\tau)\ne\psi^c(\tau)$  
so that  $\rho_\T(\tau)=\left(\begin{smallmatrix}\psi(\tau)&0\\0&\psi^c(\tau)\end{smallmatrix}\right)$, we can define
the set  $\Ec$  of idempotents of  $\Eb$  having the GMA form as above by
$$e_1=\frac{\rho_\T(\tau)-\psi^c(\tau)}{\psi(\tau)-\psi^c(\tau)}
\ \text{ and }\ e_2=\frac{\rho_\T(\tau)-\psi(\tau)}{\psi^c(\tau)-\psi(\tau)}.$$
Writing  $\Eb=\T_+\oplus B\oplus C\oplus \T_+$  with  $B\cong C\cong\T_-$,
we note that  $B$ (resp.  $C$) is the eigenspace 
under the conjugation action of  $\rho_\T(\tau)$  with eigenvalue  $\psi^-(\tau)$  (resp.  $\psi^-(\tau)^{-1}$).
Thus our expression of  $\rho_\T|_H$  is associated to  $(\Eb,e_1,e_2)$.
By ordinarity of  $\rho_\T$  on  $G_p$  (inducing  $D_\pG$),
we see  $\rho_\T|_H$  is  $\pG$-ordinary.  
Plainly  $c\in G$  interchanges  $e_1$  and  $e_2$; i.e.,
$\rho_\T(c)e_1\rho_\T(c)=e_2$.
Thus over  $D_{\pG^*}=cD_\pG c$,
we conclude  $\rho_\T|_H$  with values in  $(\Eb,\Ec)$  is also  $\pG^*$-ordinary. 
Since the residual representation is exactly  $\left(\begin{smallmatrix}\ov\psi&0\\0&\ov\psi^c\end{smallmatrix}\right)$  
(with this order  $\ov\psi$  at the top),
the choice of  $(e_2,e_1)$  is impossible violating the residual order of the characters
(the definition of  $\pG^*$-ordinarity is lower triangular on  $D_{\pG^*}$  to accommodate to preserve this residual order).
Therefore we need to choose  $\Ec=(e_1,e_2)$  for  $\pG^*$-ordinary.
%If we use  $\Ec^*=(e_2,e_1)$  in place of  $\Ec$,
%$\pG^*$-ordinarity is equivalent to have  $\rho_\T(D_{\pG^*})=\rho_\T(cD_\pG c)$  upper triangular
%and  $\rho_\T(I_{\pG^*})\subset\left(\begin{smallmatrix}1&*\\0&*\end{smallmatrix}\right)$, which violate the ordering  $\psi\oplus\psi^c$.  
\end{proof}
Under the normalization as above,
we may and do assume that  $\rho_\T(c)=\left(\begin{smallmatrix}0&1\\1&0\end{smallmatrix}\right)$.

\subsection{Local Iwasawa theory}\label{localIw}
Let  $k/\Qp$ (inside  $\oQ_p$)  is a Galois extension with  $p\nmid[k:\Qp]$.
Write  $F/k$  for the cyclotomic $\Zp$-extension inside  $\oQ_p$.
Let  $\Gamma:=\Gal(F/k)=\gamma^{\Zp}$  and put  $\Gamma_n=\Gamma^{p^n}$.
Set  $F_n:=F^{\Gamma_n}$  with  $p$-adic integer ring  $\oG_n$.
Let  $L$  (resp.  $L_n$)  be the maximal abelian  $p$-extension of  $F$  (resp.  $F_n$).
Write  $X_n:=\Gal(L_n/k_n)$  and  $X:=\Gal(L/F)$.
We have  $\Gal(F/\Qp)=\Gal(F/\Qp)\ltimes X$.
The exact sequence
$$1\ra X\ra\Gal(L/k)\ra\Gamma\ra1$$
is split just by lifting  $\gamma$  to an element  $\wt\gamma\in\Gal(L/k)$  taking splitting image  $\wt\gamma^{\Zp}$.
Therefore the commutator subgroup of  $\Gal(L/k_n)$  is given by  $(\gamma^{p^n}-1)X$,
and we have the corresponding exact sequence at each level  $n$:
$
1\ra X/(\gamma^{p^n}-1)X\ra\Gal(L_nF/F)\ra\Gamma_n\ra1.
$

Let  $k_\infty/k$  be the unramified $\Zp$-extension inside  $\oQ_p$  with its $n$-th layer  $k_n$,
and put  $\fc_n=Fk_n$.
Let  $\lc$  (resp.  $\lc_n$)  be the maximal abelian $p$-extension of  $\fc_\infty$  (resp.  $\fc_n$).
Set  $\Xc:=\Gal(\lc/\fc_\infty)$.
Pick a lift  $\phi\in\Gal(\lc/k)$  of the Frobenius element  $[p,\Qp]^f$  (for the residual degree  $f$  of  $k/\Qp$)  generating  $\Gal(k_\infty F/k)$  
and a lift  $\wt\gamma\in\Gal(\lc/k)$  of the generator $\gamma$  of  $\Gal(k\Q_{p,\infty}/k_0)=\Gamma$.
The commutator  $\tau:=[\phi,\wt\gamma]$, 
acts on  $\Xc$  by conjugation,
and  $(\tau-1)x:=[\tau,x]=\tau x\tau^{-1}x^{-1}$  for  $x\in\Xc$  is uniquely determined independent of the choice of  $\gamma$  and  $\phi$.
Define  $L'\subset\lc$  and  $L'_n\subset\lc_n$  by the fixed field of  $(\tau-1)\Xc$ (i.e., the fixed field of  $\tau$), 
which is independent of the choice of  $\wt\gamma$  and  $\phi$.
Let  $X'=\Gal(L'/\fc_\infty)$  and  $X'_n=\Gal(L'_n/\fc_n)$.   

\begin{prop}\label{Iwprop}
Let the notation and the assumptions be as above.
\begin{enumerate}
  \item 
We have a canonical decomposition 
$$X=\varprojlim_nX_n=\varprojlim_nX/(\gamma^{p^n}-1)X\cong\begin{cases}\Zp[[\Gal(F/\Qp)]]&\text{ if  $\mu_p(k)=\{1\}$,}\\
\Zp[[\Gal(F/\Qp)]]\oplus\Zp(1)&\text{ if  $\mu_p(k)=\mu_p(\oQ_p)$}\end{cases}$$  
as  $\Zp[[\Gal(F/\Qp)]]$-modules.
Thus for each finite dimensional $\Qp$-irreducible abelian representation  $\eta$  
of  $\Gal(k/\Qp)$  with values in  $\GL_{\dim(\eta)}(\Zp)$  of order prime to  $p$,
writing  $X[\eta]$  for the maximal $\eta$-isotypical quotient of  $X$, we have  
$$X[\eta]\cong\begin{cases}W(\kp)[[\Gamma]]&\text{ if $\eta\ne\omega$,}\\
\Zp[[\Gamma]]\oplus\Zp(1)&\text{ if  $\eta=\omega$}\end{cases}$$  
as  $\Gal(F/\Qp)$-modules.
Here  $\kp$  is the residue field of the subalgebra of  $M_{\dim(\eta)}(\Zp)$  generated 
by the values of  $\eta$  over  $\Zp$,  $\omega$  is the Teichm\"uller character 
and  $\s\in\Gal(F/\Qp)$  acts on  $W(\kp)$  via  $\eta$  regarded as having values in  $W(\kp)^\times$.
\item  
The restriction map  $X'\ra X$  induces an isomorphism of  $X'/(\phi-1)X'$  
onto the augmentation ideal of  $\Zp[[\Gal(F/\Q)]]\subset X$.
\item For the character  $\eta:\Gal(k/\Qp)\ra W(\kp)$  in (1),  
the factor  $X'[\eta]$  is a cyclic  $W(\kp)[[\Gamma\times\Upsilon]]$-module
(i.e., it is generated topologically over  $W(\kp)[[\Gamma\times\Upsilon]]$  by one element).
\end{enumerate}
\end{prop}
Note that the subalgebra of   $M_{\dim(\eta)}(\Zp)$  generated by the values of  $\eta$  over  $\Zp$
is isomorphic to the Witt vector ring  $W(\kp)$  with coefficients in its residue field  $\kp$.
\begin{proof}
We first prove the assertion (1).
The statement of \cite[Thm.~25]{iwasawa1973}
asserts  $X\cong\Zp[[\Gamma]]^{[k:\Qp]}$  or  $\Zp[[\Gamma]]^{[k:\Qp]}\oplus\Zp(1)$  as  $\Zp[[\Gamma]]$-modules.
Write  $Y$  be the maximal $\Zp[[\Gamma]]$-free quotient of  $X$.
Since  $\Gal(k/\Qp)$  has order prime to  $p$,  $\Gal(K/\Qp)\cong\Gal(k/\Qp)\ltimes\Gamma$, and
its action on  $Y$  is determined by its action on  $Y_0=Y/(\gamma-1)Y$.
We need to show  $Y_0\cong\Zp[\Gal(k/\Qp)]$  as  $\Gal(k/\Qp)$-modules
(which implies  $Y\cong Y_0[[\Gamma]]\cong \Zp[[\Gal(K/\Qp]]$).
Let  $\Q_{p,\infty}\subset\Qp[\mu_{p^\infty}]$  be the cyclotomic $\Zp$-extension.
Writing  $\wh M:=\varprojlim_nM/p^nM$  for a module  $M$,
by class field theory,  $\Gal(L_0K/K)$  fits into the following commutative diagram with exact rows and surjective vertical maps:
$$\begin{CD}
\Gal(L_0K/K)@>\hra>> \wh{k^\times}@>{N_{k/\Qp}}>>\wh{\Qp^\times}\\
@V{||}VV @V\text{Artin rec.}VV @V{a}VV\\
\Gal(L_0K/K)@>>\hra>\Gal(L_0/k)@>\twra>{\Res}>\Gal(\Q_{p,\infty}/\Qp),
\end{CD}$$
where the composite  $a\circ N_{k/\Qp}$  for the norm map  $N_{k/\Qp}$  has image  $\Gal(\Q_{p,\infty}/\Qp)\cong 1+p\Zp\cong\Gamma$.  

First suppose that  $\mu_p(k)=\{1\}$.
Then   $\wh{k^\times}$  is torsion-free.
The isomorphism class of a torsion-free  $\Zp[\Gal(k/\Qp)]$-module  $M$   of finite rank over  $\Zp$
is determined by the $\Qp[\Gal(k/\Qp)]$-module  $M\otimes_{\Zp}\Qp$.
Since  $\Qp[\Gal(k/\Qp)]$  is semi-simple, we conclude  $\wh{k^\times}\cong\Zp[\Gal(k/\Qp)]\oplus\Gamma$  
with  $\Gal(k/\Qp)$  acting on $\Gamma$ trivially.  
Thus we conclude
$Y_0\cong\Zp[\Gal(k/\Qp)]$  in which the $\eta$-isotypical component has rank  $\dim(\eta)=\rank_{\Zp}W(\kp)$  over  $\Zp$.

Now assume that  $\mu_p(k)$  is non-trivial.
Since  $p\nmid[k:\Qp]$,
$\mu_{p^\infty}(k)=\mu_p(k)$; so,
the torsion part of  $\wh{k^\times}$  is cyclic of order  $p$.
Let  $\wh{k^\times_f}$  be the maximal torsion-free quotient of  $\wh{k^\times}$.
Then by the same argument as in the case where  $\mu_p(k)=\{1\}$,
we find  $\wh{k^\times_f}\cong\Zp[\Gal(k/\Qp)]\oplus\Gamma$  as  $\Zp[\Gal(k/\Qp)]$-modules.
By Iwasawa's expression,  $X/(\gamma-1)X\cong\Zp^{[k:\Qp]}\oplus\mu_p(k)$
in which  $\mu_p(k)$  is identified with  $\Zp(1)/(\gamma-1)\Zp(1)$.
Again we have  $(X/(\gamma-1)X)/\mu_p(k)\cong\Zp[\Gal(k/\Qp)]$  as $\Zp[\Gal(k/\Qp)]$-modules.
We have a commutative diagram with exact row
$$\begin{CD}\Zp(1)/(\gamma-1)\Zp(1)@>\hra>>X/(\gamma-1)X@>\twra>>Y/(\gamma-1)Y\\
@V{\wr}VV @V{||}VV @VVV\\
\mu_p(k)@>>\hra> X/(\gamma-1)X@>>\twra>\Zp[\Gal(k/\Qp)]\end{CD}$$
of  $\Zp[\Gal(k/\Qp)]$-modules.
This shows  $Y_0=Y/(\gamma-1)Y\cong\Zp[\Gal(k/\Qp)]$  as  $\Zp[\Gal(k/\Qp)]$-modules, and hence
$Y\cong\Zp[[\Gal(F/\Qp)]]$.
Therefore the surjective $\Zp[[\Gal(F/\Qp)]]$-morphism  $X\twra Y$  splits,
and hence  $X\cong\Zp(1)\oplus\Zp[[\Gal(K/\Qp)]]$ as desired.

Now we prove (2).
Let  $k_\infty/k_n/k_0$  be the intermediate $n$-th layer of the unramified  $\Zp$-extension of  $k_0$  (so,  $\Gal(k_n/k_0)\cong\Z/p^n\Z$)).
%For  $k_{n,m}\subset F$  be the  $m$-th layer of the cyclotomic $\Zp$-extension of  $k_n$; so, $\Gal(k_{n,m}/k_0)\cong\Gamma/\Gamma^{p^m}\times\Upsilon/\Upsilon^{p^n}$.
Recall the integer ring  $\oG_n$  of  $k_n$.
Let  $\Xc_n=\Gal(\lc_n/\fc_n)$.
Then we have an exact sequence of  $\Zp[\Gal(k_n/\Q)]$-modules
$$\begin{CD}
\wh{\oG_n^\times}@>\hra>>\wh{k^\times_n}@>\twra>v>\Zp\\
@V{\wr}VV @V{\wr}VV @V{\wr}VV\\
\Xc_n@>>\hra>\Gal(\lc_n/k_n)@>>\twra>\Gal(k_\infty/k_n)
\end{CD}$$
where the map  $v$  is induced from the valuation  $\ord_p$  of  $k$  normalized so that  $\ord_p(p)=1$.
Writing  $\vpi$  for a prime element in  $\oG_n$,
we have  $v(\vpi)=e^{-1}$.
Then this exact sequence is split by  $v(p^{\Zp})=\Zp=e^{-1}\Zp=v(\vpi^{\Zp})$; so,
$\wh{k_n^\times}\cong \Xc_n\oplus\Zp$  as  $\Zp[\Gal(k_n/\Qp)]$-modules.
By this diagram and  $L'_n\supset k_\infty$,
we still have  $\Gal(L'_n/k_n)=X'_n\oplus\Gal(k_\infty/k_n)$  with  $\Gal(k_\infty/k_n)\cong\Zp$.

By the same argument as in the case proving (1), if  $\mu_p(k)=\mu_p(\oQ_p)$,
we have  
$\Xc_n\cong \Yc_n\oplus\Zp(1)$  as  $\Zp[[\Gal(k_n\Q_{p,\infty}/\Qp)]]$--modules for a unique direct summand  $\Yc_n$.
On  $\Zp(1)$,  $\phi$  acts trivially  (as $\nu_p([p,\Qp])=1$  for the $p$-adic cyclotomic character  $\nu_p$); so,
$[\wt\gamma,\phi]$  acts trivially on the factor  $\Zp(1)$. 
Hence we still have the decomposition  $X'_n=Y'_n\oplus\Zp(1)$.
The restriction from  $X'_m\ra X'_n$  for  $m>n$  induces on  $\Zp(1)$  multiplication by  $p^{m-n}$  as  $\phi=[p,\Qp]^f$
acts trivially on  $\mu_{p^\infty}(\oQ_p)$.
Thus passing to the limit,
the factor  $\Zp(1)$  disappears.
Therefore, by Kummer theory,  $\Coker(X'\xra{\Res} X)=\Zp\oplus\Zp(1)$  if  $\mu_p(k)=\mu_p(\oQ_p)$  and otherwise  $\Zp$; so,
by definition, 
the restriction map  $Y'_m\ra Y'_n$  is onto, and 
its image passing to the limit is the augmentation ideal of  $\Zp[[\Gal(F/\Qp)]]$
(as we lose the augmentation quotient  $\Zp$  which corresponds to the factor  $\Zp$  in  $\Gal(L'_n/k_n)$).
Since  $\Ker(X'\ra X)$  is plainly  $(\phi-1)X'$,
we find that  $X'/(\phi-1)X'$  is isomorphic to the augmentation ideal of  $\Zp[[\Gal(F/\Qp)]]$  by (1).

The same argument works well when  $\mu_p(k)=\{1\}$.
In this case, the argument is easier as the factor  $\Zp(1)$  does not show up.

We prove (3).
Note that $\Zp[[\Gal(F/\Qp)]]=\bigoplus_\chi W(\kp_\chi)[[\Gamma]]$
for  $\chi$  running over all characters of  $\Gal(k/\Q)$,
where  $\kp_\chi$  is the finite field generated by the values of  $\chi\ \modulo p$
over  $\fb_p$.
Then its augmentation ideal is given by  $(\gamma-1)\Zp[[\Gamma]]\oplus\bigoplus_{\chi\ne1} W(\kp_\chi)[[\Gamma]]$.
Thus  $X'[\eta]/(\phi-1)X'[\eta]\cong W(\kp_\eta)[[\Gamma]]$  as  $W(\kp_\eta)[[\Gamma]]$-modules
by \propref{Iwprop} (2).
This is clear if  $\eta$  is non-trivial.
If  $\eta=1$,
we note that  $(\gamma-1)\Zp[[\Gamma]]\cong\Zp[[\Gamma]]$  as  $\Zp[[\Gamma]]$-modules.
So $X'[\eta]/(\phi-1)X'[\eta]$  is cyclic over  $W(\kp)[[\Gamma]]$.
By the Nakayama's lemma, we get the desired cyclicity of  $X'[\eta]$  over  $W(\kp)[[\Gamma\times\Upsilon]]$.
\end{proof}

\subsection{Proof of Theorem \ref{thm: app main} and Corollary \ref{cor: app1}}\label{pfBC}
Recall the  $\T_+$-GMA
$\E=\left(\begin{smallmatrix}\T_+&\T_-\\\T_-&\T_+\end{smallmatrix}\right)$  given in \propref{+rep}.
Set  $\E^\text{\rm nCM}=\E\otimes_{\T_+}\T^\text{\rm nCM}_+$  and  $\E^\text{\rm CM}=\E\otimes_{\T_+}\T^\text{\rm CM}_+$,  
and write  $\rhob:W[H]\ra\E$,  $\rhob^\text{\rm nCM}:W[H]\ra\E^\text{\rm nCM}$  
and    $\rhob^\text{\rm CM}:W[H]\ra\E^\text{\rm CM}$  
for the associated Cayley--Hamilton representations.
Pick a prime  $\wp$  of  $K(\rhob)$  above  $\pG$. 
Let  $\ov I_\pG$  (resp.  $\ov I_{\pG^*}$, $\ov D_\pG$)  
be the  $\pG$-inertia (resp.  $\pG^*$-inertia, $\pG$-decomposition) subgroup of  $\Gal(K(\rhob)/K(\orho))$
corresponding to  $\wp$  and  $\wp^c$.
Regard  $[p,\Qp]^f\in D_\pG$  for the residual degree  $f$  of  $\PG=\wp\cap K(\orho)$, 
and recall  $\vp':=\rhob([p,\Qp]^f)=\left(\begin{smallmatrix}u^{-f}&*\\0&u^f\end{smallmatrix}\right)$  with  $u^f\in\T_+$.  
Put $\Lambda_0:=\Zp[[T]]\subset\Lambda_1:=W_1[[T,a]]\subset\T$  for  $a=u^{2f}-1$, and recall  $t=1+T$.
We restate Theorem~A.1.2 in the introduction in the following way:
\begin{thm}\label{inertiathm}
Let the notation be as above.
Suppose  {\em(H0--4)}.
Then we can choose conjugacy classes of  $\ov I_\pG$  and  $\ov I_{\pG^*}$  
in  $G$  and a generator  $\Theta$  of the $\s$-different  $\Ib=\T(\s-1)\T$  with  $\Theta^\s=-\Theta$  so that we have  
$$
\rhob(\ov I_\pG)=\left\{\left(\begin{smallmatrix}a&b\\0&1\end{smallmatrix}\right)\big|a\in t^{\Zp}, 
b\in\Theta\Lambda_1\right\}\subset \E^\times
$$
and
%\text{ and }\ 
$$
\rhob^\text{\rm nCM}(\ov I_\pG)=\left\{\left(\begin{smallmatrix}a&b\\0&1\end{smallmatrix}\right)\big|a\in t^{\Zp}, b\in\Theta\Lambda_1\right\}\subset \E^{\text{\rm nCM},\times}
$$
and  $\rho(\ov I_{\pG^*})= J\rho(\ov I_\pG)J^{-1}$,
where  $J=\left(\begin{smallmatrix}0&1\\1&0\end{smallmatrix}\right)$  for  $\rho=\rho_\T$
and  $\rho_{\T^\text{\rm nCM}}$.
Here  $t^{\Zp}\subset\Lambda$  is embedded in  $\E$  and  $\E^\text{\rm nCM}$  by the structure homomorphism.
\end{thm}

\noindent
We can conjugate  $\rhob$  by  $\left(\begin{smallmatrix}a&0\\0&1\end{smallmatrix}\right)$  for any  $a\in\T^\times$,
and by doing this,  $\Theta$  will be replaced by  $a\Theta$; so, actually, 
we can always assume that for any choice of the generator  $\Theta$  with  $\Theta^\s=-\Theta$
of the ideal  $(\Theta)$,
we can arrange  $\rhob(\ov I_\pG)$  (and  $\rhob(\ov I_{\pG^*})$)  as in the corollary.
\begin{proof}
Write simply  $I=\rhob(\ov I_\pG)$  and  $D=\rhob(\ov D_\pG)$.
From the definition of  $\Lambda$-algebra structure of  $\T$  
and $p$-ordinarity (e.g., \cite[(Gal), pg.~604]{hida2015}),
we know  
$I\subset M(\T)\cap \E\ \text{ and }\ \rhob(\ov I_{\pG^*})\subset JM(\T)J^{-1}\cap \E$
for the mirabolic subgroup  $M(\T):=\left\{\left(\begin{smallmatrix}a&b\\0&1\end{smallmatrix}\right)\big|a\in\T^\times, b\in\T\right\}$.
Since  $\Gal(\Qp^{ab}/\Qp)=[p,\Qp]^{\wh\Z}\ltimes\Zp^\times$  for the maximal abelian extension  $\Q^{ab}/\Q$  and the local Artin symbol  $[p,\Qp]$,
we find  
$$I\subset\left\{\left(\begin{smallmatrix}a&b\\0&1\end{smallmatrix}\right)\big|a\in t^{\Zp}, b\in\Theta\T_+\right\}\
\text{ and }\ D={\vp'}^{\Zp}\ltimes I$$
by the shape of  $\E$, and  $\det(\rhob(\ov I_\pG))=\Tc:=t^{\Zp}\subset\Lambda_0^\times$.
Thus we have an extension  $1\ra\Uc\ra I\ra\Tc\ra1$  with  $\Uc=\Ker(\det(\rhob):I\ra\Lambda^\times)$.

By \cite[Lem.~1.4]{hida2015},
this extension is split by the action of  $\Delta$    
for  $\Uc$  being an eigenspace on which  $\Delta$  acts by  $\psi^-$; 
so, we may assume to have a section  $s:\Tc\hra I$  identifying  $\Tc$  with
$\left\{\left(\begin{smallmatrix}a&0\\0&1\end{smallmatrix}\right)\big|a\in\Tc\right\}$.
Replacing  $\vp'$  by an element  $\vp\in\vp'\Uc$,
we may assume that  $\vp=\left(\begin{smallmatrix}u^{-f}&0\\0&u^f\end{smallmatrix}\right)$
commuting with  $\left(\begin{smallmatrix}t^{\Zp}&0\\0&1\end{smallmatrix}\right)=\Gal(\Q_{p,\infty}K(\orho)/K(\orho))$.
Take  $\phi\in D_\pG$  such that  $\rhob(\phi)=\vp$  
and  $\wt\gamma\in D_\pG$  with  $\rhob(\wt\gamma)=\left(\begin{smallmatrix}t&0\\0&1\end{smallmatrix}\right)$.
For the commutator  $[\phi,\wt\gamma]$,
we have  $\rhob([\phi,\wt\gamma])=1$ (i.e. it acts on  $K(\rhob)_\PG$  trivially; the requirement for
the validity of \propref{Iwprop} (3)).
The module  $\Uc$  is a $\Lambda_1$-module by the adjoint action of  $\Tc\cdot\vp^{\Zp}$.
Since  $\rhob^\text{\rm CM}|_I$  has kernel  $\Uc$,
we see that  $I=\rhob(\ov I_\pG)\cong\rhob^\text{\rm nCM}(\ov I_\pG)$; so, we only need to prove the assertion for  $\rhob$.
If  $\T_+^\text{\rm nCM}\cdot\Uc\subsetneq\T_-=\Theta\T_+=\Theta\T_+^\text{\rm nCM}$,
we have  $\Uc\T_+^\text{\rm nCM}\subset\Theta\mG^\text{\rm nCM}_+\T_+^\text{\rm nCM}=\mG_+^\text{\rm nCM}\T_-$  
for the maximal ideal  $\mG_+^\text{\rm nCM}$  of  $\T_+^\text{\rm nCM}$.

Write  $\PG|\pG$  for the prime factor in  $K(\psi^-)$  corresponding to $I_\pG$.
We apply \propref{Iwprop} to the  $\PG$-adic completion  $k$  of  $K(\psi^-)$, its cyclotomic $\Zp$-extension  $F$
and the composite  $\fc_\infty$  of  $F$  and the unramified $\Zp$-extension of  $k$.
Thus  $\Uc$  is made of unipotent matrices, and writing  
$$I_1:=\{\tau\in\ov I_\pG:\tau|_F=1\}=\{\tau\in\ov I_\pG:\tau|_{\fc_\infty}=1\},$$ 
we have  $\Uc=\rhob(I_1)$.
Therefore we may write 
$\rhob(\tau)=\left(\begin{smallmatrix}1&u(\tau)\\0&1\end{smallmatrix}\right)$  for  $\tau\in I_1$.
Let  $\ov u:=u\mod\mG_+^\text{\rm nCM}\T_-$  with values in  $\T_-/\mG_+^\text{\rm nCM}\T_-\cong\fb$.
Let  $H(\Phi^-):=\Ker(\Phi^-:H\ra \Lambda^\times)$  for the universal character  $\Phi$.
Since  $\T_-/\theta\T_-=Y_\infty^-(\psi^-)\otimes_{\Zp[\psi^-]}W$  by \corref{Ycor}
and  $\theta\T_+$  is the ideal of reducibility in  $\T_+$  of  $\rhob$  
in the sense of \cite[\S1.5]{BC2009},
this homomorphism extends to a non-zero homomorphism  $\ov u:H(\Phi^-)\ra\fb$  with  $\ov u(\tau h\tau^{-1})=\Phi^-(\tau)\ov u(h)$  unramified outside  $\pG$
over  $K(\Phi^-)=K(\orho)K^-_\infty$.
Since  $H(\Phi^-):=\Gal(K(\orho)^{(p)}/K(\Phi^-))$  only ramifies at  $p$,  $u$  is unramified at  ${\frc'}{\frc'}^c$.
Since  $\ov I_{\pG^*}$  is lower triangular contained in  $JM(\T)J^{-1}$,
$\ov u$  is unramified everywhere.
Let  $N_\infty\subset K(\orho)^{(p)}$  be the fixed field by  $\Ker(u:\Gal(K(\orho)^{(p)}/K(\Phi^-))\ra\T_-/\mG_+^\text{\rm nCM}\T_-)$  
and put  $X:=\Gal(N_\infty/K(\Phi^-))$.
Then  $N_\infty/K(\Phi^-)$  is an everywhere unramified $p$-abelian extension.
Since  $K(\Phi^-)/K(\psi^-)$  is a fully $p$-ramified $\Zp$-extension generated by an element  $\gamma$,
we find  $X/(\gamma-1)X$  is a Galois group of an everywhere unramified $p$-abelian extension of  $K(\psi^-)$,
which is non-trivial by our assumption.
Since  $p\nmid h_{K(\psi^-)}$,
this is a contradiction.
Thus the  $\T_+$-span of  $\ov u(I_1)$  is  $\fb$; so,
the  $\T_+$-span of  $u(I_1)$  is equal to  $\T_-$  by Nakayama's lemma.
Thus  $\T_+u(I_1)\not\equiv0\mod\mG_+^\text{\rm nCM}\T_-$; so,
we may assume that  $\Theta\in u(I_1)$.

Regard  $\psi^-$  as an abelian irreducible $\Zp$-representation acting on  $W$  regarded as a $\Zp$-module.
By \propref{Iwprop} (3), under the notation there,
the Galois group  $X'[\psi^-]$  is cyclic over  $W_1[[\Gamma\times\Upsilon]]$  ($\Gamma=t^{\Zp}$) 
and surjects onto  $\Uc$.
Since the action of  $W_1[[\Gamma\times\Upsilon]]$  factors through  $\Lambda_1$,
by \propref{Iwprop} (3),  $\Uc$  is cyclic over  $\Lambda_1$; so,
we have  $\Uc\cong\Lambda_1$.
Thus we conclude
$\rhob(I_1)=\Uc=\left\{\left(\begin{smallmatrix}1&a\\0&1\end{smallmatrix}\right)\big|a\in\Theta\Lambda_1\right\}$  
inside  $\rhob(H)$
(for a suitable choice of  $\Theta$).
This shows the desired expression for  $\rhob(\ov I_\pG)$.
By the same argument applied to  $\pG^*$,
we have  $\rhob(H)$  contains  $J\Uc J^{-1}$, $\Tc$  and  $J\Tc J^{-1}$,
and we obtain the form of  $\rhob(\ov I_{\pG^*})$.
\end{proof}

\medskip\noindent
{\bf Proof of Corollary~\ref{cor: app1}.} 
We have by \thmref{inertiathm},
$$\text{$\rho_P|_{I_p}$  is indecomposable $\Leftrightarrow(\Uc\ \modulo P)\ne1\Leftrightarrow P\nmid(\Theta)$.}$$
By \corref{Ycor}, $\T^\text{\rm nCM}/(\Theta)\wot_W \WG\cong\Lambda_\WG/(\lc^-_p(\psi^-))$,
we conclude that  $P\nmid(\Theta)\Leftrightarrow P\nmid(\lc^-_p(\psi^-))$.
By $\Coker(\T\wot_W \WG\ra\T^\text{\rm CM}\wot_W \WG\times\T^\text{\rm nCM}\wot_W \WG)\cong\Lambda_\WG/(\lc^-_p(\psi^-))$,
we see 
$$P\nmid(\lc^-_p(\psi^-))\Leftrightarrow P\not\in\Spec(\T^\text{\rm nCM})\cap\Spec(\T^\text{\rm CM})$$
as desired.\qed

%\end{document}

%\end{appendix}

\markleft{FRANCESC CASTELLA AND CARL WANG-ERICKSON}
\markright{CLASS GROUPS AND LOCAL INDECOMPOSABILITY FOR NON-CM FORMS}

% ----------------------------------------------------------------
\bibliographystyle{alpha}
\bibliography{CWEbib-2018-Gb}
% ----------------------------------------------------------------
\end{document}